\documentclass[twoside, reqno]{amsart}
\usepackage[english]{babel}
\usepackage[utf8]{inputenc}
\usepackage[
backend=biber,
style=alphabetic, 
giveninits=true 
]{biblatex}

    \usepackage{amsthm}
    \usepackage{thmtools}
     \usepackage{amsfonts}
    \usepackage{mathtools}
    \usepackage{amssymb}
    \usepackage{amsmath}
    \usepackage{tikz-cd}
    \usepackage[utf8]{inputenc}
    \usepackage{csquotes}
    \usepackage[headheight=14pt]{geometry} 
    \usepackage{appendix}
    \usepackage{url}
    \usepackage{hyperref}
    \usepackage{graphicx}
    \usepackage{xcolor}
    \usepackage{tikz}
    \usepackage{nameref}
    \usepackage{fancyhdr} 
    \usepackage{fancybox}
    \usepackage{enumitem} 
    \usepackage{trimclip}
    \usepackage{epigraph}
    \usepackage{faktor}
    \usepackage{ stmaryrd }
    \usepackage{comment}
    \usepackage{scalerel,stackengine}
    \usepackage{lscape}
    \usepackage{multirow}
    \usepackage{mathrsfs}

\theoremstyle{plain}
\newtheorem{thm}{Theorem}[section]
\newtheorem{lemma}[thm]{Lemma}
\newtheorem{proposition}[thm]{Proposition}
\newtheorem{cor}[thm]{Corollary}
\newtheorem{claim}[thm]{Claim}

\newtheorem{thm solo}{Theorem}
\newtheorem{prop solo}[thm solo]{Proposition}
\newtheorem{cor solo}[thm solo]{Corollary}
\newtheorem{conj solo}[thm solo]{Conjecture}

\theoremstyle{definition}
\newtheorem{definition}[thm]{Definition}
\newtheorem{example}[thm]{Example}

\theoremstyle{remark}
\newtheorem{remark}[thm]{Remark}

\numberwithin{equation}{section}

\DeclareMathOperator{\ga}{\mathfrak{g}}
\DeclareMathOperator{\ha}{\mathfrak{h}}

\DeclareMathOperator{\ma}{\mathfrak{m}}

\DeclareMathOperator{\Xs}{\mathfrak{X}}

\DeclareMathOperator{\tworlim}{2-\varinjlim}
\DeclareMathOperator{\Rlim}{R\varprojlim}

\DeclareMathOperator{\id}{id}
\DeclareMathOperator{\Hom}{Hom}

\DeclareMathOperator{\NS}{NS}

\DeclareMathOperator{\HTlog}{HTlog}

\DeclareMathOperator{\Ker}{Ker}
\DeclareMathOperator{\Eq}{Eq}

\DeclareMathOperator{\Coker}{Coker}

\DeclareMathOperator{\Ab}{Ab}

\DeclareMathOperator{\Sh}{Sh}
\DeclareMathOperator{\Sch}{Sch}

\DeclareMathOperator{\Spec}{Spec}
\DeclareMathOperator{\Spf}{Spf}
\DeclareMathOperator{\Spa}{Spa}
\DeclareMathOperator{\Spd}{Spd}

\DeclareMathOperator{\GL}{GL}

\DeclareMathOperator{\Pic}{Pic}

\DeclareMathOperator{\Lie}{Lie}
\DeclareMathOperator{\Image}{Im}
\DeclareMathOperator{\End}{End}
\DeclareMathOperator{\modulo}{mod}

\DeclareMathOperator{\Gal}{Gal}

\DeclareMathOperator{\ev}{ev}

\DeclareMathOperator{\Nr}{Nr}

\DeclareMathOperator{\rank}{rank}

\DeclareMathOperator{\cts}{cts}

\DeclareMathOperator{\an}{an}

\DeclareMathOperator{\et}{\acute{e}t}
\DeclareMathOperator{\Et}{\acute{E}t}

\DeclareMathOperator{\profet}{prof\acute{e}t}

\DeclareMathOperator{\Zar}{Zar}

\DeclareMathOperator{\rig}{rig}
\DeclareMathOperator{\alg}{alg}

\DeclareMathOperator{\RGamma}{R\Gamma}

\DeclareMathOperator{\FF}{FF}
\DeclareMathOperator{\Perf}{Perf}
\DeclareMathOperator{\LSD}{LSD}
\DeclareMathOperator{\Smd}{Smd}

\DeclareMathOperator{\Sm}{Sm}
\DeclareMathOperator{\qcqs}{qcqs}

\DeclareMathOperator{\ttt}{tt}
\DeclareMathOperator{\tor}{tor}

\DeclareMathOperator{\Exp}{Exp}

\DeclareMathOperator{\BdR+}{\mathrm{B_{dR}^+}}
\newcommand{\cj}[1]{\overline{#1}}
\DeclareMathOperator{\Adic}{Adic}

\stackMath
\newcommand\widecheck[1]{%
\savestack{\tmpbox}{\stretchto{%
  \scaleto{%
    \scalerel*[\widthof{\ensuremath{#1}}]{\kern-.6pt\bigwedge\kern-.6pt}%
    {\rule[-\textheight/2]{1ex}{\textheight}}
  }{\textheight}%
}{0.5ex}}%
\stackon[1pt]{#1}{\scalebox{-1}{\tmpbox}}%
}
\newcommand{\G}{\mathbb{G}}

\newcommand{\B}{\mathbb{B}}
\newcommand{\C}{\mathbb{C}}

\newcommand{\Pro}{\mathbb{P}}

\newcommand{\Z}{\mathbb{Z}}
\newcommand{\Q}{\mathbb{Q}}
\newcommand{\F}{\mathbb{F}}

\newcommand{\X}{\mathbb{X}}

\newcommand{\Mb}{\mathbb{M}}

\newcommand{\As}{\mathcal{A}}

\newcommand{\Os}{\mathcal{O}}
\newcommand{\Fs}{\mathcal{F}}
\newcommand{\Gs}{\mathcal{G}}
\newcommand{\Hs}{\mathcal{H}}
\newcommand{\Ls}{\mathcal{L}}

\newcommand{\womega}{\widetilde{\Omega}}

\DeclareMathOperator{\PPic}{\mathbf{Pic}}

\DeclareMathOperator{\BBun}{\mathbf{Bun}}
\DeclareMathOperator{\Fil}{Fil}
\DeclareMathOperator{\FFil}{\mathbf{Fil}}

\newcommand{\sub}{\subseteq}
\newcommand{\fa }{\forall}

\newcommand{\restr}[1]{|_{#1}}

\DeclareFieldFormat[article,inbook,incollection,inproceedings,patent,thesis,unpublished, misc]{title}{#1\isdot}

\DeclareFieldFormat{pages}{#1}

\renewbibmacro*{volume+number+eid}{%
  \printfield{volume}
  \setunit*{\addnbspace}
  \printfield{number}
  \setunit{\addcomma\space}
  \printfield{eid}
}
\DeclareFieldFormat[article]{number}{\mkbibparens{#1}}

\renewbibmacro*{issue+date}{}
\renewbibmacro*{note+pages}{
  \printfield{note}
  \setunit{\bibpagespunct}
  \printfield{pages}
  \setunit{\addcomma\addspace}
  \usebibmacro{date}
  \newunit}

\NewBibliographyString{toappearin}
\NewBibliographyString{submittedto}
\DefineBibliographyStrings{english}{%
  toappearin  = {to appear in},
  submittedto = {submitted to},
}

\renewbibmacro{in:}{
\ifentrytype{article}
    {\addperiod
    \addspace
    \printfield{pubstate}
     \clearfield{pubstate}
  \printunit{\addspace}}
{\bibstring{in}\printunit{\addspace}}
}

 \title{A Hodge--Tate decomposition with rigid analytic coefficients}
 \author{Lucas Gerth}
 \date{} 

\newcommand{\RNum}[1]{\uppercase\expandafter{\romannumeral #1\relax}}

\begin{document}

\maketitle
\setlength{\parskip}{0pt}
\setlength{\parindent}{15pt}
\setlist{topsep=5pt, leftmargin=*}

\begin{abstract}
Let $X$ be a smooth proper rigid analytic space over a complete algebraically closed field extension $K$ of $\mathbb{Q}_p$. We establish a Hodge--Tate decomposition for $X$ with $G$-coefficients, where $G$ is any commutative locally $p$-divisible rigid group. This generalizes the Hodge--Tate decomposition of Faltings and Scholze, which is the case $G=\mathbb{G}_a$. For this, we introduce geometric analogs of the Hodge--Tate spectral sequence with general locally $p$-divisible coefficients. We prove that these spectral sequences degenerate at $E_2$. Our results apply more generally to a class of smooth families of commutative adic groups over $X$ and in the relative setting of smooth proper morphisms $X\rightarrow S$ of seminormal rigid spaces. We deduce applications to analytic Brauer groups and the geometric $p$-adic Simpson correspondence.
\end{abstract}


\section{Introduction}
\subsection{The Hodge--Tate decomposition}
Let $X$ be a proper smooth rigid space over a complete algebraically closed field extension $K$ of $\Q_p$. The Hodge--Tate decomposition, first conjectured by Tate in his celebrated article \cite[§4.1]{Tat67}, asserts the existence of an isomorphism
\begin{align}\label{classical Hodge--Tate decomposition: intro}
    H_{\et}^n(X,\Q_p) \otimes_{\Q_p} K = \bigoplus_{i+j=n} H^i(X,\Omega_{X}^j)(-j),
\end{align}
that is Galois-equivariant when $X$ admits a model over a discretely valued subfield of $K$. The algebraic case is due to Faltings \cite[III Thm. 4.1]{Fal88} and was reproven via different methods by Tsuji \cite{Tsuji1999} and Nizio\l \,\,\cite{Niziol2008}. The existence of the decomposition in full generality was first proven in \cite[Thm. 13.3]{Bhatt2018}, building on \cite[§3.3]{Scholze2013}. In this article, we prove the following generalization of the Hodge--Tate decomposition.
\begin{thm}\label{Theorem 0}{(Theorem \ref{Theorem: The classical sseq splits})}
    Let $X$ be a proper smooth rigid space over $K$ and let $G$ be a locally $p$-divisible rigid group. Then a choice of a $\BdR+/\xi^2$-lift $\X$ of $X$ and the datum of an exponential $\Exp$ for $K$ induce a Hodge--Tate decomposition with $G$-coefficients, natural in $\X$ and $G$
    \begin{align}\label{Hodge--Tate decomposition for locally p-divisible groups} H_{v}^n(X,G) = H_{\et}^n(X,G) \oplus \bigoplus_{i+j=n, i<n} H^i(X,\Omega_{X}^{j})(-j)\otimes_{K} \Lie(G).\end{align}
\end{thm}
We recover (\ref{classical Hodge--Tate decomposition: intro}) as the case $G=\G_a$ using the Primitive Comparison Theorem \cite[Thm. 5.1]{scholze2013padicHodge}. By an exponential, we mean a continuous homomorphism $\Exp \colon K \rightarrow 1+\ma_K$ splitting the logarithm map $\log\colon 1+\ma_K \rightarrow K$ and extending the exponential power series. The locally $p$-divisible rigid groups form a class of commutative rigid groups, recently introduced in \cite[§6]{heuer2023padic}, of which examples are $G=\G_a,\G_m$ and more generally any connected commutative algebraic group over $K$. This also includes non-algebraic groups, such as abeloid varieties and analytic $p$-divisible groups in the sense of Fargues \cite[§2]{Farg19}. As we explain below, we also prove a more general relative version of this result (Theorem \ref{theorem 3}) and give applications to rigid analytic Brauer groups (Example \ref{Example: brauer groups shouldn't be torsion}) and non-abelian $p$-adic Hodge theory (§\ref{subsection: application to non-abelian $p$-adic Hodge theory}). We also deduce a formula for the cohomology groups $H_{v}^n(X,\B^{\varphi=p})$, where $\B$ is the $v$-sheaf analog of the ring of analytic functions on the “punctured Fargues--Fontaine disc” $\Spa(A_{\inf},A_{\inf})\backslash V(p[p^{\flat}])$ (Theorem \ref{thm: HT decomposition for BC space}).

For the proof of Theorem \ref{Theorem 0}, generalizing an idea of Faltings and Scholze, we construct a spectral sequence, natural in $X$ and $G$
\begin{align}\label{classical Hodge--Tate spectral sequence for general G: intro}
    E_2^{ij} = \left\{\begin{aligned}
  &H_{\et}^i(X, \Omega_X^j)(-j)\otimes_{K} \Lie(G)  \quad &\text{if }j>0\\
  &H_{\et}^i(X, G)  \quad &\text{if }j=0
\end{aligned}\right\}
 \Longrightarrow H_v^{i+j}(X,G).
\end{align}
We call it the \textbf{Hodge--Tate spectral sequence with $G$-coefficients}. We then show the following.
\begin{thm}\label{Theorem 1}{(Corollary \ref{Cor: The classical sseq degenerates})}
    Let $X$ be a proper smooth rigid space over $K$ and let $G$ be a locally $p$-divisible rigid group. Then the Hodge--Tate spectral sequence with $G$-coefficients (\ref{classical Hodge--Tate spectral sequence for general G: intro}) degenerates at the $E_2$-page.
\end{thm}
 For $G=\G_a$, the sequence (\ref{classical Hodge--Tate spectral sequence for general G: intro}) is the Hodge--Tate spectral sequence of Scholze \cite[Thm. 3.20]{Scholze2013}. For $G=\G_m$, this sequence is already defined by Ben Heuer in \cite[§2.4]{heuer2021lineonpefd}, where it is called the “multiplicative Hodge--Tate spectral sequence”. This last sequence is also the subject of recent work of Ertl--Gilles--Nizio\l \,
\cite{ertl2024vpicardgroupsteinspaces} where it is studied for $X$ a smooth Stein space.

\subsection{The geometric Hodge--Tate spectral sequence}
The key idea for proving Theorem \ref{Theorem 1} is to upgrade the Hodge--Tate spectral sequence (\ref{classical Hodge--Tate spectral sequence for general G: intro}) to a geometric analog. For this, we generalize the approach of \cite
{heuer2023diamantine}. We allow $K$ to be an arbitrary perfectoid field extension of $\Q_p$ and we let $\pi \colon X \rightarrow \Spa(K)$ denote the structure map. We consider the category $\Perf_K$ of affinoid perfectoid spaces over $K$. We then define for a locally $p$-divisible rigid group $G$ the \textbf{diamantine higher direct images}
\[ \BBun_{G,\tau}^n \coloneqq R^n\pi_{\tau,*}G \colon \Perf_K \rightarrow \Ab, \]
where $\tau$ is either the étale or the $v$-topology. Explicitly, this is the $\tau$-sheafification of the presheaf
\[ Y \in \Perf_K \mapsto H_{\tau}^n(X\times_K Y,G).\]
When $G=\G_m$ and $n=1$, this recovers the diamantine Picard functors $\PPic_{X,\tau}$ of \cite{heuer2023diamantine}.

The étale Picard functor is often representable by a rigid group and is expected to be so in general. We refer the reader to the introduction of \cite{heuer2023diamantine} for an exposition of some known cases. This holds in particular when $\pi\colon X \rightarrow \Spa(K)$ is the analytification of a proper smooth algebraic variety $\pi^{\alg}\colon X^{\alg} \rightarrow\Spec(K)$, in which case $\PPic_{X,\et}$ is the analytification of the Picard variety of $X^{\alg}$. For the higher pushforwards, the situation is different: Let $\pi_{\Et}^{\alg}\colon \Sch_{X^{\alg},\et} \rightarrow \Sch_{K,\et}$ denote the map of big étale sites, then the sheaves
\[ R^n\pi_{\Et,*}^{\alg}\G_m\colon \Sch_K \rightarrow \Ab\]
are generally not representable by smooth group schemes for $n\geq 2$ (see Remark \ref{remark: algebraic higher pushforward are not representable}). Moreover, the analytification of this sheaf and the analytic higher direct image are generally not isomorphic. This is exemplified in rank $n=2$ where the analytic Brauer group $H_{\et}^2(X,\G_m)$ typically contains non-torsion elements, in contrast with the algebraic Brauer group $H_{\et}^2(X^{\alg},\G_m)$ which is torsion \cite[Cor. IV.2.6]{milne1980etale}. This mirrors phenomena appearing in complex geometry: For $X$ a projective smooth variety over $\C$, $H_{\an}^2(X(\C),\G_m)$ is typically non-torsion, as can be seen using the exponential sequence.

Nevertheless, we obtain the following: Given a locally $p$-divisible rigid group $G$, we let $\widehat{G} \sub G$ denote the open subgroup of $p$-topological torsion of $G$, see Definition \ref{definition of topological p-torsion subgroup}.
\begin{proposition}{(Proposition \ref{representability of diamantine higher direct image for topologically torsion groups})}
    Let $X$ be a proper smooth rigid space and let $G$ be a locally $p$-divisible rigid group. Then for all $n\geq 0$ and $\tau \in \{\et,v\}$, the sheaf $\BBun_{\widehat{G},\tau}^n$ is representable by a locally $p$-divisible rigid group with Lie algebra $H_{\tau}^n(X, \Os_X\otimes_K \Lie(G))$.
\end{proposition}
For $G\neq \widehat{G}$ and $n\geq 2$, the sheaves $\BBun_{G,\tau}^n$ need not be representable by smooth rigid groups. We provide a counter-example in Corollary \ref{cor: non-representability of analytic Brauer functor}.

As an application, we obtain a new perspective on the presence of non-torsion elements in Brauer groups of proper smooth rigid analytic varieties (cf. Example \ref{Example: brauer groups shouldn't be torsion}): The $p$-adic exponential yields an open subgroup $U \sub \BBun_{\widehat{\G}_m,\et}^2$ isomorphic to a rigid polydisc $U\cong \B^d$, where $d=h^2(X,\Os_X)$. Moreover, we show in Proposition \ref{proposition 5-term exact sequence and abutment} that the kernel of the map of sheaves
\[ \BBun_{\widehat{\G}_m,\et}^2 \rightarrow \BBun_{\G_m,\et}^2\]
is a discrete group, so that the image of $U$ is not torsion, as soon as $d>0$. This generalizes to any locally $p$-divisible group $G$ and any degree $n\geq 1$.

In this article, we construct an $E_2$-spectral sequence of étale sheaves on $\Perf_K$
\begin{equation}\label{geometric Hodge--Tate spectral sequence: intro}
 \mathbf{E}_2^{ij}(G) = \left\{\begin{aligned}
  &H_{\et}^i(X,\Omega_X^j)(-j)\otimes \Lie(G)\otimes_K \G_a \quad &\text{if }j>0\\
  &\BBun_{G,\et}^i \quad &\text{if }j=0
\end{aligned}\right\}  \Longrightarrow \BBun_{G,v}^{i+j},
\end{equation}
that we call the \textbf{geometric Hodge--Tate spectral sequence with} $G$-\textbf{coefficients}. We obtain the following.
\begin{thm}{(Theorem \ref{thm: geometric Hodge--Tate spectral sequence for general G degenerates})}\label{Theorem 2}
Let $X$ be a proper smooth rigid space over a perfectoid field extension $K$ of $\Q_p$ and let $G$ be a locally $p$-divisible rigid group. Then the geometric Hodge--Tate spectral sequence (\ref{geometric Hodge--Tate spectral sequence: intro}) degenerates at the $E_2$-page.
\end{thm}
In particular, setting $\BBun_{G,\tau} \coloneqq \BBun_{G,\tau}^1$, Theorem \ref{Theorem 2} yields a natural short exact sequence on $\Perf_{K,\et}$
\begin{equation}\label{formula: geometric Hodge--Tate sequence for locally p-divisible G}\begin{tikzcd}
	0 & {\BBun_{G,\et}} & {\BBun_{G,v}} & {H^0(X,\Omega_X^1)(-1)\otimes\Lie(G)\otimes_K \G_a} & {0.}
	\arrow[from=1-1, to=1-2]
	\arrow[from=1-2, to=1-3]
	\arrow[from=1-3, to=1-4]
	\arrow[from=1-4, to=1-5]
\end{tikzcd}\end{equation}
This reproves the existence of the exact sequence in \cite[Thm. 2.7]{heuer2023diamantine} for $G=\G_m$ and geometrizes the short exact sequence \cite[Thm. 7.3.(1)]{heuer2023padic}. In particular, this yields evidence for a new instance of the geometric $p$-adic Simpson correspondence, as we explain further below.

For algebraically closed $K$, taking the $K$-rational points of the geometric spectral sequence (\ref{geometric Hodge--Tate spectral sequence: intro}) gives back the group-theoretic spectral sequence (\ref{classical Hodge--Tate spectral sequence for general G: intro}). Hence Theorem \ref{Theorem 2} implies Theorem \ref{Theorem 1}. This was one motivation for introducing the moduli-theoretic spectral sequences (\ref{geometric Hodge--Tate spectral sequence: intro}), as it is the geometry that makes the proof of the degeneration available. 

Moreover, we also exploit the geometric properties of the sheaves $\BBun_{G,\tau}^n$ in order to split the sequence (\ref{classical Hodge--Tate spectral sequence for general G: intro}). Indeed, we recall the following result. 
\begin{lemma}{(\cite[Thm. 6.12]{heuer2023padic})}\label{lemma intro: exponential}
    Let $H$ be a locally $p$-divisible rigid group over $K$ and assume that $K$ is algebraically closed. Then a choice of exponential $\Exp$ for $K$ induces a natural exponential
    \[ \Exp_H\colon \Lie(H) \rightarrow H(K).\]
\end{lemma}
Given a $\BdR+/\xi^2$-lift $\X$ of $X$ and an exponential $\Exp$ for $K$, we therefore obtain two splittings of the maps below
\[\begin{tikzcd}
	{\Fil^iH_{v}^n(X,\widehat{G})} & {\Fil^iH_{v}^n(X,\Os_X \otimes\Lie(G))} & {H^i(X,\Omega_X^{n-i})(i-n)\otimes \Lie(G),}
	\arrow["\log", shift left, from=1-1, to=1-2]
	\arrow["\Exp", shift left, dashed, from=1-2, to=1-1]
	\arrow[shift left, from=1-2, to=1-3]
	\arrow["{s_{\X}}", shift left, dashed, from=1-3, to=1-2]
\end{tikzcd}\]
whose composition is the required splitting. This yields the Hodge--Tate decomposition and concludes the proof of Theorem \ref{Theorem 0}.

We also extend our results to a class of smooth families of commutative adic groups $\Gs \rightarrow X$. These are the so-called admissible locally $p$-divisible $X$-groups, see Definition \ref{Def: admissible locally p-divisible groups} for the precise definition. In addition to constant families $G\times_K X$, where $G$ is a locally $p$-divisible rigid group, this also includes vector bundles, families of analytic $p$-divisible groups \cite[Def. 4.1]{Farg22} as well as $\Gs = B^{\times}$, for $B$ a finite locally free commutative $\Os_X$-algebra. This last example appears naturally in the context of abelianization in the $p$-adic Simpson correspondence \cite{heuer2023simpsoncorr} \cite{heuer2024curve}. An admissible locally $p$-divisible group $\Gs \rightarrow X$ defines a $v$-sheaf on $X$ and we show that analogs of all aforementioned results hold for such groups $\Gs$.

We further extend our results to the relative setting of proper smooth morphisms $\pi\colon X\rightarrow S$ of rigid spaces. Relative variant of the Hodge--Tate spectral sequence were considered e.g. in \cite[§2.2]{ScholzeCariani2017}, \cite{ABBES_2024}, \cite[§7.1]{gaisin2022relativearminfcohomology} and \cite{heuer2024relative}. In the latter, which is the closest to our technical setup, Heuer shows that, for reduced $S$, we have a spectral sequence of $v$-vector bundles \cite[Cor. 5.12, Thm. 5.7.3]{heuer2024relative}
\begin{align}\label{eq: relative HT sseq of Heuer}
    E_2^{ij} = R^i\pi_{\et,*}\Omega_{X/S}^j \otimes_{\Os_{S_{\et}}}\Os_{S_{v}}(-j) \Longrightarrow (R^{i+j}\pi_{v,*}\Q_p) \otimes_{\Q_p} \Os_{S_v}
\end{align}
that degenerates at $E_2$.

In our situation, we also have a relative variant of the geometric Hodge--Tate spectral sequences. Let $\pi\colon X\rightarrow S$ be a proper smooth morphism of \emph{seminormal} rigid spaces. We define analogously 
\[ \BBun_{\Gs,X/S,\tau}^n \coloneqq R^n\pi_{\tau,*}\Gs\colon \Perf_{S,\tau} \rightarrow \Ab,\]
where $\tau$ is either the étale or the $v$-topology. Our most general statement is the following.
\begin{thm}{(Theorem \ref{thm: geometric Hodge--Tate spectral sequence for general G degenerates}, Corollary \ref{possible simplification of the abutment})}\label{theorem 3}
Let $\pi\colon X \rightarrow S$ be a proper smooth morphism of seminormal rigid spaces over a perfectoid field extension $K$ of $\Q_p$, and let $\Gs$ be an admissible locally $p$-divisible $X$-group. Assume that $S=\Spa(K)$ or that $\Lie(\Gs)$ comes via pullback from a vector bundle on $S$. Then there is an $E_2$-spectral sequence of sheaves on $\Perf_{S,\et}$  \begin{equation}\label{relative geometric Hodge--Tate spectral sequence: intro}
 \mathbf{E}_2^{ij}(\Gs) = \left\{\begin{aligned}
  &R^i\pi_{\Et,*}(\Omega_X^j\otimes_{\Os_X} \Lie(\Gs))(-j) \quad &\text{if }j>0\\
  &\BBun_{\Gs,X/S,\et}^i \quad &\text{if }j=0
\end{aligned}\right\}  \Longrightarrow \BBun_{\Gs,X/S,v}^{i+j}.
\end{equation}
Moreover, this spectral sequence degenerates at the $E_2$-page. 
\end{thm}
The degeneration of the relative Hodge--Tate spectral sequence (\ref{eq: relative HT sseq of Heuer}) is an essential ingredient for our proof of Theorem \ref{theorem 3}. As a corollary, we obtain the following.
\begin{cor}{(Remark \ref{Remark: BunGet is a v-sheaf})}
    In the situation of Theorem \ref{theorem 3}, the sheaves $\BBun_{\Gs,X/S,\et}^n$ are small $v$-sheaves on $\Perf_S$.
\end{cor}

\subsection{Application to non-abelian $p$-adic Hodge theory}\label{subsection: application to non-abelian $p$-adic Hodge theory}
At last, we connect our results to the $p$-adic Simpson correspondence. Let $X$ a proper smooth connected rigid space over an algebraically closed field $K/\Q_p$ and fix $x\in X(K)$. The field of non-abelian $p$-adic Hodge theory arose from the wish to find a $p$-adic analog of the Corlette--Simpson correspondence in complex geometry. Initiated by Deninger--Werner \cite{DeningerWerner2005} and Faltings \cite{Faltings2005} independently, it was first envisaged (in their respective settings) as a correspondence between continuous $K$-linear (generalized) representations of the étale fundamental group $\pi_1(X,x)$ and Higgs bundles $(E,\theta)$ on $X$ satisfying certain properties. Reformulated in terms of Scholze's perfectoid foundations for $p$-adic Hodge theory, it is now understood as an equivalence of categories \cite[Thm. 1.1]{heuer2023simpsoncorr}
\begin{equation}\label{classical p-adic Simpson}\begin{aligned}
    \{v\text{-vector bundles } V \text{ on }X\,\} \xlongrightarrow{\sim} \{\, \text{Higgs bundles } (E,\theta) \text{ on }X \,\}.
\end{aligned}  \end{equation}
depending on choice of $\BdR+/\xi^2$-lift $\X$ of $X$ and an exponential $\Exp$ for $K$. A first instance of this modern formulation is the isomorphism \cite[Thm. 5.7]{heuer2021line}
\begin{align}\label{eq: tt mult HT of degree 1}
    \Hom_{\cts}(\pi_1(X,x),K^{\times}) = \Pic_{\et,\profet}(X) \oplus H^0(X,\Omega_X^1)(-1),
\end{align}
where $\Pic_{\et,\profet}(X)$ consists of those étale line bundles that are trivialized by a pro-finite-étale cover $\widetilde{X} \rightarrow X$. From \emph{loc. cit}, one can also deduce a decomposition
\begin{align}\label{eq: mult HT of degree 1}
    H_v^1(X,\G_m) = \Pic_{\et}(X) \oplus H^0(X,\Omega_X^1)(-1),
\end{align}
underlying a $p$-adic Simpson correspondence for line bundles. These two decompositions were later generalized \cite[Thm. 7.3-4]{heuer2023padic} to locally $p$-divisible rigid group $G$ in place of $\G_m$, and it is shown there that it induces a $p$-adic Simpson correspondence for $G$-torsors. The Hodge--Tate decomposition of Theorem \ref{Theorem 0} recovers (\ref{eq: mult HT of degree 1}) and its extensions to general $G$ as the special case $n=1$. Moreover, our decomposition with $n> 1$
\[ H_v^n(X,\G_m) = H_{\et}^n(X,\G_m) \oplus \bigoplus_{i=0}^{n-1}H^i(X,\Omega_X^{n-i})(i-n)\]
is new already for $\G_m$ and the author is not aware of an analogous result in complex geometry. We suspect that it is an instance of a higher analogue of the $p$-adic Simpson correspondence, involving étale and $v$-topological $n$-gerbes. We also obtain a version of the decomposition involving the étale fundamental group $\pi_1(X,x)$ in the case where $X$ is either a curve of genus $g\geq 1$ or an abeloid variety, cf. Theorem \ref{thm: case of abeloids and curves}.

Moreover, the perspective of moduli spaces in non-abelian $p$-adic Hodge theory was explored in \cite{heuer2022geometric} and \cite{heuer2024curve}, where Heuer--Xu show that the $p$-adic Simpson correspondence can be upgraded from an equivalence of categories to a comparison of small $v$-stacks. Our results yield a new instance of this geometric $p$-adic Simpson correspondence. Namely, the sequence (\ref{formula: geometric Hodge--Tate sequence for locally p-divisible G}) can be interpreted geometrically as saying that the coarse moduli space $\BBun_{G,v}$ of $v$-$G$-torsors on $X$ is a $\BBun_{G,\et}$-torsor over the “Hitchin base” $H^0(X,\widetilde{\Omega}_X^1)(-1) \otimes \Lie(G)\otimes_K \G_a$. We claim that, under mild assumptions on $G$, one can derive from (\ref{formula: geometric Hodge--Tate sequence for locally p-divisible G}) a geometric $p$-adic Simpson correspondence for $G$-torsors, following the strategy of \cite{heuer2022geometric}.

Finally, it was discovered in \cite{heuer2023simpsoncorr} that the more difficult higher rank correspondence could be tackled by trading the non-abelian $p$-adic Hodge theory of the group $\GL_n$ for the relative $p$-adic Hodge theory of smooth families of abelian groups $\Gs \rightarrow X$, through the process of \emph{abelianization}. Our approach yields a systematic study of smooth commutative relative groups and their Hodge--Tate theory in a more general setting. Given a finite locally free commutative $\Os_X$-algebra $B$ on a proper smooth rigid space $X$ over $K$, Theorem \ref{Theorem 2} yields a short exact sequence
\begin{equation}\label{eq: Bx short exact sequence}\begin{tikzcd}
	0 & {\BBun_{B^{\times},\et}} & {\BBun_{B^{\times},v}} & {H^0(X,\Omega_X^1\otimes_{\Os_X}B)(-1) \otimes_K \G_a} & {0.}
	\arrow[from=1-1, to=1-2]
	\arrow[from=1-2, to=1-3]
	\arrow[from=1-3, to=1-4]
	\arrow[from=1-4, to=1-5]
\end{tikzcd}\end{equation}
This recovers \cite[Thm. 1.5]{heuer2023simpsoncorr}. Note that via the canonical decomposition of Theorem \ref{Theorem 0}
\[ H_v^1(X,B^{\times}) = H_{\et}^1(X,B^{\times}) \oplus H^0(X,\Omega_X^1 \otimes B)(-1),\]
we recover the association $\theta \in H^0(X,\Omega_X^1 \otimes B)(-1) \mapsto [\Ls_{\theta}] \in H_v^1(X,B^{\times})$ of \cite[4]{heuer2023simpsoncorr} at least when $B$ is locally free. We note that it would be interesting to extend the methods of this article to more general sheaves of groups over $X$ such as $B^{\times}$, for $B$ a coherent commutative $\Os_X$-algebra.

\subsection*{Acknowledgment}
This project originated from a discussion with Ben Heuer, where the question of the degeneration of the multiplicative Hodge--Tate spectral sequence was raised. We thank him heartily for many helpful conversations, especially his suggestion to think about geometric Hodge--Tate spectral sequences, for answering all our questions and for carefully reading through the project's early versions. We also thank Annette Werner for very helpful discussions and for her many insightful comments on early drafts of this article. We thank Annie Littler, Konrad Zou and Jefferson Baudin for their comments and for interesting discussions. Finally, we thank the anonymous referee for helpful comments.

The present article is part of the PhD thesis of the author. This project was funded by Deutsche
Forschungsgemeinschaft (DFG, German Research Foundation) through the Collaborative Research Centre TRR 326 \textit{Geometry and Arithmetic of Uniformized Structures} - Project-ID 444845124. Our research was also partially supported by the Simons Foundation through the Simons Collaboration on Perfection in Algebra, Geometry, and Topology.

\subsection*{Notations and conventions}
\begin{itemize}
    \item We fix a prime number $p$. We denote by $K$ a non-archimedean field extension of $\Q_p$, with ring of integers $\Os_K$ and subspace of topologically nilpotent elements $\ma$. More generally, we consider open and bounded valuation subrings $K^+ \sub K$ contained in $\Os_K$. Such a pair $(K,K^+)$ will sometimes be referred to as a non-archimedean field. Throughout, we will work with analytic adic spaces in the sense of Huber \cite{huber2013étale}. For simplicity, we will simply talk about adic spaces over $K$ when dealing with adic spaces over $\Spa(K,K^+)$.
    \item A rigid space over $K$ will mean for us an adic space locally of topologically finite type over $K$. When $K^+=\Os_K$, there is an equivalence of categories between (quasi-)separated rigid analytic varieties in the sense of Tate and (quasi-)separated rigid spaces in the above sense \cite[(1.1.11)]{huber2013étale}.
    \item We work with perfectoid spaces in the sense of \cite{scholze2011perfectoid}. We let $\Perf_K$ denote the category of affinoid perfectoid spaces over $K$, or equivalently of perfectoid $(K,K^+)$-algebras. We may equip it with the étale or the $v$-topology, and we denote the resulting sites by $\Perf_{K,\et}$ and $\Perf_{K,v}$.
    \item Throughout we will often work with sousperfectoid adic spaces over $\Q_p$, i.e. adic spaces that are locally of the form $\Spa(R,R^+)$ where $R$ is sousperfectoid \cite[§6.3]{SW20}. This contains both perfectoid spaces and smooth rigid spaces over perfectoid fields. 
    \item If $X$ is either a sousperfectoid space or a rigid space, the étale site $X_{\et}$ of Kedlaya--Liu \cite[Def. 8.2.19]{kedlaya2015relative} consists only of sheafy adic spaces. A map $f\colon Y \rightarrow X$ is called standard-étale if $X$ and $Y$ are affinoid and $f$ can be written as a composition of finite-étale maps and rational open immersions. These form a basis of the étale site. More generally, the theory of smooth adic spaces over sousperfectoid spaces is well-behaved, see \cite[§IV.4.1]{fargues2024geometrization}. In particular, this produces sousperfectoid spaces again. A map $Y\rightarrow X$ is called standard-smooth if it can be written as a composition of finite-étale maps, rational open immersions and projections $\B_S^1=\B^1 \times_{\Spa(\Q_p)} S \rightarrow S$, where $\B^1=\Spa(\Q_p\langle T \rangle)$.
    \item Scholze's diamond functor \cite[§15]{scholze2022etale} associates to any adic space $X$ over $\Q_p$ a diamond $X^{\diamondsuit}$ over $\Spd(\Q_p)$. Under the equivalence $\Perf_{\F_p, /\Spd(\Q_p)} \cong \Perf_{\Q_p}$, sending a characteristic $p$ perfectoid space $S\rightarrow \Spd(\Q_p)$ to the corresponding untilt $S^{\sharp} \rightarrow \Spa(\Q_p)$, we may equivalently consider $X^{\diamondsuit}$ as a sheaf on $\Perf_{\Q_p,v}$, namely the functor of points. By \cite[Lemma 15.6]{scholze2022etale}, for any adic space $X$ over $\Q_p$, we have an equivalence of sites $X_{\et} \cong X_{\et}^{\diamondsuit}$. Furthermore, we will often restrict to a situation where the functor $(\cdot)^{\diamondsuit}$ is fully faithful. Thus, we will freely switch back and forth between adic spaces and their associated diamonds, and we will sometimes omit the diamond symbol when the context is clear.
    \item We will also require to work with various big sites over adic spaces. If $X$ is an adic space over $\Q_p$, we denote by $X_v = \LSD_{X,v}$ the $v$-site of $X$, with underlying category all locally spatial diamonds over $X^{\diamondsuit}$ \cite[§14]{scholze2022etale}. We will occasionally also equip this category with the étale topology and denote it by $\LSD_{X,\et}$. We will also denote by $\Perf_X$ the category of affinoid perfectoid spaces living over $X$, which we can equip either with the étale or the $v$-topology. Note that $\Perf_X$ need not be a slice category. If $X$ is a sousperfectoid space or a rigid space, we also set $\Sm_{/X,\et}$ to be the category of adic spaces smooth over $X$, equipped with the étale topology.
    \item Given a morphism of adic space $\pi\colon X\rightarrow S$, we have an associated morphism of small étale site $\pi_{\et}\colon X_{\et} \rightarrow S_{\et}$. We will also consider the induced morphism between big étale sites, for a choice of big étale site, e.g. $\Sm_{/S,\et}$ or $\Perf_{S,\et}$. To emphasize this difference, we will denote this morphism by $\pi_{\Et}$.
    \item This article is primarily concerned with commutative groups. Therefore, all groups will be taken to be commutative, unless mentioned otherwise.
\end{itemize}

\section{Preliminaries}\label{Section: Preliminaries}

\subsection{Good adic spaces}
We define a class of adic spaces that we will be working with throughout the paper.
\begin{definition}\label{def: good adic spaces}
    Let $X$ be an adic space over $\Q_p$ that is either a sousperfectoid space or a rigid space over some non-archimedean field $(K,K^+)$. We say that $X$ is good if the natural map
    \[ \Os_{X_{\et}} \rightarrow\nu_*\Os_{X_v}\]
    is an isomorphism, where $\nu\colon X_v \rightarrow X_{\et}$ is the natural map of sites. We denote by $\Adic_{\Q_p}$ the category of good adic spaces over $\Q_p$.
\end{definition}

\begin{example}
    \begin{enumerate}
        \item Any perfectoid space is a good adic space.
        \item Let $X$ be a rigid space over a non-archimedean field $(K,K^+)$ over $\Q_p$. Then by \cite[Thm. 10.3, Lemma 6.4]{Kedlaya2020Sheafiness}, $X$ is good if and only if it is seminormal.
    \end{enumerate}
\end{example}
\begin{remark}
    A similar condition on adic spaces was recently introduced in \cite[Def. 3.1.9]{graham2025padicfouriertheoryfamilies} under the name locally fiercely $v$-complete adic spaces. We show below in Proposition \ref{Prop: good spaces closed under smooth maps} that good adic spaces over $\Q_p$ are stable under smooth maps. It follows that the good adic spaces over $\Q_p$ are exactly the locally fiercely $v$-complete spaces that are either sousperfectoid or rigid spaces. We do not know whether any sousperfectoid space over $\Q_p$ is good.
\end{remark}

\begin{proposition}\label{prop: diamond functor ff on good adic spaces}
    Let $X$ be a good adic space over $\Q_p$.
    \begin{enumerate}
        \item Let $\nu\colon X_v \rightarrow X_{\et}$ denote the natural map of sites, then we have isomorphisms
        \[ \Os_{X_{\et}}^+ \xrightarrow{\cong} \nu_*\Os_{X_v}^+, \quad \Os_{X_{\et}}^+/p^n \xrightarrow{\cong} \nu_*(\Os_{X_v}^+/p^n), \quad \Os_{X_{\et}}^+ = \varprojlim_n \Os_{X_{\et}}^+/p^n .\]
        \item For $Y$ another good adic space over $\Q_p$, we have a natural isomorphism
        \[ \Hom_{\Spa(\Q_p)}(X,Y) \xrightarrow{\cong} \Hom_{\Spd(\Q_p)}(X^{\diamondsuit},Y^{\diamondsuit}).\]
    \end{enumerate}
\end{proposition}
\begin{proof}
    To prove the second point, we may work locally on $\vert Y \vert = \vert Y^{\diamondsuit} \vert$ and assume that $Y=\Spa(B,B^+)$ is affinoid. We have the following adjunction, left to the reader
    \[ \Hom_{\Spd(\Q_p)}(X^{\diamondsuit},\Spd(B,B^+)) = \Hom_{\cts}((B,B^+),(\Os_v(X^{\diamondsuit}),\Os_v^{+}(X^{\diamondsuit}))).\]
    It thus remains to show that the natural map
    \[ \Os^+(X) \rightarrow \Os_v^+(X^{\diamondsuit})\]
    is an isomorphism, so that it suffices to show the first point. The isomorphism $\nu_*(\Os_{X_v}^+/p^n) = \Os_{X_{\et}}^+/p^n$ is \cite[Prop. 2.8, Prop. 2.13]{MannWerner2022loc} when $X$ is a rigid space and \cite[Corollary 2.15]{heuer2022gtorsors} when $X$ is sousperfectoid. The isomorphism $\nu_*\Os_{X_v}^+ = \Os_{X_{\et}}^+$ follows from $\nu_*\Os_{X_v} = \Os_{X_{\et}}$. Finally, to show the last isomorphism, it is enough to show the isomorphism
    \[\Os_{X_{v}}^+ \xrightarrow{\cong }\varprojlim_n \Os_{X_{v}}^+/p^n.\]
    This can be checked on affinoid perfectoids $Y$. Injectivity is clear. For surjectivity, given a collection $(s_n)_{n\geq 1}$ of compatible sections $s_n \in \Os_v^+/p^n(Y)$, let $c_n\in H_v^1(Y,\Os_v^+)$ be the obstruction class to lifting $s_n$ to a section in $\Os^+(Y)/p^n$. Then $c_n = pc_{n+1}$. Since $H_v^1(Y,\Os_v^+)$ is almost zero, by \cite[Prop. 8.8]{scholze2022etale}, it is in particular annihilated by $p$. Hence $c_n$ vanishes and $(s_n)_n$ is in the image of the above map, as required. This concludes the proof. 
    \end{proof}

    \begin{proposition}\label{Prop: good spaces closed under smooth maps}
    Let $X$ be a good adic space over $\Q_p$ and let $f\colon X'\rightarrow X$ be a smooth map. Then $X'$ is a good adic space.
\end{proposition}
\begin{proof}
    The claim is local on $X_{\et}$. Therefore, we may assume that $X=\Spa(A,A^+)$, $X'=\Spa(B,B^+)$, and $f$ is standard-smooth, i.e a composition of finite-étale maps, rational open immersions and projections $\B_S^1 \rightarrow S$. We may further assume that $X'=\B_X^1$, so that $B=A\langle T \rangle$. Fix a pro-étale présentation $X^{\diamondsuit} = Y/R$ for perfectoids $Y$ and $R\sub Y\times Y$. By Proposition \ref{prop: diamond functor ff on good adic spaces}(1), we have left exact sequences, for each $n\geq 1$
\[\begin{tikzcd}
	0 & {\Os_{\et}^+/p^n(X)} & {\Os_{\et}^+/p^n(Y)} & {\Os_{\et}^+/p^n(R).}
	\arrow[from=1-1, to=1-2]
	\arrow[from=1-2, to=1-3]
	\arrow["{p_1^*-p_2^*}", from=1-3, to=1-4]
\end{tikzcd}\]
Upon adjoining a polynomial variable, taking the inverse limit over $n$ and inverting $p$, we obtain the left exact sequence
\[\begin{tikzcd}
	0 & {\Os_{\et}(\B_X^1)} & {\Os_{\et}(\B_Y^1)} & {\Os_{\et}(\B_R^1).}
	\arrow[from=1-1, to=1-2]
	\arrow[from=1-2, to=1-3]
	\arrow[from=1-3, to=1-4]
\end{tikzcd}\]
By a $5$-lemma argument, to show that
\[ \Os_{\et}(\B_X^1) \cong \Os_v(\B_X^{1,\diamondsuit}),  \]
it is enough to show that
\[ \Os_{\et}(\B_Y^1) \cong \Os_v(\B_Y^{1,\diamondsuit}), \quad \Os_{\et}(\B_R^1) \cong \Os_v(\B_R^{1,\diamondsuit}).  \]
Hence, we may assume that $X$ is affinoid perfectoid. In that case, $X$ is diamantine, in the sense of \cite[Def. 11.1]{Kedlaya2020Sheafiness}, so that the result follows from \cite[Thm. 11.18]{Kedlaya2020Sheafiness}.
\end{proof}

We denote by $\Adic_{X}$ the category of good adic spaces over a fixed good adic space $X$. By Proposition \ref{Prop: good spaces closed under smooth maps}, we may endow it with the étale topology, and we denote by $\Adic_{X,\et}$ the resulting site. We will also sometimes consider $v$-sheaves on $\Adic_{X}$. These are defined to be presheaves $\Fs$ such that, for any $v$-cover $Y' \rightarrow Y$ of good adic spaces over $X$ (i.e. a map that induces a $v$-cover $Y'^{\diamondsuit} \rightarrow Y^{\diamondsuit}$) and for any $v$-cover $Y'' \rightarrow Y'^{\diamondsuit} \times_{Y^{\diamondsuit}}Y'^{\diamondsuit}$ from a good adic space $Y''$, we have
\[\begin{tikzcd}
	{\Fs(Y)=\Eq(\Fs(Y')} & {\Fs(Y'')).}
	\arrow[shift right, from=1-1, to=1-2]
	\arrow[shift left, from=1-1, to=1-2]
\end{tikzcd}\]
These are equivalent to sheaves on the site $X_v$.

\subsection{Smoothoid spaces}\label{subsection: Smoothoid spaces}

Fix a perfectoid field $(K,K^+)$ over $\Q_p$. We recall some facts about smoothoid spaces, introduced in \cite{heuer2022moduli}, which form a subclass of good adic spaces over $K$ admitting a well-behaved notion of differentials. All the results of this section are taken from \cite[§2]{heuer2022moduli}.

\begin{definition}{(\cite[Def. 2.2]{heuer2022moduli})}\label{def: Smoothoid spaces}
    An adic space $X$ over $K$ is called smoothoid if it admits an open cover by adic spaces that are smooth over a perfectoid space $Y$ over $K$. We denote by $\Smd_K$ the category of smoothoid spaces over $K$. 
\end{definition}

\begin{example}
    Given a smooth morphism of rigid spaces $X\rightarrow S$ over $K$ and a morphism $Y\rightarrow S$ from a perfectoid space $Y$, the fibre product $X\times_S Y$ is smoothoid. 
\end{example} 

By \cite[Lemma 2.6]{heuer2022moduli}, any smoothoid space over $K$ is sousperfectoid and a good adic space in the sense of Definition \ref{def: good adic spaces}. Moreover, the category $\Smd_K$ of smoothoid spaces is clearly closed under taking small étale sites. We may thus turn $\Smd_K$ into a site by equipping it with the étale topology, which we denote by $\Smd_{K,\et}$.

\begin{definition}{(\cite[Def. 2.10]{heuer2022moduli})}
    Let $X$ be a smoothoid space and let $\nu\colon X_v \rightarrow X_{\et}$ be the natural map of sites. We set, for any $n\geq 0$
    \[ \womega_X^n \coloneqq R^n\nu_*\Os_X.\]
\end{definition}
\begin{example}\label{Example: womega of fiber product}
\begin{enumerate}
    \item For $X$ a smooth rigid space, we have \cite[Prop. 3.23]{Scholze2013}
\[ \womega_X^1 = \Omega_{X/K}^1\{-1\},\]
where $\Omega_{X/K}^1$ is the sheaf of relative differentials of Huber \cite[(1.6.2)]{huber2013étale} and $\{-1\}$ denotes a Breuil--Kisin--Fargues twist \cite[Def. 2.24]{heuer2021line}. If $K$ contains $\Q_p(\mu_{p^{\infty}})$, this is canonically isomorphic to a Tate twist, by adapting the proof of \cite[Prop. 6.7]{scholze2013padicHodge}. In general, if $C$ denotes a completed algebraic closure of $K$, we have a canonical, $\Gal_K$-equivariant isomorphism
\[ \womega_X^1 \otimes_K C \cong \Omega_{X/K}^1\otimes_K C(-1).\]
\item Let $Y$ be a perfectoid space over $K$, then by $v$-acyclicity of affinoid perfectoid spaces \cite[Prop. 8.8]{scholze2022etale}, $\womega_Y^n = 0$.
\item Given a diagram
\[\begin{tikzcd}
	& Y \\
	X & {S,}
	\arrow["g", from=1-2, to=2-2]
	\arrow["\pi", from=2-1, to=2-2]
\end{tikzcd}\]
where $\pi$ is a smooth morphism of rigid spaces and $Y$ is any perfectoid space, we have \cite[Prop. 2.9.(2)]{heuer2022moduli}
\[ \womega_{X\times_S Y}^1 = g^*\Omega_{X/S}^1\{-1\}.\]
This generalizes the two previous examples.
\end{enumerate}
\end{example}

For a general smoothoid space $X$, the sheaf $\womega_X^n$ is a vector bundle \cite[Prop. 2.9.(1)]{heuer2022moduli}, satisfying all the compatibility properties of the Kähler differentials. The assignement
\begin{align}\label{formula: womega}
    \womega^n\colon \Smd_K \rightarrow \Ab\, , \quad X \mapsto \Gamma(X,\widetilde{\Omega}_X^n)
\end{align}
defines a sheaf for the étale topology and we have
 \[ \womega^n = R^n\mu_*\Os,\]
 where $\mu\colon K_v \rightarrow \Smd_{K,\et}$ is the natural map of sites \cite[Lemma 2.12]{heuer2022moduli}.

\subsection{Zariski-constructible sheaves}\label{subsection: Zariski-constructible sheaves}
We recall the notion of Zariski-constructible sheaves on rigid spaces in characteristic $0$ and some of their basic properties. Let $X$ be a rigid space over a non-archimedean field $(K,K^+)$ over $\Q_p$. Recall that $X$ comes with a natural Zariski topology $X_{\Zar}$ \cite[§2.1]{newton2019irred}. 
\begin{definition}{(\cite{Hansen_2020})}\label{Def: zariski-constructible sheaves}
    A sheaf $\Fs$ on $X_{\et}$ is Zariski-constructible if there exists a locally finite stratification $X=\coprod_i X_i$ where each $X_i$ is locally closed in $X$ in the Zariski topology and $\Fs\restr{X_i}$ is étale-locally constant finite. The sheaf $\Fs$ is ind-Zariski-constructible if it can be written as a filtered colimit of Zariski-constructible étale sheaves.
\end{definition}
\begin{remark}\label{Remark: zariski constructibility and overconvergence}
    The Zariski-constructible sheaves on $\Spa(L,L^+)$ are exactly the finite étale-locally constant sheaves. Hence any Zariski-constructible sheaf $\Fs$ on a rigid space $X$ is overconvergent \cite[Def. 8.2.1]{huber2013étale}. In particular, if $X_0 = X\times_{\Spa(K,K^+)} \Spa(K,\Os_K)$ and we let $i\colon X_0 \rightarrow X$ denote the induced inclusion, we have $i_*i^*\Fs = \Fs$. Moreover, the maps $i_*$ and $i^*$ are immediately seen to define an equivalence of categories between the respective categories of Zariski-constructible sheaves. This allows us to transport to our setup many results from \cite{Hansen_2020} and \cite{bhatt2021functors}, where there is a standing assumption that $K^+=\Os_K$.
\end{remark}

\begin{thm}{(\cite[Thm. 3.10, Thm. 3.15]{bhatt2021functors})}\label{Thm: properties of zariski-constructible sheaves}
    Let $\pi\colon X \rightarrow S$ be a proper map of rigid spaces over $K$. Let $\Fs$ be an ind-Zariski-constructible abelian sheaf on $X_{\et}$. Fix any $n\geq 0$.
    \begin{enumerate}
        \item The sheaf $R^n\pi_{\et,*}\Fs$ is an ind-Zariski-constructible sheaf on $S_{\et}$. It is Zariski-constructible if $\Fs$ is. 
        \item Let $Y\rightarrow S$ be any morphism from a rigid space over $K$ and form the pullback diagram
\[\begin{tikzcd}
	{X\times_S Y} & X \\
	Y & {S.}
	\arrow["{g'}", from=1-1, to=1-2]
	\arrow["{\pi'}"', from=1-1, to=2-1]
	\arrow["\pi", from=1-2, to=2-2]
	\arrow["g"', from=2-1, to=2-2]
\end{tikzcd}\]
Then we have a natural isomorphism of sheaves on $Y_{\et}$
\[ g^*R^n\pi_{\et,*}\Fs \cong R^n\pi'_{\et,*}(g'^*\Fs).  \]
    \end{enumerate}
\end{thm}
\begin{proof}
    Let us write $\Fs = \varinjlim_{i\in I} \Fs_i$ where $I$ is cofiltered and the $\Fs_i$ are Zariski-constructible. As $\pi\colon X \rightarrow S$ is qcqs, we have, by \cite[\href{https://stacks.math.columbia.edu/tag/0739}{Tag 0739}]{stacks-project}
    \[ R^n\pi_{\et,*}\Fs = \varinjlim_i R^n\pi_{\et,*}\Fs_i.  \]
    From this, we reduce to the case where $\Fs = \Fs_i$ is Zariski-constructible. Let $X_0,S_0,Y_0$ denote the respective base change along $\Spa(K,\Os_K) \rightarrow \Spa(K,K^+)$. By \cite[Prop. 8.2.3(ii)]{huber2013étale}, the sheaves $R^n\pi_{\et,*}\Fs$ are overconvergent. Furthermore, their formations are easily seen to commute with base change along $S_0 \sub S$. Hence, we may assume that $K^+=\Os_K$, in which case this is \cite[Thm. 3.10, Thm. 3.15]{bhatt2021functors}.
\end{proof}

\subsection{Locally $p$-divisible groups}\label{subsection: locally $p$-divisible groups}
We now introduce locally $p$-divisible groups, which will be the central objects of focus of this paper. At the same time, we will recall some facts about smooth relative adic groups, logarithms and exponentials that we will need throughout. We fix an adic space $X$ over $\Q_p$ that is either a sousperfectoid space or a rigid space over some non-archimedean field $(K,K^+)$ over $\Q_p$.

\begin{definition}
    A smooth relative group over $X$ is a group object $\Gs \rightarrow X$ in the category of adic spaces smooth over $X$.
\end{definition}

\begin{example}
    \begin{enumerate}
        \item Let $G$ be a rigid group over a non-archimedean field $K$. Then $G$ is smooth by \cite[Prop. 1]{Farg19}. Whenever $X$ lives over $K$, the product $G_X \coloneqq G\times_K X$ defines a smooth relative group over $X$. We may for example take $G=\G_a$, $\G_m$ or $\G_a^+ \coloneqq \B^1 = \Spa(K\langle T \rangle)$ the closed unit ball. We will sometimes write $G$ for $G_X$, as it will be clear from context which base we are considering. We refer to \cite{Farg19} and \cite{heuer2022gtorsors} for more background on rigid groups.
        \item Let $V$ be a finite dimensional $K$-vector space, then we have the smooth adic group $V \otimes_K \G_a$. We call a relative group of this form a vector group. More generally, if $E$ is a finite locally free $\Os_X$-module, the corresponding vector bundle $E \otimes_{\Os_X} \G_a$ is a smooth relative group over $X$.
        \item Let $\Gs^{\alg}$ be a smooth group scheme over a finite type $K$-scheme $X^{\alg}$. Then its analytification $(\Gs \rightarrow X) = (\Gs^{\alg} \rightarrow X^{\alg})^{\an}$ is a smooth relative group. 
         \item Assume that $X$ is a rigid space over $K$ and let $\Xs$ be an admissible formal scheme over $\Spf(K^+)$ with adic generic fibre $X$. Let $\mathfrak{S}$ be a smooth formal group scheme over $\Xs$, then its adic generic fibre $\mathfrak{S}_{\eta}$ is a smooth relative group over $X$. A smooth $X$-group is said to have good reduction if it is of this form.
    \end{enumerate}
\end{example}

\begin{definition}
    Given a smooth relative group $\Gs \rightarrow X$ with identity section $e\colon X \rightarrow \Gs$, its Lie algebra is defined to be
\[ \Lie(\Gs) \coloneqq (e^*\Omega_{\Gs/X}^1)^{\vee}.\] 
It is a finite-dimensional locally free $\Os_X$-module and we let
\[ \ga \coloneqq \Lie(\Gs) \otimes_{\Os_X} \G_a\]
denote the underlying geometric vector bundle.
\end{definition}

We now move to logarithms. Given a commutative rigid group $G$, Fargues \cite[Théorème 1.2]{Farg19} defines the subgroup $\widehat{G}$ of $p$-topological torsion of $G$ and shows that it is the domain of a logarithm map, functorially associated with $G$. These definitions were later recast by Heuer in terms of $v$-sheaves \cite[Def. 2.6]{heuer2022geometric} and generalized by Heuer--Xu to relative groups \cite[§3]{heuer2024curve}.
\begin{definition}{(\cite[Def. 2.6]{heuer2022geometric})}\label{definition of topological p-torsion subgroup}
    Let $\Fs$ be an abelian sheaf on $X_v$. The $p$-topological torsion subsheaf $\widehat{\Fs} \sub \Fs$ is defined to be the image as sheaves on $X_v$ of the map
    \[ \ev_1\colon \underline{\Hom}(\underline{\Z_p},\Fs) \rightarrow \Fs,\]
    where $\underline{\Z_p}$ denotes the profinite sheaf $\varprojlim_n \, \underline{\Z/p^n}$.
\end{definition}

\begin{lemma}{(\cite[Prop. 3.2.4]{heuer2024curve})}\label{statements about p-topological torsion subgroup}
Let $\Gs \rightarrow X$ be a commutative smooth relative group. 
    \begin{enumerate}
        \item The $v$-sheaf $\widehat{\Gs}$ is representable by an open subgroup of $\Gs$ and evaluation at $1$ defines an isomorphism of $v$-sheaves
        \[ \underline{\Hom}(\underline{\Z_p},\Gs) \xrightarrow{\cong} \widehat{\Gs}. \]
        \item There is a unique homomorphism of adic groups
        \[ \log_{\Gs}\colon \widehat{\Gs} \rightarrow \ga\]
        whose derivative is $D(\log_{\Gs}) = \id\colon \Lie(\Gs)\rightarrow \Lie(\Gs)$. The map $\log_{\Gs}$ is natural in $\Gs$ and fits in an exact sequence of smooth $X$-groups
\begin{equation}\label{logarithm exact sequence}\begin{tikzcd}
	0 & {\Gs[p^{\infty}]} & {\widehat{\Gs}} & {\ga.}
	\arrow[from=1-1, to=1-2]
	\arrow["{\log_{\Gs}}", from=1-3, to=1-4]
	\arrow[from=1-2, to=1-3]
\end{tikzcd}\end{equation}
\item If $[p]\colon \Gs \rightarrow \Gs$ is étale, so is $\log_{\Gs}$.
    \end{enumerate}
\end{lemma}

\begin{remark}\label{rmk: p map automatically étale}
When $X$ is a smooth rigid space, $[p]\colon \Gs \rightarrow \Gs$ is automatically étale by \cite[Lemma 3.2.2]{heuer2024curve}.
\end{remark}

\begin{definition}
     Let $\Gs \rightarrow X$ be a smooth commutative group. We define the sheaf $\cj{\Gs}$ on $\Adic_{X,\et}$ to be the quotient sheaf 
     \[ \cj{\Gs} = \Gs/\widehat{\Gs}.\]
 \end{definition}
\begin{remark}
    As we will see later in Proposition \ref{Prop: approximating sheaves are small v-sheaves}, $\cj{\Gs}$ is a $v$-sheaf, such that we get the same result if we compute the quotient in the category of $v$-sheaves instead.
\end{remark}

We now introduce a relative variant of the locally $p$-divisible rigid groups of \cite[§6]{heuer2023padic}.
\begin{definition}
    Let $\Gs$ be a commutative smooth $X$-group.
    \begin{enumerate}
        \item We say that $\Gs$ is an analytic $p$-divisible $X$-group if $\Gs=\widehat{\Gs}$ and $[p]\colon\Gs \rightarrow \Gs$ is finite étale and surjective (\cite[Def. 4.1]{Farg22}).
        \item We say that $\Gs$ is a locally $p$-divisible $X$-group if $\log_{\Gs}\colon \widehat{\Gs} \rightarrow \ga$ is étale and surjective.
    \end{enumerate}
\end{definition}

\begin{example}\label{Ex: locally p-divisible groups}
    \begin{enumerate}
        \item Any analytic $p$-divisible group (e.g. a vector bundle) is locally $p$-divisible \cite[§4]{Farg22}.
        \item Let $G$ be a (locally) $p$-divisible rigid group over $K$, then $G_X$ is a (locally) $p$-divisible $X$-group. Examples of locally $p$-divisible rigid groups include $G=\G_a$, $\G_m$ or more generally all commutative connected algebraic groups $G$ over $K$ \cite[Example 6.9]{heuer2023padic}.
        \item If $\Gs$ contains an open subgroup $U\sub \Gs$ with $[p]\colon U \rightarrow U$ étale surjective, then $\Gs$ is locally $p$-divisible \cite[Prop. 3.2.4(3)]{heuer2024curve}. If $X=\Spa(K,\Os_K)$ for algebraically closed $K$, the converse is true \cite[Prop. 6.10]{heuer2023padic}.
        \item Let $X$ be a rigid space and $\As \rightarrow X$ be a relative abeloid variety, i.e. a proper smooth commutative $X$-group with geometrically connected fibres. Then $[p]\colon \As \rightarrow \As$ is surjective, such that $\As$ is locally $p$-divisible, by the previous point.
        \item Let $\Gs \rightarrow X$ denote the analytification of a smooth commutative group scheme $\Gs^{\alg} \rightarrow X^{\alg}$ where $X^{\alg}$ is an algebraic variety over $K$. Then $\Gs$ is locally $p$-divisible \cite[Lemma 3.2.6]{heuer2024curve}.
    \end{enumerate}
\end{example}

We will restrict to the locally $p$-divisible groups satisfying a minor technical assumption.
\begin{definition}\label{Def: admissible locally p-divisible groups}
    Let $\Gs \rightarrow X$ be a locally $p$-divisible $X$-group. 
    \begin{enumerate}
        \item If $X$ is a rigid space, the group $\Gs$ is said to be \emph{admissible} if $\Gs[p^{\infty}]$ is ind-Zariski-constructible. 
        \item If $X$ is a sousperfectoid space, the group $\Gs$ is admissible if it arises as the pullback of an admissible locally $p$-divisible group $\Gs' \rightarrow X'$ along a map $X\rightarrow X'$, where $X'$ is a rigid space.
    \end{enumerate}
\end{definition}
By $\Gs[p^{\infty}]$ being ind-Zariski-constructible, we mean that there exists an ind-Zariski-constructible sheaf $\Fs$
on $X_{\et}$ such that $\Gs[p^{\infty}] = \nu^*\Fs$, where $\nu\colon X_v \rightarrow X_{\et}$ is the natural map of sites. There is no meaningful notion of Zariski-constructible sheaves on general sousperfectoid spaces, which is the reason for this roundabout definition. We will be mainly interested in the case where the base $X$ is a rigid space, but even so, some proofs will require to consider groups $\Gs$ over bases of the shape $X\times_S Y$, where $X\rightarrow S$ is a smooth morphism of rigid spaces and $Y$ is a perfectoid space over $S$.

All locally $p$-divisible groups appearing in Example \ref{Ex: locally p-divisible groups} except $(3)$ are admissible\footnote{We note the following caveat: In the case of rigid groups $G$, admissibility amounts to asking that $G[p^{\infty}]$ is a filtered colimit of finite étale-locally constant sheaves. This is automatic if $K^+=\Os_K$ but not in general.}. One more important example is the following. By $\Sm_{/X, \et}$, we denote the site of all rigid spaces smooth over $X$, endowed with the étale topology.
\begin{lemma}\label{Lemma: units in B is admissible}
    Let $X$ be a smooth rigid space and let $B$ be a finite locally free commutative $\Os_X$-algebra. Then the sheaf
    \[ B^{\times}\colon \Sm_{/X, \et} \rightarrow \Ab\, , \quad (Y\xrightarrow{g }X) \mapsto \Gamma(Y,g^*B)^{\times}\]
    is represented by an admissible locally $p$-divisible $X$-group.
\end{lemma}
\begin{remark}
    This is the prime example of a locally $p$-divisible group $\Gs$ where the torsion subgroups $\Gs[p^m]$ need not be finite étale over $X$, and which motivates the generality in Definition \ref{Def: admissible locally p-divisible groups}.
\end{remark}
\begin{proof}
    Let $\B \rightarrow X$ denote the vector bundle underlying the locally free sheaf $B$. The norm map $\Nr\colon B \rightarrow \Os_X$ induces a map of $X$-spaces $N\colon \B \rightarrow \G_{a,X}$ and $B^{\times}$ is easily seen to be represented by the open subspace $N^{-1}(\G_{m,X})$. Let $f\colon X' \rightarrow X$ denote the finite flat morphism associated with $B$. It is clear that 
    \[ f_{\Et,*}^{\rig}\G_{m,X'} = B^{\times},\]
    where
    \[ f_{\Et}^{\rig}\colon \Sm_{/X',\et} \rightarrow \Sm_{/X,\et} \]
    is the natural map of sites. By Theorem \ref{Thm: properties of zariski-constructible sheaves}(1), $B^{\times}[p^m] = f_{\Et,*}^{\rig}\mu_{p^m}$ is Zariski-constructible. Moreover, $f_{\Et,*}^{\rig}$ is exact, by \cite[Cor. 2.6.6]{huber2013étale}, such that 
    \[ R^1f_{\Et,*}^{\rig}\mu_{p^m}=0.\]
    It follows that $[p]\colon B^{\times} \rightarrow B^{\times}$ is surjective, and it is also étale by Remark \ref{rmk: p map automatically étale}. In particular $B^{\times}$ is locally $p$-divisible, as required.
\end{proof}

We will occasionally use the $p$-adic exponential, which converges on an open subgroup of the Lie algebra and defines a natural section of the logarithm.
\begin{lemma}\label{lemma: exponential converges}
    Let $X$ be a qcqs rigid space over $(K,\Os_K)$ and let $\Gs$ be a commutative smooth $X$-group. Then there exists an open $\Os_X^+$-subvector bundle $\ga^{\circ} \sub \ga$ and an open embedding of $X$-groups
    \[ \exp_{\Gs}  \colon \ga^{\circ} \hookrightarrow \widehat{\Gs} \]
    that is a section of the logarithm $\log_{\Gs}$. The exponential is functorial in $\Gs$ (i.e. the obvious diagram commutes after shrinking $\ga^{\circ}$).
\end{lemma}
\begin{remark}\label{rem: exponential on trivial vector bundles}
    The lemma stays valid if $X$ is any quasi-compact rigid space over $(K,K^+)$ or sousperfectoid space, assuming that the Lie algebra $\ga$ comes from an $\Os_X^+$-vector bundle, e.g. if $\ga$ is trivial.
\end{remark}
\begin{proof}
We start by showing that there exists a finite locally free $\Os_X^+$-submodule $E \sub \Lie(\Gs)$ such that $\Lie(\Gs) = E[\frac{1}{p}]$. For this, we view $X$ as a coherent rigid space in the sense of \cite[Def. 4.1.6]{Abbes2010} 
and thus we find a formal model $\Xs$ of $X$ topologically of finite presentation over $\Spf(\Os_K)$. By \cite[Thm. 4.8.18.(ii)]{Abbes2010}, the $\Os_X$-module $\Lie(\Gs)$ admits a coherent formal model $\Fs$ on $\Xs$. By Raynaud--Gruson's flattening, see \cite[Thm. 5.8.1]{Abbes2010}, up to replacing $\Xs$ by an admissible formal blowup, we may assume that $\Fs$ is locally free. Then $E =\Fs \otimes_{\Os_{\Xs}} \Os_X^+$ is as required. 

Next, we consider the induced open subgroup 
   \[\ha = E\otimes_{\Os_X^+} \G_a^+ \sub \ga.\]
   By \cite[Prop. 3.2.1]{heuer2024curve}, we can find an open cover of $X$ by affinoid opens $X_i \sub X$ and open subgroups $V_i \sub \ga_{X_i}$ such that the logarithm restricts to an isomorphism $\log_{\Gs}^{-1}(V_i) \xrightarrow{\cong} V_i$. By \cite[Lemma 3.1.2]{heuer2024curve}, there exists integers $n_i\geq 0$ such that $p^{n_i}\ha_{X_i} \sub V_i$ for each $i$. As $X$ is quasi-compact, we can replace the open cover $\{X_i\}_i$ by a finite subcover, and we find that $\ga^{\circ} \coloneqq p^n \ha$ is as required, for $n \gg 0$.
\end{proof}

We now fix a non-archimedean field $(K,K^+)$ over $\Q_p$. We recall from \cite[§6]{heuer2023padic} how a choice of exponential for $K$ induces a compatible system of exponentials on $K$-rational points for admissible locally $p$-divisible rigid groups $G$.
\begin{definition}
    An exponential for $K$ is a continuous group morphism
    \[ \Exp \colon K \rightarrow 1+\ma\]
    extending the usual analytic exponential function and splitting the logarithm map $\log\colon 1+\ma \rightarrow K$.
\end{definition}
Exponentials for $K$ exist as soon as $K$ is algebraically closed \cite[Lemma 6.2]{heuer2023padic}.

\begin{proposition}\label{exponential on K points of locally p-divisible group}
    Let $K$ be an algebraically closed complete extension of $\Q_p$ and let $G$ be an admissible locally $p$-divisible rigid group over $K$. Then a choice of exponential $\Exp$ for $K$ induces a continuous section
    \[ \Exp_G\colon \Lie(G) \rightarrow \widehat{G}(K)\]
    to the logarithm map $\log_G$ on $K$-points. For fixed $\Exp$, the map $\Exp_G$ is natural in $G$.
\end{proposition}
\begin{proof}
    Let $G'= G \times_{\Spa(K,K^+)}\Spa(K,\Os_K)$ be the associated locally $p$-divisible group over $(K,\Os_K)$. By \cite[Thm. 6.12]{heuer2023padic}, we get an exponential, functorial in $G'$
    \[ \Exp_{G'}\colon \Lie(G) \rightarrow \widehat{G}'(K,\Os_K) = \widehat{G}(K,\Os_K).\]
    It remains to show that $\widehat{G}(K,\Os_K) = \widehat{G}(K,K^+)$. By the logarithm sequence (\ref{logarithm exact sequence}), it is enough to show that $G[p^{\infty}](K,\Os_K) = G[p^{\infty}](K,K^+)$. This follows from the fact that $G[p^{\infty}]$ is ind-Zariski-constructible and hence overconvergent, by Remark \ref{Remark: zariski constructibility and overconvergence}.
\end{proof}

\section{Hodge--Tate spectral sequences}\label{Section: Hodge--Tate spectral sequences}
With these preparations, we are now ready to define and study the Hodge--Tate spectral sequences with locally $p$-divisible coefficients and their geometric counterparts. 

\subsection{The group theoretic spectral sequence}
We fix a perfectoid field extension $(K,K^+)$ of $\Q_p$. Let $X$ be a smoothoid space over $K$ (cf. §\ref{subsection: Smoothoid spaces}).
\begin{definition}
    Let $\Gs$ be a commutative smooth $X$-group, viewed as a $v$-sheaf on $X$. We define the \textbf{Hodge--Tate spectral sequence with $\Gs$-coefficients} to be the Leray spectral sequence associated to $\nu\colon X_v \rightarrow X_{\et}$
    \begin{align}\label{classical Hodge--Tate spectral sequence for general G} 
    E_2^{ij}(\Gs) = H_{\et}^i(X,R^j\nu_* \Gs) \Rightarrow H_v^{i+j}(X,\Gs).
\end{align}
\end{definition}

\begin{remark}
    When $\Gs=\G_a$ and $X$ is a smooth rigid space, this recovers the Hodge--Tate spectral sequence \cite[Thm. 3.20]{Scholze2013}.
\end{remark}

We have the following result, generalizing the commutative case of \cite[Thm. 4.1]{heuer2022moduli}. We delay the proof to the next section.
\begin{proposition}\label{sheafified correspondence old school}
    Let $X$ be a smoothoid space over $K$ and let $\Gs$ be an admissible locally $p$-divisible $X$-group. Then $\nu_*\Gs = \Gs$ and we have for each $j>0$ a canonical isomorphism of sheaves on $X_{\et}$, natural in $X$ and $\Gs$
    \begin{equation*}
        \HTlog\colon R^j\nu_* \Gs \xrightarrow{\cong} \widetilde{\Omega}_X^j \otimes_{\Os_X} \Lie(\Gs). 
    \end{equation*}
\end{proposition}

Using Proposition \ref{sheafified correspondence old school}, the Hodge--Tate spectral sequence takes the following form
\begin{equation}\label{classical Hodge--Tate sequence for G reformulated}
 E_2^{ij}(\Gs)= \left\{\begin{aligned}
  &H_{\et}^i(X,\widetilde{\Omega}_X^j \otimes_{\Os_X} \Lie(\Gs)) \quad &\text{if }j>0\\
  &H_{\et}^i(X,\Gs) \quad &\text{if }j=0
\end{aligned}\right\}  \Longrightarrow H_{v}^{i+j}(X,\Gs).
\end{equation}

We will prove later in Corollary \ref{Cor: The classical sseq degenerates} that the Hodge--Tate spectral sequence with $\Gs$-coefficients degenerates in the proper smooth case. For now, we treat the case where $\Gs=\ga$ is a vector bundle, in which case it essentially follows from the case $\Gs=\G_a$ and is due to Scholze \cite[Thm. 3.20]{Scholze2013} and Bhatt--Morrow--Scholze \cite[Thm. 13.3.(ii)]{Bhatt2018}. 
\begin{proposition}\label{prop: degeneration of additive Hodge--Tate spectral sequence}
    Let $X$ be a proper smooth rigid space and $E$ a finite locally free $\Os_X$-module. Then the Hodge--Tate spectral sequence
\begin{align} E_2^{ij} = H_{\et}^i(X,\womega_X^j \otimes_{\Os_X} E) \Longrightarrow H_v^{i+j}(X,E)\end{align}
degenerates at $E_2$. If $K$ is algebraically closed, a $\BdR+ /\xi^2$-lift $\X$ of $X$ induces a splitting of the spectral sequence and hence a decomposition, natural in $\X$
\[ H_v^n(X,E) = \bigoplus_{i=0}^{n}H^i(X,\womega_X^{n-i} \otimes_{\Os_X} E).\]
\end{proposition}
\begin{proof}
    Let $C$ be a completed algebraic closure of $K$ and let $C^+$ denote the completion of the integral closure of $K^+$ in $C$. We denote by $X_C$ (resp. $E_C$) the base-change of $X$ (resp. $E$) from $K$ to $C$. By \cite[Cor. 3.10]{heuer2024relative}, we have
    \[ H_v^n(X_C,E_C) = H_v^n(X,E)\otimes_K C\, , \quad H^i(X_C,\womega_{X_C}^j \otimes E_C) = H^i(X,\womega_{X}^j \otimes E)\otimes_K C.\]
    As the degeneration can be deduced from a dimension count, we may assume without loss that $K$ is algebraically closed. By the projection formula, we have
    \[ R\nu_*E = R\nu_*\Os_X \otimes_{\Os_X}E.\]
    Now by \cite[Prop. 7.2.5]{guo2023}, any lift $\X$ of $X$ to $\BdR+/\xi^2$ induces a splitting in the derived category $D(X_{\et})$, natural in $\X$
    \[ R\nu_*\Os_X  = \bigoplus_{i\geq 0} \womega_X^i[-i].\]
    There is thus a trivial Cartan--Eilenberg resolution of the complex $R\nu_*E$ given by
    \[ J^{\bullet,\bullet} = \bigoplus_{i\geq 0} J_i^{\bullet}[-i],\]
    where $J_i$ is an injective resolution of $\womega_X^i \otimes E$ in $D(X_{\et})$. It follows that the Hodge--Tate spectral sequence for $E$ decomposes as the direct sum of the spectral sequence associated to the $1$-column bicomplexes $J_i^{\bullet}$, and this decomposition is natural in $\X$. The result follows. 
\end{proof}

\begin{remark}
    For $E=\Os_X$ and algebraically closed $K$, this recovers the Hodge--Tate decomposition
\begin{align}\label{eq: the classical Hodge--Tate decomposition}
    H_{\et}^n(X,\Q_p)\otimes_{\Q_p} K = \bigoplus_{i+j=n} H^i(X,\Omega_{X/K}^j)(-j)
\end{align}
via Scholze's Primitive Comparison Theorem \cite[Thm. 5.1]{scholze2013padicHodge}.
\end{remark}

\subsection{Rigid approximation}\label{subsection: Rigid approximation}
In this technical section, we prove Proposition \ref{sheafified correspondence old school}. We will need a technique of approximation by rigid spaces, developed in \cite{heuer2022gtorsors}, that we improve upon. The results of this section will also be used extensively later on.

We first recall the notions of tilde-limits of adic spaces.

\begin{definition}{(\cite[(2.4.1)]{huber2013étale}, \cite[§2.4]{scholze2013moduli})}\label{tilde-limit}
    Let $(Y_i)_{i \in I}$ denote an inverse system of adic spaces with quasi-compact and quasi-separated (in short, qcqs) transition maps and let $Y$ be an adic space together with maps $f_i\colon Y\rightarrow Y_i$ compatible with the transition maps. We say that $Y$ is a tilde-limit of the inverse system $(Y_i)_i$ and we write 
    \[ Y\sim \varprojlim_i Y_i\]
    if the following conditions are satisfied:
    \begin{enumerate}
        \item The induced map $\vert Y \vert \rightarrow \varprojlim_i \vert Y_i \vert$ is a homeomorphism, and
        \item There exists an open cover of $Y$ by affinoids $U$ such that the map $\varinjlim \Os(U_i) \rightarrow \Os(U)$ has dense image, where the direct limit runs over all $i\in I$ and all affinoid opens $U_i \sub Y_i$ through which the map $U \rightarrow Y_i$ factors.
    \end{enumerate}
Assume moreover that all $Y_i$ and $Y$ are affinoid. We write 
    \[Y \approx \varprojlim_i Y_i \]
    if, in addition, already $\varinjlim_i \Os(Y_i) \rightarrow \Os(Y)$ has dense image.
\end{definition}

\begin{lemma}\label{diamond and étale site of strong tilde timit}
    Let $Y \sim \varprojlim_i Y_i$ and assume that all $Y_i$ and $Y$ are qcqs. 
    \begin{enumerate}
        \item  On the level of diamonds, we have 
        \[ Y^{\diamondsuit} = \varprojlim_i Y_i^{\diamondsuit}.\] 
        \item For the associated qcqs étale sites, we have 
         \[ Y_{\et, \qcqs} \cong \tworlim_i Y_{i,\et, \qcqs}. \]
    \end{enumerate}
\end{lemma}
\begin{proof}
    For the first point, see \cite[Prop. 2.4.5]{scholze2013moduli}. The second point then follows from \cite[Prop. 11.23]{scholze2022etale}.
\end{proof}

\begin{definition}{(\cite[Def. 2.13]{heuer2022gtorsors})}\label{approximation property}
    Let $X$ be any good adic space over $\Q_p$. A sheaf of sets $\Fs$ on $\Adic_{X,\et}$ is said to satisfy the approximation property if, for any affinoid perfectoid tilde-limit $Y \approx \varprojlim_i Y_i$ of good affinoid adic spaces $Y_i$ over $X$, we have
    \[ \Fs(Y) = \varinjlim_i \Fs(Y_i).\]
\end{definition}

\begin{remark}
    A similar condition recently appeared in \cite{scholze2024berkovichmotives}, where such sheaves are called finitary.
\end{remark}

\begin{example}\label{Ex: sheaves with the apporixation property}
    \begin{enumerate}
        \item Let $\Fs$ be a sheaf on $\Adic_{X,\et}$ that arises as the pullback of a sheaf on $X_{\et}$ along the natural map of sites $\Adic_{X,\et} \rightarrow X_{\et}$. Then $\Fs$ has the approximation property, by adapting the proof of \cite[Prop. 14.9]{scholze2022etale}.
        \item Let $Y \rightarrow X$ be an étale morphism, then $Y$ viewed as a sheaf on $\Adic_{X,\et}$ has the approximation property. This is a special case of the first point.
        \item Let $G$ be a (not necessarily commutative) rigid group over a non-archimedean field extension $K$ of $\Q_p$, and let $U \sub G$ be any open subgroup. Then the sheaf of cosets $G/U$ on $\Adic_{K,\et}$ has the approximation property, by \cite[Prop. 4.1]{heuer2022gtorsors}. We generalize this in Proposition \ref{Prop: approximating sheaves are small v-sheaves} below.
    \end{enumerate}
\end{example}

The decisive property satisfied by sheaves with the approximating property is the following.
\begin{proposition}\label{Prop: main known results about approximating sheaves}
Let $X$ be a good adic space over $\Q_p$ and let $\Fs$ be a sheaf of sets on $\Adic_{X,\et}$ with the approximation property. 
\begin{enumerate}
    \item $\Fs$ is a $v$-sheaf. In particular, $\Fs$ extends uniquely to a sheaf on $X_v$.
    \item For any inverse limit of spatial diamonds $Y = \varprojlim_i Y_i$ with qcqs transition maps, we have
    \[ \Fs(Y) = \varinjlim_i \Fs(Y_i).\]
    \item Let $\nu \colon X_v \rightarrow X_{\et}$ denote the natural map of sites. If $\Fs$ is a sheaf of groups (resp. of abelian groups), then $R^n\nu_*\Fs = 0$ for $n=1$ (resp. for all $n\geq 1$). In particular,
    \[ H_{\et}^n(X,\Fs) = H_{v}^n(X,\Fs),\]
for $n=0,1$ (resp. for all $n\geq 0$).
\end{enumerate}
\end{proposition}
\begin{proof}
    This is proven in \cite[Prop. 2.14, Lemma 2.20]{heuer2022gtorsors} for $X=\Spa(K,K^+)$, but the proof extends without change to the general case.
\end{proof}

Our main new input is the following characterization of sheaves with the approximation property.
\begin{proposition}\label{Prop: approximating sheaves are small v-sheaves}
    Let $X$ be a good adic space over $\Q_p$ and let $\Fs$ be a sheaf on $\Adic_{X,\et}$. Then the following are equivalent:
    \begin{enumerate}
        \item $\Fs$ satisfies the approximation property, and
        \item $\Fs$ admits a presentation as an étale quotient sheaf
        \[ \Fs = V/R,\]
        where $V$ is an adic space smooth over $X$ and $R\sub V\times_X V$ is an open equivalence relation.
    \end{enumerate}
In that case, the following additional statements hold: For any map $V \rightarrow \Fs$ from a smooth $X$-space $V$, the diamond $R = V \times_{\Fs} V$ is open in $V\times_X V$.
\end{proposition}

With this result in hand, we can prove Proposition \ref{sheafified correspondence old school}.

\begin{proof}[Proof of Proposition \ref{sheafified correspondence old school}]
    By Proposition \ref{Prop: approximating sheaves are small v-sheaves}, the sheaf $\cj{\Gs} = \Gs/\widehat{\Gs}$ has the approximation property. It follows from Proposition \ref{Prop: main known results about approximating sheaves}(3) that $R\nu_*\cj{\Gs} =\cj{\Gs}$, such that
    \[ R^j\nu_*\Gs = R^j\nu_*\widehat{\Gs}\, , \quad \fa j>0.\]
    Moreover, recall that the logarithm defines an exact sequence
\[\begin{tikzcd}
	0 & {\Gs[p^{\infty}]} & {\widehat{\Gs}} & \ga & {0.}
	\arrow[from=1-1, to=1-2]
	\arrow[from=1-2, to=1-3]
	\arrow["{\log_{\Gs}}", from=1-3, to=1-4]
	\arrow[from=1-4, to=1-5]
\end{tikzcd}\]
The étale $X$-group $\Gs[p^{\infty}]$ also satisfies the approximation property, by Example \ref{Ex: sheaves with the apporixation property}(2). Hence we also have $R\nu_*\Gs[p^{\infty}] =\Gs[p^{\infty}]$ and
    \[ R^j\nu_*\widehat{\Gs}= R^j\nu_*\ga\, , \quad \fa j>0.\]
Finally, we have
\[ R^j\nu_*\ga = R^j\nu_*\G_a \otimes \ga = \womega_X^j \otimes \ga\, , \quad \fa j\geq 0,\]
where we use the projection formula for the first equality.
\end{proof}

\begin{proof}[Proof of Proposition \ref{Prop: approximating sheaves are small v-sheaves}]
\begin{itemize}
\item \underline{$(1) \Rightarrow (2)\colon$}

Let $\Fs$ have the approximation property. We may assume that $X$ is affinoid. By arguing as in \cite[Prop. III.1.3]{fargues2024geometrization}, we find that $\Fs$ is a small $v$-sheaf. Let $T \rightarrow \Fs$ be a surjection from a perfectoid space $T$ and fix an open cover of $T$ by affinoid perfectoids $T_i$. We may write by \cite[Prop. 3.17]{heuer2023diamantine} 
\[ T_i \approx \varprojlim_j U_{ij} \rightarrow X,\]
where each $U_{ij}$ is affinoid smooth over $X$. Applying the approximation property, we find that $T_i \rightarrow \Fs$ factors through $V_i \coloneqq U_{i,j(i)} \rightarrow \Fs$ for some $j=j(i)$. Then we obtain a surjection of $v$-sheaves $q\colon V = \coprod_{i} V_i \rightarrow \Fs$ from a smooth $X$-space. It remains to show that $R = V\times_{\Fs} V \sub V\times_X V$ is an open subspace, for any such map $q$. Let $y\in R$ with its associated geometric point 
$\cj{y}\colon \Spa(C,C^+) \rightarrow V\times_X V$. We may write 
\[ \Spa(C,C^+) \approx \varprojlim_k W_k,\]
where the limit ranges over all affinoid spaces $W_k$ étale over $V\times_X V$ together with a factorization
\[\begin{tikzcd}
	{\Spa(C,C^+)} & {W_k} \\
	& {V \times_X V.}
	\arrow["{\cj{y_k}}", from=1-1, to=1-2]
	\arrow["{\cj{y}}"', from=1-1, to=2-2]
	\arrow[from=1-2, to=2-2]
\end{tikzcd}\]
Then the two following compositions
\[\begin{tikzcd}
	&& V \\
	{\Spa(C,C^+)} & {V \times_X V} && \Fs \\
	&& V
	\arrow["q", from=1-3, to=2-4]
	\arrow["{\cj{y}}", from=2-1, to=2-2]
	\arrow[from=2-2, to=1-3]
	\arrow[from=2-2, to=3-3]
	\arrow["q"', from=3-3, to=2-4]
\end{tikzcd}\]
are equal. By the approximation property, we deduce that already the compositions
\[\begin{tikzcd}
	&& V \\
	{W_k} & {V \times_X V} && \Fs \\
	&& V
	\arrow["q", from=1-3, to=2-4]
	\arrow[from=2-1, to=2-2]
	\arrow[from=2-2, to=1-3]
	\arrow[from=2-2, to=3-3]
	\arrow["q"', from=3-3, to=2-4]
\end{tikzcd}\]
agree for $k\gg 0$. This shows that the map $W_k \rightarrow V \times_X V$ factors through $R \sub V \times_X V$ and its image is an open neighborhood of $y \in V \times_X V$ contained in $R$. We deduce that $R$ is open in $V \times_X V$, as required. This shows the first implication and the last part of the statement.
    \item \underline{$(2) \Rightarrow (1)\colon$}

Fix an affinoid perfectoid tilde-limit $Y\approx \varprojlim_i Y_i$ of good affinoid adic spaces. We first show that the natural map
\[ \phi\colon \varinjlim_i \Fs(Y_i) \rightarrow \Fs(Y).\]
is injective. Let $i_0\in I$ and $\cj{s},\cj{t}\in \Fs(Y_{i_0})$ be such that $\cj{s}\restr{Y}=\cj{t}\restr{Y}$. We need to show that already $\cj{s}\restr{Y_i}=\cj{t}\restr{Y_i}$ for $i\gg i_0$. This may be checked after taking a standard-étale cover of $Y_{i_0}$ (replacing $Y_i$ and $Y$ by their respective pullbacks along this cover). Hence we may assume that the given sections admit lifts $s,t\in V(Y_{i_0})$. By assumption, the following diagram commutes
\[\begin{tikzcd}
	&& V \\
	Y & {Y_{i_0}} && {\Fs,} \\
	&& V
	\arrow[from=1-3, to=2-4]
	\arrow[from=2-1, to=2-2]
	\arrow["s", from=2-2, to=1-3]
	\arrow["t"', from=2-2, to=3-3]
	\arrow[from=3-3, to=2-4]
\end{tikzcd}\]
so that $(s,t)\restr{Y}$ has image in $R \sub V\times_X V$. By Lemma \ref{lemma: maps from inverse limit eventually factor through open subspace} below, we deduce that already $(s,t)\restr{Y_i}$ maps in $R$, for $i\gg i_0$, such that $\cj{s}\restr{Y_i} = \cj{t}\restr{Y_i}$. This shows injectivity.

We now show that $\phi$ is surjective. Let $\cj{s} \in \Fs(Y)$. Assume that $\cj{s}$ is in the image of $\varphi$ after replacing $Y$ by an étale cover, then we can conclude. Indeed, let $Y' \rightarrow Y$ be a quasi-compact étale cover, corresponding to an étale cover $Y_{i_0}' \rightarrow Y_{i_0}$ for some $i_0\in I$, and let $\cj{s_i'} \in \Fs(Y_i')$ be a section for $i\geq i_0$ mapping to $\cj{s}\restr{Y'}$. We claim that the section $\cj{s_i'}$ descends to a section $\cj{s_i} \in \Fs(Y_i)$  (which is then automatically mapped to $\cj{s}$), up to increasing $i$: For this, it is enough to show that $p_1^*\cj{s_i'} = p_2^*\cj{s_i'} \in \Fs(Y_i' \times_{Y_i} Y_i')$: This holds after pullback to $Y' \times_Y Y'$, hence it holds after increasing $i$, by the injectivity step.

Next we reduce to the case where $Y=\Spa(C,C^+)$ for a complete algebraically closed field $C$. Let $y\in \vert Y \vert$ and let $\{ (U_j,u_j) \}_j$ denote the cofiltered inverse system of quasi-compact étale maps $U_j \rightarrow Y$ together with a lift $u_j \in \vert U_j \vert$ of $y$. Then the corresponding geometric point $\Spa(C(y),C(y)^+) \rightarrow Y$ satisfies $\Spa(C(y),C(y)^+) \approx \varprojlim_j U_j$. Consider the cofiltered category consisting of triplets $(i,j,U_{ij}\rightarrow Y_i)$, where $U_{ij}\rightarrow Y_i$ is an étale map with $U_j = U_{ij}\times_{Y_i}Y$.  
Then it follows from \cite[Lemma 3.13]{heuer2023diamantine} (or alternatively Lemma \ref{stability property of strong tilde-limit bis}(2) below) that
\[ U_j \approx \varprojlim_{i} U_{ij}. \]
By sequential approximation, we deduce that also
\[ \Spa(C(y),C(y)^+) \approx \varprojlim_{i,j} U_{ij}.\]
Suppose that $\cj{s}(y)\in \Fs(C(y),C(y)^+)$ coincides with $\cj{s}_{ij}(y)$ for some $\cj{s}_{ij}\in \Fs(U_{ij})$. Then $\cj{s}\restr{U_j}$ and $\cj{s}_{ij}\restr{U_j}$ map to the same element of $\Fs(C(y),C(y)^+)$. By the injectivity step, these two sections are equal, up to increasing $j$. Hence, if surjectivity of $\phi$ holds for points, for each $y\in \vert Y \vert$, the restriction of $\cj{s}$ to some étale neighborhood of $y$ is in the image of $\phi$, which concludes, by the previous reduction step.

From these two reduction steps, we can without loss assume that $Y=\Spa(C,C^+)$, $\cj{s}$ admits a lift $s\in V(Y)$, $X$ is affinoid and $V$ is standard-smooth over $X$. Hence, we may write $V\rightarrow X$ as a composition
\[ V \xrightarrow{g} B \rightarrow X,\]
where $g$ is standard-étale, $B=\B_X^n$ and the last map is the canonical projection. We write 
\[ \Os(B) = \Os(X)\langle T_1, \ldots, T_n\rangle.\] 
Here, $T_1,\ldots T_n$ are Tate variables and we abbreviate these by $\underline{T}$. Let 
\[ \underline{t} = g(s)\in B(Y) \cong (C^+)^{\oplus n}.\]
By definition of tilde-limits
\[ \theta\colon \varinjlim_i \Os^+(Y_i) \rightarrow \Os^+(Y)=C^+ \]
has dense image. Hence we may find $\underline{t_i} \in \Os^+(Y_i)^{\oplus n}$ such that $\theta(\underline{t_i}) \cong \underline{t} \, (\, \modulo \varpi^N)$, where $\varpi \in K^+$ is a pseudo-uniformizer and $N\gg 0$ is a large integer. We write $\underline{t'} = \theta(\underline{t_i}) \in (C^+)^{\oplus n}$.

It remains to construct a lift $s' \in V(Y)$ of $\underline{t'} \in B(Y)$ that is $R$-equivalent to the original section $s \in V(Y)$. Note that $s'$ will automatically live in finite level, i.e. arise from a section $s_i \in V(Y_i)$ that lifts $\underline{t_i} \in B(Y_i)$, up to increasing $i$. This follows from the fact that $V \rightarrow B$ is étale, such that by \cite[Prop. 14.9]{scholze2022etale},
\[ \Hom_B(Y,V) = \varinjlim_i \Hom_B(Y_i,V).\]
 
Let $U \sub R \sub V \times_X V$ be a rational open subset containing the image of $(s,s)\colon Y \rightarrow R$. We form the following pullback diagram
\[\begin{tikzcd}
	{Y'} & U \\
	& {V\times_X V} \\
	Y & {V\times_X B.}
	\arrow[from=1-1, to=1-2]
	\arrow[from=1-1, to=3-1]
	\arrow["\lrcorner"{anchor=center, pos=0.125}, draw=none, from=1-1, to=3-2]
	\arrow[hook, from=1-2, to=2-2]
	\arrow["{\id \times g}", from=2-2, to=3-2]
	\arrow["{(s,s')}", dashed, from=3-1, to=1-2]
	\arrow["{(s,\underline{t'})}"', from=3-1, to=3-2]
\end{tikzcd}\]
A lift $s'$ as above corresponds to a diagonal dotted arrow in the above diagram and thus a section to the étale map $Y' \rightarrow Y$. Let $ W \sub (\id\times g)(U)$ be a rational open neighborhood of the image of $(s,\underline{t})\colon Y \rightarrow V \times_X B$ in the open image $(\id\times g)(U) \sub V\times_X B$. We write $W=U(\tfrac{F_1,\ldots,F_r}{G})$ for $F_j,G\in \Os^+(V\times_{X}B)$ and we let $f_j,g$ (resp. $f_j', g'$) denote their pullbacks in $\Os^+(Y)=C^+$ under the map $(s,\underline{t})$ (resp. the map $(s,\underline{t'})$). By assumption, $(s,\underline{t})(Y) \sub W$ such that
\[ \vert f_j\vert \leq \vert g \vert \neq 0, \,\, \fa j=1, \ldots r,\]
where $\vert \cdot \vert$ denotes the valuation on $C$. By construction of $\underline{t'}$, we have that $f_j\equiv f_j' \,(\, \modulo \varpi^N)$ and $g \equiv g'  \,(\, \modulo \varpi^N)$. Hence, as soon as $N$ is big enough,
\[ \vert f_j'\vert = \vert f_j\vert  \leq \vert g \vert= \vert g' \vert, \, \,\fa j=1, \ldots r.\]
It follows that also $(s,\underline{t'})\colon Y \rightarrow V\times_X B$ has image in the open subspace $W \sub V\times_X B$. After passing to pullbacks, we deduce that the étale map $Y' \rightarrow Y$ is surjective. Since $Y=\Spa(C,C^+)$ is strictly totally disconnected, the map $Y' \rightarrow Y$ thus admits a section, as required. This concludes the proof.
\end{itemize}
\end{proof}

The following lemma was used in the proof above.
\begin{lemma}\label{lemma: maps from inverse limit eventually factor through open subspace}
    Let $Z= \varprojlim_{i \in I} Z_i$ be an inverse limit of spatial diamonds with qcqs transition maps. Let $T$ be a locally spatial diamond and let $U\sub T$ be an open subdiamond. Let $t\in T(Z_{i_0})$ for some $i_0 \in I$ and assume that the composition $Z\rightarrow Z_{i_0} \xrightarrow{t} T$ factors through $U \sub T$. Then already the composition $Z_i \rightarrow Z_{i_0} \xrightarrow{t} T$ has image in $U$ for $i\gg i_0$. 
\end{lemma}
\begin{proof}
    Let us take an open cover $T=\bigcup_j T_j$ by open spatial subdiamonds and $Z_{i_0} = \bigcup_k Z_{i_0,k}$ an open cover by spatial subdiamonds subdividing the preimage of the $T_j$. By quasi-compactness of $Z_{i_0}$, we may assume that there are finitely many $k$. Then letting $Z_{i,k},Z_k$ denote the inverse image of $Z_{i_0,k}$ in $Z_i, Z$, we obtain again inverse limits $Z_k = \varprojlim_{i \geq i_0} Z_{i,k}$. Replacing $Z$ by $Z_k$, we may assume that $T$ is spatial. Let us take an open cover $U= \bigcup_{j} U_{j}$ by spatial subdiamonds. Again by quasi-compactness, finitely many $U_j$ cover the image of $Z \rightarrow U$. Hence we may assume that $U$ is quasi-compact and hence spatial. For $i\geq i_0$, the preimage $V_i \sub Z_i$ of $U$ under the composition $Z_i \rightarrow Z_{i_0} \xrightarrow{t} T$ are quasi-compact and $\varprojlim_{i\geq i_0} V_i = Z$. By \cite[\href{https://stacks.math.columbia.edu/tag/0A2W}{Tag 0A2W}]{stacks-project}, it follows that $V_i=Z_i$ for $i\gg i_0$, as required.
\end{proof}

\subsection{The diamantine higher direct images}\label{subsection: Definition of diamantine HDI}
In this section, we define the diamantine higher direct images and study their basic properties. Our definitions and choices for the presentation are heavily influenced by \cite{heuer2023diamantine}.\\ 

Fix a non-archimedean field $(K,K^+)$ over $\Q_p$. Let $\pi\colon X \rightarrow S$ be a smooth morphism of seminormal rigid spaces over $K$. It induces maps of sites
\[ \pi_{\Et}^{\diamondsuit}\colon \LSD_{X,\et} \rightarrow \Perf_{S,\et}\, , \quad \pi_{v}^{\diamondsuit}\colon \LSD_{X,v} \rightarrow \Perf_{S,v}.\]
\begin{definition}
Let $\Gs$ be a commutative smooth $X$-group and let $\tau$ be either the étale or the $v$-topology. For $n\geq 0$, the $n$-th \textbf{diamantine higher direct image of} $\Gs$ is defined as the abelian sheaf
\[ \BBun_{\Gs,X/S,\tau}^{n,\diamondsuit} \coloneqq R^n\pi_{\tau,*}^{\diamondsuit}\Gs \colon \Perf_{S} \rightarrow \Ab.\]
We will sometimes simplify the notation and write instead $\BBun_{\Gs,\tau}^{n,\diamondsuit}$ or $\BBun_{\Gs,\tau}^{n}$ when the context is clear.
\end{definition}

We also have the following variant for the étale topology. Recall that in Lemma \ref{Lemma: units in B is admissible}, we introduced the site $\Sm_{/S,\et}$ with underlying category all rigid spaces smooth over $S$, equipped with the étale topology. We have a morphism of sites 
\[ \pi_{\Et}^{\rig}\colon \Sm_{/X,\et} \rightarrow \Sm_{/S,\et}.\]
\begin{definition}
    Let $n\geq 0$, we define
\[ \BBun_{\Gs,X/S,\et}^{n, \rig} \coloneqq R^n\pi_{\Et,*}^{\rig}\Gs \colon \Sm_{/S,\et} \rightarrow \Ab.\]
\end{definition}

We may compare the two variants. We denote by
\[ \delta \colon \LSD_{S,\et} \rightarrow \Sm_{/S,\et}\]
the morphism of sites induced by the functor 
\[ \Sm_{/S} \rightarrow \LSD_{S}\, , \quad Y \mapsto Y^{\diamondsuit}.\]
By Proposition \ref{Prop: good spaces closed under smooth maps}, since $S$ is a seminormal rigid space, this functor is fully faithful. We also have a morphism of sites
\[ \iota\colon \LSD_{S,\et} \rightarrow \Perf_{S,\et}.\]
For any sheaf $\Fs$ on $\Sm_{/S,\et}$, we define its diamantification to be
\[ \Fs^{\diamondsuit} \coloneqq \iota_* \delta^* \Fs \colon \Perf_{S,\et} \rightarrow \Ab.\]
This is consistent, since if $\Fs$ is representable by a smooth $S$-space $V$, then $\Fs^{\diamondsuit}$ coincides with $V^{\diamondsuit}$. 
We have a commutative diagram
\[\begin{tikzcd}
	{\LSD_{X,\et}} && {\Sm_{/X,\et}} \\
	{\LSD_{S,\et}} && {\Sm_{/S,\et}.}
	\arrow["\delta", from=1-1, to=1-3]
	\arrow[from=1-1, to=2-1]
	\arrow["{\pi_{\Et}^{\rig}}", from=1-3, to=2-3]
	\arrow["\delta", from=2-1, to=2-3]
\end{tikzcd}\]
It induces base-change morphisms, for any abelian sheaf $\Fs$ on $\LSD_{X,\et}$
\begin{align}\label{diamantification and higher pushforward}
    (R^n\pi_{\Et,*}^{\rig}\Fs)^{\diamondsuit} \rightarrow R^n\pi_{\Et,*}^{\diamondsuit}\Fs.
\end{align}

We now summarize the main properties satisfied by the diamantine higher direct images. This generalizes \cite[Thm. 1.1]{heuer2023diamantine}, which is the case $\Gs=\G_m$ and $n=1$.
\begin{proposition}\label{Prop: main properties of diamantine higher direct images}
    Let $\pi\colon X \rightarrow S$ be a proper smooth morphism of seminormal rigid spaces and let $\Gs$ be an admissible locally $p$-divisible $X$-group. Assume either that $S=\Spa(K,K^+)$ or that $\ga$ arises via pullback from a vector bundle on $S$. 
    \begin{enumerate}
        \item For $\tau \in \{\et,v\}$, the sheaf $\BBun_{\ga,\tau}^{n,\diamondsuit}$ is a $\tau$-vector bundle on $S$. If $S=\Spa(K,K^+)$ and $K$ is perfectoid, we further have
        \[ \BBun_{\ga,\tau}^{n,\diamondsuit} = H_{\tau}^n(X,\ga)\otimes \G_a.\]
        \item For each $1 \leq m \leq \infty$, we have
        \[ \BBun_{\Gs[p^m],v}^{n,\diamondsuit} = \BBun_{\Gs[p^m],\et}^{n,\diamondsuit} = \nu^* R^n\pi_{\et,*}\Gs[p^m]. \]
        and this sheaf is ind-Zariski-constructible. If $S=\Spa(K,K^+)$ or $[p]\colon \widehat{\Gs} \rightarrow \widehat{\Gs}$ is finite étale surjective, then $\BBun_{\Gs[p^m],\et}^{n,\diamondsuit}$ is representable by an étale group over $S$.
        \item For $\tau \in \{\et,v\}$, we have a short exact sequence of sheaves on $\Perf_{S,\tau}$
\begin{equation}\label{geometric cohomological log exact sequence}\begin{tikzcd}
	0 & {\BBun_{\Gs[p^{\infty}],\tau}^{n,\diamondsuit}} & {\BBun_{\widehat{\Gs},\tau}^{n,\diamondsuit}} & {\BBun_{\ga,\tau}^{n,\diamondsuit}} & {0.}
	\arrow[from=1-1, to=1-2]
	\arrow[from=1-2, to=1-3]
	\arrow["{\log_*}", from=1-3, to=1-4]
	\arrow[from=1-4, to=1-5]
\end{tikzcd}\end{equation}
    
        \item The natural map (\ref{diamantification and higher pushforward}) yields an isomorphism
    \[ (\BBun_{\Gs,\et}^{n,\rig})^{\diamondsuit} \xrightarrow{\cong} \BBun_{\Gs,\et}^{n,\diamondsuit}.\]
    In particular, $\BBun_{\Gs,\et}^{n,\rig}$ is representable by a smooth $S$-group $\Hs$ if and only if $\BBun_{\Gs,\et}^{n,\diamondsuit}$ coincides with $\Hs^{\diamondsuit}$.
    \end{enumerate}
\end{proposition}
\begin{remark}
    Without further assumptions in $(2)$, the sheaf $\BBun_{\Gs[p^{m}],\et}^{n}$ need a priori not be representable by an étale $S$-group. In general, one can show that a Zariski-constructible sheaf $\Fs$ on $S_{\et}$ is represented by a diamond, using the equivalence of categories between constructible sheaves over a qcqs algebraic variety $Y$ and algebraic spaces qcqs over $Y$ \cite[Exposé IX, Prop. 2.7]{SGA4}. However, this diamond need not be locally separated over $S$ and thus may not come from an adic space, as can be seen by taking $\Fs$ to be the skyscraper sheaf at the origin of the closed unit ball $S=\B_K^1$.
\end{remark}

The main ingredient for the proof of Proposition \ref{Prop: main properties of diamantine higher direct images} is the following result, generalizing \cite[Prop. 3.2]{heuer2023diamantine}.
\begin{proposition}\label{properties of approximating sheaves}
    Let $\pi\colon X \rightarrow S$ be a qcqs smooth map of good adic spaces over $\Q_p$. Let $\Fs$ be a sheaf of sets (resp. of groups, resp. of abelian groups) on $\Adic_{X,\et}$ with the approximation property. Let $Y\approx \varprojlim_i Y_i$ be a tilde-limit of good affinoid adic spaces over $S$. Let $i_0 \in I$ and let $U_{i_0} \rightarrow X \times_S Y_{i_0}$ be a qcqs étale map. For all $i\geq i_0$, form the products $U_i = U_{i_0} \times_{Y_{i_0}}Y_i$ and $U = U_{i_0}\times_{Y_{i_0}}Y$. Then $U \sim \underset{i\geq i_0}{\varprojlim}\, U_i$ and
        \[ H_{\et}^n(U,\Fs) = \underset{i\geq i_0}{\varinjlim}\,  H_{\et}^n(U_i,\Fs),\]
    for $n=0$ (resp. for $n=0,1$, resp. for any $n\geq 0$).
\end{proposition}

We will need the following lemma, which says that tilde-limits behave well under certain pullbacks.
\begin{lemma}\label{stability property of strong tilde-limit bis}
Let $S$ be a good affinoid adic space over $\Q_p$ and let $Y \approx \varprojlim_{i\in I} Y_i$ be a tilde-limit of good affinoid adic spaces over $S$.
    \begin{enumerate}
        \item Let $X \rightarrow S$ be a standard-smooth morphism from a good affinoid adic space $X$. Then we have a tilde-limit
    \[ X\times_S Y \approx \varprojlim_i X\times_S Y_i.\]
    \item Let $U \rightarrow Y$ be a standard-étale morphism, arising via pullback from $U_{i_0} \rightarrow Y_{i_0}$ for some $i_0\in I $. Then if $U_i$ denotes the pullback of $U_{i_0}$ along $Y_i \rightarrow Y_{i_0}$ for $i\geq i_0$, we have
    \[ U \approx \varprojlim_{i\geq i_0} U_i.\]
    \end{enumerate}
\end{lemma}
\begin{proof}
The second point is \cite[Lemma 3.13]{heuer2023diamantine}. In our case, it also follows from the first point if we redefine $S\coloneqq Y_{i_0}$ and $X\coloneqq U_{i_0}$. Let us now prove the first point. We claim that under our assumptions, the fiber products of diamonds $X\times_S Y_i$ are representable by good affinoid adic spaces, namely the adic spectrum of
\[ (\Os(X),\Os^+(X))\,  \underset{(\Os(S),\Os^+(S))}{\widehat{\otimes}}(\Os(Y_i),\Os^+(Y_i)),\]
and similarly for $Y$. By assumption, the map $\Os(S) \rightarrow \Os(X)$ is a composition of rational open immersions, finite étale maps and adjunction of Tate variables. Then so is $\Os(Y_i) \rightarrow \Os(X)\, \widehat{\otimes}_{\Os(S)} \,\Os(Y_i)$. It follows from Proposition \ref{Prop: good spaces closed under smooth maps} that the completed tensor product of Huber pairs above yields a good adic space, which then represents the diamond $X\times_S Y_i$.

Next, we want to show that the image of the map
\[ \varinjlim_i \Os(X)\, \widehat{\otimes}_{\Os(S)} \,\Os(Y_i) \rightarrow \Os(X) \,\widehat{\otimes}_{\Os(S)}\, \Os(Y)\]
is dense in its target. It is enough to show that the image of
\[ \varinjlim_i \Os(X) \otimes_{\Os(S)} \Os(Y_i) = \Os(X)\otimes_{\Os(S)}\varinjlim_i \Os(Y_i)\]
is dense. This image is immediately seen to be dense in $\Os(X)\otimes_{\Os(S)} \Os(Y)$, which concludes.

It remains to show the condition on topological spaces. The diamond functor commutes with fiber products and, on the level of sheaves, fiber products commute with inverse limits. It follows that
\[ (X\times_S Y)^{\diamondsuit} = \varprojlim_i (X\times_S Y_i)^{\diamondsuit}.\]
The topological spaces of an analytic adic space and its diamondification are identified \cite[Lemma 15.6]{scholze2022etale}. On the level of diamonds, taking underlying topological spaces commutes with the inverse limit \cite[Lemma 11.22]{scholze2022etale}, which concludes.
\end{proof}

\begin{proof}[Proof of Proposition \ref{properties of approximating sheaves}]
We start with the case $n=0$.
\begin{claim}
    In the situation of Proposition \ref{properties of approximating sheaves}, assume furthermore that $U_{i_0}$ is standard-étale over $X'\times_{S'} Y_{i_0}$ where $X' \sub X$ and $S'\sub S$ are affinoid open subspaces such that all $Y,Y_i$ are defined over $S'$ and the restriction $\pi\colon X' \rightarrow S'$ is standard-smooth. Then we have an isomorphism
        \[ \underset{i\geq i_0}{\varinjlim}\, \Fs(U_i) \xrightarrow{\cong} \Fs(U).\]
\end{claim}
\begin{proof} By Lemma \ref{stability property of strong tilde-limit bis}(1)-(2), we find that
    \[ U \approx \varprojlim_{i\geq i_0} U_i.\]
    The claim now follows from Proposition \ref{Prop: main known results about approximating sheaves}(2).
\end{proof}
    Since standard-étale maps form a basis of the étale site of $X\times_S Y$, it follows that we have an isomorphism of sheaves
    \[ \Fs = \varinjlim_{i\geq i_0} q_i^{*}q_{i,*}\Fs\]
     on the site $(X\times_S Y)_{\et,\qcqs} \cong \tworlim_i (X\times_S Y_i)_{\et,\qcqs}$, where $q_i\colon X\times_S Y \rightarrow X\times_S Y_i$ denotes the projection and we denote the restriction of $\Fs$ to the site $(X \times_S Y)_{\et,\qcqs}$ by the same symbol. The statement now follows from \cite[\href{https://stacks.math.columbia.edu/tag/09YP}{Tag 09YP}]{stacks-project} in the abelian case, and the non-abelian case is deduced similarly. 
\end{proof}

\begin{cor}\label{higher pushforward of approximating sheaves are again approximating}
    Let $X \rightarrow S$ be a smooth qcqs morphism of good adic spaces over $\Q_p$. Let $\Fs$ be a sheaf of sets (resp. groups, resp. abelian groups) on $\Adic_{X,\et}$ with the approximation property and let $n=0$ (resp. $n=0,1$, resp. $n\geq 0$).
    \begin{enumerate}
        \item The sheaf 
    \[ R^n\pi_{\Et,*}\Fs \colon \Adic_{S,\et} \rightarrow \Ab \]
    has the approximation property.
    \item We have
    \[ R^n\pi_{\Et,*}\Fs \cong R^n\pi_{v,*}\Fs. \]
    \item The natural map (\ref{diamantification and higher pushforward}) is an isomorphism
    \[(R^n\pi_{\Et,*}^{\rig}\Fs)^{\diamondsuit} \xrightarrow{\cong} R^n\pi_{\Et,*}^{\diamondsuit}\Fs.\]
    \end{enumerate}
    \end{cor}
    \begin{proof}
        Let $Y\approx \varprojlim_i Y_i$ be an affinoid perfectoid tilde-limit of good affinoid adic spaces over $S$. Let $V\rightarrow Y$ be a qcqs étale map, corresponding to a qcqs étale maps $V_i \rightarrow Y_i$ for $i\gg 0$. By Proposition \ref{properties of approximating sheaves}, we have
        \[ H_{\et}^n(X\times_S V,\Fs) = \varinjlim_{i} H_{\et}^n(X\times_S V_i,\Fs).\]
        Sheafifying on $Y$ yields 
        \[ R^n\pi_{\Et,*}\Fs(Y) = \varinjlim_{i} R^n\pi_{\Et,*}\Fs(Y_i).\]
        This proves that $R^n\pi_{\Et,*}\Fs$ has the approximation property and is, in particular, a $v$-sheaf. Considering now $R^n\pi_{v,*}\Fs$, which is the $v$-sheafification of the presheaf taking $Y \rightarrow S$ to $H_{v}^n(X\times_S Y,\Fs)$. By Proposition \ref{Prop: main known results about approximating sheaves}(3), we have
        \[ H_{v}^n(X\times_S Y,\Fs) = H_{\et}^n(X\times_S Y,\Fs).\]
        We deduce that
        \[ R^n\pi_{v,*}\Fs = R^n\pi_{\Et,*}\Fs.\]
        Hence the first two points are proven. For the last point, we argue as in the proof of \cite[Cor. 3.16]{heuer2023diamantine}. Namely, observe that the two sheaves under consideration are the étale sheafification of the following presheaves. The left-hand side comes from the presheaf
        \[ Y \in \Perf_S \mapsto \underset{Y \rightarrow Y_i}{\varinjlim}H_{\et}^n(X\times_S Y_i,\Fs),\]
        where the colimit runs over all factorization of $Y\rightarrow S$ into an adic space $Y_i$ smooth over $S$. Here we use that $Y_{\et, \qcqs} \cong \tworlim_i Y_{i,\et,\qcqs}$ so that it is enough to sheafify over $Y$. The right-hand side comes from the presheaf
        \[ Y \in \Perf_S \mapsto H_{\et}^n(X\times_S Y,\Fs).\]
        By \cite[Prop. 3.17]{heuer2023diamantine}, we may write
        \[ Y\approx \underset{Y \rightarrow Y_i}{\varprojlim}\, Y_i,\]
        where the limit ranges over all morphisms $Y \rightarrow Y_i$ into affinoid adic spaces smooth over $S$. It now follows from Proposition \ref{properties of approximating sheaves} that these two presheaves are isomorphic.
    \end{proof}

\begin{proof}[Proof of Proposition \ref{Prop: main properties of diamantine higher direct images}]
    \begin{enumerate}
        \item Let us show that
        \[ R^n\pi_{\Et,*}^{\diamondsuit}\ga = (R^n\pi_{\et,*}\ga) \otimes_{\Os_{S_{\et}}} \G_a.\]
        When $S=\Spa(K,K^+)$, it follows from \cite[Cor. 3.10]{heuer2024relative}, which says that for any affinoid perfectoid $Y$ over $K$
        \[ H_{\et}^n(X\times_K Y,\ga) = H_{\et}^n(X,\ga) \otimes_K \Os(Y)\]
        If $K$ is perfectoid, \emph{loc. cit.} also shows that
        \[ H_{v}^n(X\times_K Y,\ga) = H_{v}^n(X,\ga) \otimes_K \Os(Y),\]
        which yields
        \[ R^n\pi_{v,*}^{\diamondsuit}\ga = H_v^n(X,\ga) \otimes_K \G_a.\]
        Assume now that $S$ is arbitrary and that $\ga$ is the pullback of a vector bundle $E$ on $S$. By the projection formula, we may assume that $\ga=\G_a$. It now follows from the fact that $R^n\pi_{\an,*}\Os_X$ is an analytic vector bundle whose formation commutes with base-change along any perfectoid space or rigid space $Y \rightarrow S$ \cite[Thm. 5.7.3]{heuer2024relative}. Note that this proves the $4$th point for $\Gs=\ga$.

    \item By Corollary \ref{higher pushforward of approximating sheaves are again approximating}(2)-(3), we have, for any $1\leq m \leq \infty$
    \[ R^n\pi_{v,*}^{\diamondsuit}\Gs[p^{m}] = R^n\pi_{\Et,*}^{\diamondsuit}\Gs[p^{m}] = (R^n\pi_{\Et,*}^{\rig}\Gs[p^{m}])^{\diamondsuit}.\]
    By proper base change (Theorem \ref{Thm: properties of zariski-constructible sheaves}(2)), the sheaf $R^n\pi_{\Et,*}^{\rig}\Gs[p^{m}]$ arises from the ind-Zariski-constructible sheaf $R^n\pi_{\et,*}\Gs[p^{m}]$ via pullback along $\Sm_{/S,\et} \rightarrow S_{\et}$, which shows the first part of the statement. If $S=\Spa(K,K^+)$, we conclude that the ind-Zariski-constructible sheaf $R^n\pi_{\et,*}\Gs[p^{m}]$ is representable by an étale-locally constant rigid group. Assume now that $[p]\colon \widehat{\Gs} \rightarrow \widehat{\Gs}$ is finite étale surjective. For $m< \infty$, the group $\Gs[p^m] \rightarrow S$ is finite étale, such that $R^n\pi_{\et,*}\Gs[p^{m}]$ is a local system, by \cite[Thm. 10.5.1]{SW20}, and is thus easily seen to be representable by a finite étale $S$-group. By a standard limit argument (see the proof of \cite[Thm. 4.4]{heuer2024relative}), we deduce that $\Mb^n = R^n\pi_{v,*}^{\diamondsuit}T_p\Gs$ is a lisse $\Z_p$-sheaf on $S_v$, where $T_p\Gs = \varprojlim_m \Gs[p^m]$ is the Tate module. It remains to show that $R^n\pi_{\Et,*}^{\diamondsuit}\Gs[p^{\infty}]$ is representable by an étale $S$-group. By writing $\Gs[p^{\infty}] = T_p\Gs[\tfrac{1}{p}]/T_p\Gs$, we obtain a short exact sequence of $v$-sheaves over $S$
\[\begin{tikzcd}
	0 & {\Mb^n[\tfrac{1}{p}]/\Mb^n} & {R^n\pi_{v,*}^{\diamondsuit}\Gs[p^{\infty}]} & {\Mb_{\tor}^{n+1}} & {0,}
	\arrow[from=1-1, to=1-2]
	\arrow[from=1-2, to=1-3]
	\arrow[from=1-3, to=1-4]
	\arrow[from=1-4, to=1-5]
\end{tikzcd}\]
where $\Mb_{\tor}^{n+1}$ is the torsion subsheaf of $\Mb^{n+1}$, a finite locally free abelian sheaf. Hence it is enough to show that $\Mb^n[\tfrac{1}{p}]/\Mb^n$ is representable by a separated étale $S$-group. By \cite[Prop. 10.11.(iv)]{scholze2022etale}, this may be checked $v$-locally. Hence we may assume that $S$ is strictly totally disconnected perfectoid space and $\Mb^n \cong \underline{M}$ for a finitely generated $\Z_p$-module $M$, in which case it is clear.

    \item Applying $R\pi_{\tau,*}^{\diamondsuit}$ to the logarithm exact sequence (\ref{logarithm exact sequence}) yields a long exact sequence of sheaves on $\Perf_{S,\tau}$
\[\begin{tikzcd}
	\ldots & {R^{n}\pi_{\tau,*}^{\diamondsuit}\widehat{\Gs}} & {R^{n}\pi_{\tau,*}^{\diamondsuit}\ga} & {R^{n+1}\pi_{\tau,*}^{\diamondsuit}\Gs[p^{\infty}]} & {R^{n+1}\pi_{\tau,*}^{\diamondsuit}\widehat{\Gs}} & \ldots
	\arrow[from=1-1, to=1-2]
	\arrow[from=1-2, to=1-3]
	\arrow["\delta", from=1-3, to=1-4]
	\arrow[from=1-4, to=1-5]
	\arrow[from=1-5, to=1-6]
\end{tikzcd}\]
We claim that $\delta$ vanishes, for all $n\geq 0$. For this, we use the following argument that we learned in \cite[Prop. 6.1.15]{heuer2024curve}. Assume first that $\tau$ is the étale topology. By the first two points, the source of $\delta$ is a vector bundle while its target comes from pullback from an ind-Zariski-constructible sheaf on $S_{\et}$. Hence, by \cite[Prop. 2.6.1]{huber2013étale}, the vanishing of $\delta$ may be checked on geometric fibers. We may thus assume that $S=\Spa(C,C^+)$ for an algebraically closed field $C$, in which case 
\[ R^{n}\pi_{\Et,*}^{\diamondsuit}\ga \cong \G_a^{\oplus r}, \quad R^{n+1}\pi_{\Et,*}^{\diamondsuit}\Gs[p^{\infty}] = \underline{\Lambda},\]
for a discrete group $\Lambda$. Now there is no nonzero maps of groups from the connected rigid group $\G_a^{\oplus r}$ to a discrete group. This shows that $\delta=0$ as required. When $\tau=v$, the vanishing of $\delta$ can be shown after pulling back along a pro-étale cover by a perfectoid space $\widetilde{S} \rightarrow S$. Then the above argument applies verbatim over $\widetilde{S}$, where $R^{n}\pi_{v,*}^{\diamondsuit}\ga$ is representable by an étale vector bundle, since étale and $v$-vector bundles on a perfectoid space agree \cite[Thm. 3.5.8]{kedlaya2019relative2}.

\item To show that the natural map
    \[(R^{n}\pi_{\Et,*}^{\rig}\Gs)^{\diamondsuit} \rightarrow R^{n}\pi_{\Et,*}^{\diamondsuit}\Gs\]
    is an isomorphism, it suffices, by the $5$-Lemma and Corollary \ref{higher pushforward of approximating sheaves are again approximating}(3) applied to the sheaf $\cj{\Gs} = \Gs/\widehat{\Gs}$, to show the statement with $\Gs$ replaced by $\widehat{\Gs}$. Applying again Corollary \ref{higher pushforward of approximating sheaves are again approximating}(3) to the sheaf $\Gs[p^{\infty}]$, it is enough to show the statement for $\ga$, which we have already seen in the proof of the first point. This completes the proof.
\end{enumerate}
\end{proof}

\subsection{The geometric spectral sequence}
We now introduce the geometric Hodge--Tate spectral sequence. We fix a perfectoid field $(K,K^+)$ over $\Q_p$. Let $\pi \colon X \rightarrow S$ be a smooth morphism of seminormal rigid spaces over $K$. From now on we will lighten the notation and write $\pi_{\tau,*}$ instead of $\pi_{\tau,*}^{\diamondsuit}$ for $\tau \in\{\Et,v\}$. We recall that we denote by $\Smd_{X,\et}$ the category of smoothoid spaces $Y$ (cf. Definition \ref{def: Smoothoid spaces}) together with a morphism $Y\rightarrow X$, equipped with the étale topology. We have natural morphisms of sites fitting in a commutative diagram
\begin{equation}\label{defining diagram}\begin{tikzcd}
	{X_v} & {\Smd_{X,\et}} \\
	{\Perf_{S,v}} & {\Perf_{S,\et}.}
	\arrow["\mu", from=1-1, to=1-2]
	\arrow["{\pi_v}"', from=1-1, to=2-1]
	\arrow["\mu"', from=2-1, to=2-2]
	\arrow["{\pi_{\Et}}", from=1-2, to=2-2]
\end{tikzcd}\end{equation}
\begin{definition}
    Let $\Gs$ be a smooth commutative $X$-group, viewed as a sheaf on $X_v$. We define the \textbf{geometric Hodge--Tate spectral sequence with} $\Gs$-\textbf{coefficients} to be the following Leray spectral sequence, consisting of sheaves on $\Perf_{S,\et}$
    \begin{align}\label{geometric Hodge--Tate spectral sequence}
    \mathbf{E}_2^{ij}(\Gs) = R^i\pi_{\Et,*}(R^j\mu_* \Gs) \Rightarrow R^{i+j}(\mu_*\pi_{v,*})\Gs.
\end{align}
\end{definition}

When $S=\Spa(K,K^+)$, we may compare it to the group-theoretic Hodge--Tate spectral sequence (\ref{classical Hodge--Tate spectral sequence for general G}).
\begin{lemma}\label{lemma about comparison maps}
    Let $\pi\colon X \rightarrow S=Spa(K,K^+)$ be a smooth rigid space. The canonical maps given by étale sheafification
\begin{align}\label{obvious maps between spectral sequences}
    E_2^{ij} = \,&H_{\et}^i(X,R^j\nu_* \Gs) \rightarrow R^i\pi_{\Et,*}(R^j\mu_* \Gs)(K)= \mathbf{E}_2^{ij}(K)
\end{align}
commute with differentials. The induced maps on the $E_{\infty}$-page are compatible with the natural maps on the abutments given by étale sheafification.
\begin{align}\label{obvious maps between abutment}
    H_v^n(X,\Gs) \rightarrow R^n(\mu_*\pi_{v,*})\Gs(K).
\end{align}
If $K$ is algebraically closed, the above maps are isomorphisms
\[ E_2^{ij} \cong \mathbf{E}_2^{ij}(K) \, , \quad H_v^n(X,\Gs) \cong R^n(\mu_*\pi_{v,*})\Gs(K).  \]
\end{lemma}
\begin{proof}
 We have a commutative diagram
\[\begin{tikzcd}
	{\Sh(X_v)} && {\Sh(\Smd_{X,\et})} \\
	& {\Sh(\Perf_{K,\et})} \\
	& {\Ab.}
	\arrow["{\mu_*}", from=1-1, to=1-3]
	\arrow["{(\mu \circ \pi_v)_*}"{pos=0.7}, from=1-1, to=2-2]
	\arrow["{\Gamma_v}"', from=1-1, to=3-2]
	\arrow["{\pi_{\Et,*}}"'{pos=0.6}, from=1-3, to=2-2]
	\arrow["{\Gamma_{\et}}", from=1-3, to=3-2]
	\arrow[from=2-2, to=3-2]
\end{tikzcd}\]
By the usual comparison of cohomology between small and big étale sites, it is equivalent to work with $X_{\et}$ or $\Smd_{X,\et}$ to compute $E_{2}^{ij}$. By construction of the Leray sequence, there is a collection of isomorphisms 
\[ E_0^{ij} =\Gamma_{\et}(X,I^{i,j}) = \pi_{\Et,*}(I^{i,j})(K) = \mathbf{E}_0^{ij}(K),\]
compatible with differentials, where $ I^{\bullet,\bullet}$ is Cartan--Eilenberg resolution of $R\mu_*\Gs$. These isomorphisms are immediately seen to induce maps 
\begin{align*}\label{comparison maps between spectral sequences}
   E_r^{ij} \rightarrow \mathbf{E}_r^{ij}(K)\, , \quad \fa \, 0\leq r \leq \infty,
\end{align*}
that coincide with étale sheafification for $r=2$. Moreover, if $K$ is assumed to be algebraically closed, taking $(K,K^+)$-rational points is an exact functor, such that the comparison maps $E_r^{ij} \rightarrow \mathbf{E}_r^{ij}(K)$ are isomorphisms, for each $0\leq r \leq \infty$. This concludes the proof.
\end{proof}

The geometric Hodge--Tate spectral sequence takes the following explicit form.
\begin{proposition}\label{prop(new): explicit description of the geometric HT seq}
    Let $\pi\colon X \rightarrow S$ be a proper smooth map of seminormal rigid spaces and let $\Gs$ be an admissible locally $p$-divisible $X$-group. Assume either that $S=\Spa(K,K^+)$ or that $\ga$ comes from a vector bundle on $S$. Then the geometric Hodge--Tate spectral sequence with $\Gs$-coefficients (\ref{geometric Hodge--Tate spectral sequence}) has the following form
\begin{equation}\label{geometric Hodge--Tate sequence for G reformulated}
 \mathbf{E}_2^{ij}(\Gs) = \left\{\begin{aligned}
  &R^i\pi_{\Et,*}(\womega_{X/S}^j \otimes_{\Os_X} \Lie(\Gs)) \quad &\text{if }j>0\\
  &\BBun_{\Gs,\et}^i \quad &\text{if }j=0
\end{aligned}\right\}  \Longrightarrow R^{i+j}(\mu_* \pi_{v,*})\Gs.
\end{equation}
\end{proposition}
\begin{remark}
    We will prove later in Corollary \ref{possible simplification of the abutment} that the abutment satisfies
    \[ R^{i+j}(\mu_* \pi_{v,*})\Gs = \mu_*\BBun_{\Gs,v}^{i+j}.\]
\end{remark}
\begin{remark}
    For $\Gs = \G_a$, this recovers the relative Hodge--Tate spectral sequence of \cite{heuer2024relative}. Note that we require the base $S$ to be seminormal, while \emph{loc. cit.} allows reduced rigid spaces.
\end{remark}

In order to prove Proposition \ref{prop(new): explicit description of the geometric HT seq}, we begin by recalling that Heuer's Hodge--Tate logarithm map glues to the big étale site $\Smd_{X,\et}$.

\begin{lemma}\label{local correspondence on all smoothoids}
    Let $X$ be a rigid space and let $\Gs$ be an admissible locally $p$-divisible $X$-group. Then we have a natural isomorphism of sheaves on $\Smd_{X,\et}$
    \[ R^j\mu_*\Gs =\begin{cases}
        \widetilde{\Omega}^j \otimes_{\Os_X} \Lie(\Gs) \quad &\text{if }j>0\\
        \Gs \quad &\text{if }j=0,
    \end{cases} \]
where $\womega^j$ is the sheaf on $\Smd_{X,\et}$ defined at (\ref{formula: womega}).
\end{lemma}
\begin{proof}
   For $Z \in \Smd_{X}$, the isomorphism $\HTlog_Z$ on $Z_{\et}$ of Proposition \ref{sheafified correspondence old school} is natural in $Z$. This allows us to glue these isomorphisms to an isomorphism of sheaves on $\Smd_{X,\et}$, as required.
\end{proof}

Proposition \ref{prop(new): explicit description of the geometric HT seq} now follows from the following result.
\begin{proposition}\label{coherent base change of omega tilde}
    Assume that we are in the situation of Proposition \ref{prop(new): explicit description of the geometric HT seq}. Then for any $i\geq 0$ and $j>0$, 
    \[ R^i\pi_{\Et,*}(R^j\mu_*\Gs) = R^i\pi_{\Et,*}(\womega_{X/S}^j \otimes_{\Os_X} \Lie(\Gs)), \]
    and this sheaf is an étale vector bundle on $S$. If $S=\Spa(K,K^+)$, it is equal to 
    \[ H_{\et}^i(X,\womega_{X}^j \otimes_{\Os_X} \Lie(\Gs)) \otimes \G_a.\]
\end{proposition}
\begin{proof}
Let us take an affinoid perfectoid $Y\xrightarrow{g} S$ and form the following cartesian diagram of adic spaces
\[\begin{tikzcd}
	{X\times_S Y} & Y \\
	X & {S.}
	\arrow["{\pi'}", from=1-1, to=1-2]
	\arrow["{g'}"', from=1-1, to=2-1]
	\arrow["g", from=1-2, to=2-2]
	\arrow["\pi"', from=2-1, to=2-2]
\end{tikzcd}\]
The sheaf
\[ R^i\pi_{\Et,*}(R^j\mu_*\Gs) = R^i\pi_{\Et,*}(\widetilde{\Omega}^j \otimes \Lie(\Gs)) \]
is the étale sheafification of the presheaf taking $Y \in \Perf_S$ to
\[ H_{\et}^i(X\times_S Y, \widetilde{\Omega}_{X \times_S Y}^j \otimes\Lie(\Gs)) = H_{\et}^i(X\times_S Y, g'^*\widetilde{\Omega}_{X/S}^j \otimes \Lie(\Gs)),\]
where we use Example \ref{Example: womega of fiber product}. This shows that
\[ R^i\pi_{\Et,*}(R^j\mu_*\Gs) = R^i\pi_{\Et,*}(\widetilde{\Omega}_{X/S}^j \otimes \Lie(\Gs)) \]
and it remains to show that this sheaf is a vector bundle. In the case where $S=\Spa(K,K^+)$, we have by \cite[Cor. 3.10.1]{heuer2024relative},
\[ H_{\et}^i(X\times_K Y, g'^*\widetilde{\Omega}_{X}^j \otimes \Lie(\Gs)) = H_{\et}^i(X,\widetilde{\Omega}_{X}^j \otimes \Lie(\Gs)) \otimes_K \Os(Y). \]
This shows that
\[  R^i\pi_{\Et,*}(\widetilde{\Omega}_{X/S}^j \otimes \Lie(\Gs)) = H_{\et}^i(X,\widetilde{\Omega}_{X}^j \otimes \Lie(\Gs)) \otimes_K \G_a.\]
Assume now that we are in the second case, so that $\Lie(\Gs) = \pi^*E$ for some vector bundle $E$ on $S$. By the projection formula and the above observations, it is enough to show that $R^i\pi_{\Et,*}\widetilde{\Omega}_{X/S}^j$ is a vector bundle. It follows from the fact that $R^i\pi_{\an,*}\womega_{X/S}^j$ is a vector bundle whose formation commutes with pullbacks along perfectoid spaces $Y \rightarrow S$ \cite[Thm. 5.7.3]{heuer2024relative}.
\end{proof}

\subsection{Degeneration}\label{subection: Degeneration}
We now settle to prove the degeneration of the Hodge--Tate spectral sequences. We continue to denote by $(K,K^+)$ a fixed perfectoid field extension of $\Q_p$.

\begin{thm}\label{thm: geometric Hodge--Tate spectral sequence for general G degenerates}
Let $\pi \colon X \rightarrow S$ be a proper smooth map of seminormal rigid space over $K$ and let $\Gs$ be an admissible locally $p$-divisible group over $X$. Assume either that $S=\Spa(K,K^+)$ or that $\ga$ comes from a vector bundle on $S$. Then the geometric Hodge--Tate spectral sequence with $\Gs$-coefficients \begin{equation*}
 \mathbf{E}_2^{ij}(\Gs) = \left\{\begin{aligned}
  &R^i\pi_{\Et,*}(\womega_{X/S}^j \otimes_{\Os_X} \Lie(\Gs)) \quad &\text{if }j>0\\
  &\BBun_{\Gs,\et}^i \quad &\text{if }j=0
\end{aligned}\right\}  \Longrightarrow R^{i+j}(\mu_*\pi_{v,*})\Gs
\end{equation*}
degenerates at the $E_2$-page.
\end{thm}

Before we pass to the proof of the theorem, we note that it immediately implies the following.
\begin{cor}\label{Cor: The classical sseq degenerates}
    Let $X$ be a proper smooth rigid spaces over $K$ with $K$ algebraically closed, and let $\Gs$ be an admissible locally $p$-divisible group over $X$. Then the Hodge--Tate spectral sequence with $\Gs$-coefficients
    \begin{equation*}
 E_2^{ij}(\Gs)= \left\{\begin{aligned}
  &H_{\et}^i(X,\widetilde{\Omega}_X^j \otimes_{\Os_X} \Lie(\Gs)) \quad &\text{if }j>0\\
  &H_{\et}^i(X,\Gs) \quad &\text{if }j=0
\end{aligned}\right\}  \Longrightarrow H_{v}^{i+j}(X,\Gs).
\end{equation*}
degenerates at $E_2$.
\end{cor}
\begin{proof}
This follows from Theorem \ref{thm: geometric Hodge--Tate spectral sequence for general G degenerates} in the case $S=\Spa(K,K^+)$ and the fact that the comparison maps (\ref{obvious maps between spectral sequences}) are isomorphisms, as $K$ is assumed to be algebraically closed.
\end{proof}

We now turn towards the proof of Theorem \ref{thm: geometric Hodge--Tate spectral sequence for general G degenerates}. We start with the case where $\Gs=\ga$ is a vector bundle.
\begin{proposition}\label{geometric Hodge--Tate spectral sequence for G_a degenerates}
    In the situation of Theorem \ref{thm: geometric Hodge--Tate spectral sequence for general G degenerates}, the geometric Hodge--Tate spectral sequence
    \begin{align}\label{Hodge--Tate sequence for G_a}
    \mathbf{E}_2^{ij}(\ga) = R^i\pi_{\Et,*}(\widetilde{\Omega}_{X/S}^j\otimes_{\Os_X} \ga) \Rightarrow R^{i+j}(\mu_*\pi_{v,*})\ga = \mu_*\BBun_{\ga,v}^{i+j}
\end{align}
degenerates at $E_2$.
\end{proposition}
\begin{proof}
Let us first show that the abutment satisfies
\[ R^n(\mu_*\pi_{v,*})\ga=\mu_*\BBun_{\ga,v}^{n}.\]
By a Leray sequence argument, it is enough to show that
\[ R^m\mu_*(\BBun_{\ga,v}^{n})=0, \quad \fa m>0.\]
Let $Y \in \Perf_S$ be an affinoid perfectoid space and let $\nu\colon Y_v \rightarrow Y_{\et}$ be the natural map. It is enough to show that
\[ R^m\nu_*(\BBun_{\ga,v}^{n})=0.\]
We know by Proposition \ref{Prop: main properties of diamantine higher direct images}(1) that $E = \BBun_{\ga,v}^{n}$ is represented by a $v$-vector bundle, but étale and $v$-vector bundles on $Y$ agree by \cite[Thm. 5.3.8]{kedlaya2019relative2}. Therefore Proposition \ref{sheafified correspondence old school} applies and we find that
    \[ R^m\nu_*E = \womega_Y^m \otimes E =0, \]
    as required. 
    
    We now prove the degeneration of the sequence. Let us first assume that $S=\Spa(K,K^+)$. In that case, by \cite[Cor. 3.10.1]{heuer2024relative}, we further have, for any perfectoid space $Y$ 
\[ H_{v}^i(X\times_K Y, \ga) = H_{v}^i(X,\ga) \otimes_K \Os(Y). \]
This shows that
\[  R^n\pi_{v,*}\ga = H_{v}^n(X,\ga) \otimes_K \G_a.\]
The geometric Hodge--Tate spectral sequence thus has the following shape
\begin{align*}
    \mathbf{E}_2^{ij}(\ga) = H_{\et}^i(X,\widetilde{\Omega}_{X/S}^j\otimes_{\Os_X} \ga)\otimes_K \G_a \Rightarrow H_v^{i+j}(X,\ga) \otimes_K \G_a,
\end{align*}
and its degeneration can clearly be checked on $K$-points, which is Proposition \ref{prop: degeneration of additive Hodge--Tate spectral sequence}. Let us now assume that $S$ is arbitrary and that $\ga$ comes from a vector bundle on $S$. In that case the statement follows from the projection formula and the degeneration of the relative Hodge--Tate spectral sequence for $\Os_X$ \cite[Cor. 5.12]{heuer2024relative}.
\end{proof}

\begin{proof}[Proof of Theorem \ref{thm: geometric Hodge--Tate spectral sequence for general G degenerates}]
We have morphisms of $X$-groups and hence of $v$-sheaves
\[ \Gs \supseteq \widehat{\Gs} \xrightarrow{\log} \ga.\]
These induce morphisms between the $E_2$-pages of their respective geometric Hodge--Tate spectral sequences. Moreover, these maps induce isomorphisms
\[ R^j\mu_*\Gs \xleftarrow{\cong} R^j\mu_*\widehat{\Gs} \xrightarrow{\cong} R^j\mu_*\ga, \]
for all $j\geq 1$, by Lemma \ref{local correspondence on all smoothoids}. These induce isomorphisms between the respective terms $\mathbf{E}_2^{ij}$, for $j\geq 1$ and any $i$. As a consequence, since the geometric Hodge--Tate spectral sequence for $\ga$ degenerates, by Proposition \ref{geometric Hodge--Tate spectral sequence for G_a degenerates}, we deduce that $\mathbf{d}_2^{ij}\colon \mathbf{E}_2^{ij}(\Gs) \rightarrow \mathbf{E}_2^{i+2,j-1}(\Gs)$ is the zero map, for all $j\geq 2$ and all $i$. It remains to show that the differential
\[ \mathbf{d}_2^{i,1}\colon R^i\pi_{\Et,*}(\widetilde{\Omega}_{X/S}^1\otimes\Lie(\Gs)) \rightarrow \BBun_{\Gs,\et}^{i+2}\] 
vanishes, for all $i\geq 0$. We have a commutative diagram
\[\begin{tikzcd}
	{R^i\pi_{\Et,*}(\widetilde{\Omega}_{X/S}^1\otimes\Lie(\Gs))} & {R^i\pi_{\Et,*}(\widetilde{\Omega}_{X/S}^1\otimes\Lie(\Gs))} & {R^i\pi_{\Et,*}(\widetilde{\Omega}_{X/S}^1\otimes\Lie(\Gs))} \\
	{\BBun_{\Gs,\et}^{i+2}} & {\BBun_{\widehat{\Gs},\et}^{i+2}} & {\BBun_{\ga,\et}^{i+2},}
	\arrow["{\mathbf{d}_2^{i,1}(\Gs)}", from=1-1, to=2-1]
	\arrow["{=}"', from=1-2, to=1-1]
	\arrow["{=}", from=1-2, to=1-3]
	\arrow["{\mathbf{d}_2^{i,1}(\widehat{\Gs})}", from=1-2, to=2-2]
	\arrow["{\mathbf{d}_2^{i,1}(\ga)=0}", from=1-3, to=2-3]
	\arrow[from=2-2, to=2-1]
	\arrow["\log"', from=2-2, to=2-3]
\end{tikzcd}\]
so that it is enough to show that $\mathbf{d}=\mathbf{d}_2^{i,1}(\widehat{\Gs}) =0$. From the above diagram, we have a factorization
\[\begin{tikzcd}
	&& {R^i\pi_{\Et,*}(\widetilde{\Omega}_{X/S}^1\otimes\Lie(\Gs))} \\
	0 & {\BBun_{\Gs[p^{\infty}],\et}^{i+2}} & {\BBun_{\widehat{\Gs},\et}^{i+2}} & {\BBun_{\ga,\et}^{i+2}} & {0.}
	\arrow["{\textbf{d}'}"', dashed, from=1-3, to=2-2]
	\arrow["{\textbf{d}}", from=1-3, to=2-3]
	\arrow["0", from=1-3, to=2-4]
	\arrow[from=2-1, to=2-2]
	\arrow[from=2-2, to=2-3]
	\arrow["\log"', from=2-3, to=2-4]
	\arrow[from=2-4, to=2-5]
\end{tikzcd}\]
As in the proof of Proposition \ref{Prop: main properties of diamantine higher direct images}(3), we deduce that, as the target of $\mathbf{d}'$ is an ind-Zariski-constructible sheaf on $S_{\et}$ (Proposition \ref{Prop: main properties of diamantine higher direct images}(2)) and its source is fiberwise connected (Proposition \ref{coherent base change of omega tilde}), it must be the zero map. Hence $\mathbf{d}=\mathbf{d}_2^{i,2}(\widehat{\Gs})$ vanishes, as required.
\end{proof}

\begin{cor}\label{possible simplification of the abutment}
     Assume we are in the situation of Theorem \ref{thm: geometric Hodge--Tate spectral sequence for general G degenerates}. Then if $\mu\colon \Perf_{S,v} \rightarrow \Perf_{S,\et}$ denotes the natural map of sites, 
     \[ R^m\mu_*(\BBun_{\Gs,v}^{n}) = 0\, , \quad \fa n\geq 0\, , \quad \fa m >0.\]
     In particular, the abutment of the geometric Hodge--Tate spectral sequence satisfies
    \[ R^{n}(\mu_*\pi_{v,*})\Gs= \mu_*\BBun_{\Gs,v}^{n}.\]
\end{cor}
\begin{proof}
The last equality in the statement follows from the first part, using the Leray sequence computing $R^{n}(\mu_*\pi_{v,*})\Gs$. Let's prove the first part, starting with the case $\Gs=\widehat{\Gs}$. Let $Y$ be any affinoid perfectoid space over $S$ and consider the canonical map $\nu\colon Y_v \rightarrow Y_{\et}$. We need to show that
\[ R^m\nu_*(\BBun_{\widehat{\Gs},v}^{n}) = 0, \quad \fa m>0.\]
From the short exact sequence (\ref{geometric cohomological log exact sequence}), it is enough to show that
    \[ R^m\nu_*(\BBun_{\Gs[p^{\infty}],v}^{n}) = R^m\nu_*(\BBun_{\ga,v}^{n}) = 0.\]
    By Corollary \ref{higher pushforward of approximating sheaves are again approximating}(1)-(2), the sheaf $\BBun_{\widehat{\Gs}[p^{\infty}],v}^{n}$ has the approximation property, hence it follows from Proposition \ref{Prop: main known results about approximating sheaves}(3) that the left term vanishes. The vanishing of the right term was already shown in the proof of Proposition \ref{geometric Hodge--Tate spectral sequence for G_a degenerates}.
    
    We now treat the case of general $\Gs$. For this, we need the following intermediary result.
    \begin{lemma}\label{lemma: BunGv surjected upon by smooth space}
    In the situation of Theorem \ref{thm: geometric Hodge--Tate spectral sequence for general G degenerates}, there exists a surjection $V \rightarrow \BBun_{\Gs,v}^n$ of $v$-sheaves on $\Perf_{S}$, where $V$ is a smooth $S$-space. 
\end{lemma}
Let us assume the lemma. We have an exact sequence of $v$-sheaves on $\Perf_{S,v}$
\[\begin{tikzcd}
	\ldots & {\BBun_{\widehat{\Gs},v}^n} & {\BBun_{\Gs,v}^n} & {\BBun_{\cj{\Gs},v}^n} & {\BBun_{\widehat{\Gs},v}^{n+1}} & \ldots
	\arrow["{\delta_{n-1}}", from=1-1, to=1-2]
	\arrow[from=1-2, to=1-3]
	\arrow[from=1-3, to=1-4]
	\arrow["{\delta_n}", from=1-4, to=1-5]
	\arrow[from=1-5, to=1-6]
\end{tikzcd}\]
Those sheaves and the above exact sequence admit a natural extension to $\Adic_{S,v}$. By Corollary \ref{higher pushforward of approximating sheaves are again approximating}, the sheaves $\BBun_{\cj{\Gs},v}^n$ satisfy the approximation property. We claim that the sheaves $\Ker(\delta_n)$ and $\Image(\delta_n)$ also satisfy the approximation property. Indeed, consider the composition $V\rightarrow \BBun_{\Gs,v}^n \rightarrow \BBun_{\cj{\Gs},v}^n$, where $V \rightarrow \BBun_{\Gs,v}^n$ is as in Lemma \ref{lemma: BunGv surjected upon by smooth space}. Then by Proposition \ref{Prop: approximating sheaves are small v-sheaves}, the diamond
\[ R = V \times_{\BBun_{\cj{\Gs},v}^n} V \]
is open in $V \times_S V$, such that the étale quotient $V/^{\et}R$ has the approximation property. In particular, it coincides with its $v$-sheafification $V/^{v}R = \Ker(\delta_n)$. We now consider the short exact sequence of étale sheaves
\[\begin{tikzcd}
	0 & {\Ker(\delta_n)} & {\BBun_{\cj{\Gs},v}^n} & {\Image_{\et}(\delta_n)} & {0,}
	\arrow[from=1-1, to=1-2]
	\arrow[from=1-2, to=1-3]
	\arrow["{\delta_n}", from=1-3, to=1-4]
	\arrow[from=1-4, to=1-5]
\end{tikzcd}\]
    where the first and second sheaves satisfy the approximation property. It is then easily seen that the third sheaf $\Image_{\et}(\delta_n)$ automatically satisfies the approximation property as well. In particular, $\Image_{\et}(\delta_n)$ is a $v$-sheaf and thus the above sequence is also right-exact in the $v$-topology.
    
    We now consider the exact sequence of $v$-sheaves
\[\begin{tikzcd}
	0 & {\Image(\delta_{n-1})} & {\BBun_{\widehat{\Gs},v}^n} & {\BBun_{\Gs,v}^n} & {\Ker(\delta_n)} & {0.}
	\arrow[from=1-1, to=1-2]
	\arrow[from=1-2, to=1-3]
	\arrow[from=1-3, to=1-4]
	\arrow[from=1-4, to=1-5]
	\arrow[from=1-5, to=1-6]
\end{tikzcd}\]
Each sheaf $\Fs$ on the left or right of $\BBun_{\Gs,v}^{n}$ in the above sequence satisfies $R^m\mu_*\Fs = 0$, by the case $\Gs=\widehat{\Gs}$ and Proposition \ref{Prop: main known results about approximating sheaves}(3) respectively. By decomposing the above sequence into short exact sequences, we finally deduce that
\[ R^m\mu_*(\BBun_{\Gs,v}^{n}) = 0,\]
which concludes the proof.
\end{proof}

\begin{proof}[Proof of Lemma \ref{lemma: BunGv surjected upon by smooth space}]
    By Theorem \ref{thm: geometric Hodge--Tate spectral sequence for general G degenerates}, we have a descending filtration 
    \[ 0=\FFil^{n+1} \sub \FFil^n \sub \ldots \sub \FFil^0= R^n(\mu_*\pi_{v,*})\Gs\]
    by étale sheaves on $\Perf_{S}$ with $\FFil^n = \BBun_{\Gs,\et}^n$ and short exact sequences, for $0\leq i < n$
\[\begin{tikzcd}
	0 & {\FFil^{i+1}} & {\FFil^{i}} & {\mathbf{V}_{i,n}} & {0,}
	\arrow[from=1-1, to=1-2]
	\arrow[from=1-2, to=1-3]
	\arrow[from=1-3, to=1-4]
	\arrow[from=1-4, to=1-5]
\end{tikzcd}\]
where $\mathbf{V}_{i,n} = R^i\pi_{\Et,*}(\womega_{X/S}^{n-i}\otimes\Lie(\Gs))$. By passing to the $v$-sheafification $\mu^*$, we obtain a filtration by $v$-sheaves on $\Perf_S$
\[ \mu^*\FFil^{\bullet} \sub \mu^*R^n(\mu_*\pi_{v,*})\Gs = \BBun_{\Gs,v}^n,\]
with short exact sequences 
\[\begin{tikzcd}
	0 & {\mu^*\FFil^{i+1}} & {\mu^*\FFil^{i}} & {\mathbf{V}_{i,n}} & {0.}
	\arrow[from=1-1, to=1-2]
	\arrow[from=1-2, to=1-3]
	\arrow[from=1-3, to=1-4]
	\arrow[from=1-4, to=1-5]
\end{tikzcd}\]
Here, we use the fact that $\mathbf{V}_{i,n}$ is a vector bundle and thus is already a $v$-sheaf. Observe that the sequence is already exact for the étale topology, using that the surjection of étale sheaves $\FFil^i \rightarrow \mathbf{V}_{i,n}$ factors through $\mu^*\FFil^i$. By Lemma \ref{lemma: étale surjection of $v$-sheaves} below, the map $\mu^*\FFil^i \rightarrow \mathbf{V}_{i,n}$ admits sections étale locally on $\mathbf{V}_{i,n}$. Hence, by arguing inductively, it is enough to show that each $\mu^*\FFil^i$ admits a surjection of $v$-sheaves from a smooth $S$-space, and we reduce to $i=n$. For this, it is enough to show that $\FFil^n = \BBun_{\Gs,\et}^n$ admits a surjection of étale sheaves from a smooth $S$-space. By Proposition \ref{Prop: main properties of diamantine higher direct images}(4), it is enough to show that the sheaf $\BBun_{\Gs,\et}^{n,\rig}$ on $\Sm_{/S,\et}$ admits a surjection from a smooth $S$-space. This follows from the co-Yoneda Lemma, since the site $\Sm_{/S,\et}$ is essentially small.
\end{proof}

\begin{lemma}\label{lemma: étale surjection of $v$-sheaves}
    Let $S$ be a good adic space over $\Q_p$ and $\pi\colon \Fs \rightarrow \Gs$ be a map of $v$-sheaves on $\Perf_S$. Suppose that $\pi$ is a surjection of étale sheaves. Then for any good adic space $Y$ over $S$ and any map of $v$-sheaves $g\in \Gs(Y)$, there exists an étale cover $Y'\rightarrow Y$ and a map $f\in \Fs(Y')$ such that $\pi(f) = g\restr{Y'}$.
\end{lemma}
\begin{proof}
    We may assume that $Y$ is affinoid. By \cite[Lemma 15.3]{scholze2022etale}, there exists a $v$-cover 
\[ \widetilde{Y}\approx \varprojlim_j Y_j \rightarrow Y\] 
where $\widetilde{Y}$ is an affinoid perfectoid space,$Y_j$ are affinoid, finite étale over $Y$ and each $\widetilde{Y} \rightarrow Y_j$ is surjective. Since $\Fs \rightarrow \Gs$ is surjective for the étale topology on $\Perf_S$, we may find a qcqs étale cover $\widetilde{Z} \rightarrow \widetilde{Y}$ and a map $\widetilde{s}$ fitting in the following commutative diagram of sheaves on $\Perf_S$ 
\[\begin{tikzcd}
	{\widetilde{Z}} && \Fs \\
	{\widetilde{Y}} & Y & \Gs
	\arrow["{\widetilde{f}}", dashed, from=1-1, to=1-3]
	\arrow[from=1-1, to=2-1]
	\arrow["\pi", from=1-3, to=2-3]
	\arrow[from=2-1, to=2-2]
	\arrow["g", from=2-2, to=2-3]
\end{tikzcd}\]
By Lemma \ref{diamond and étale site of strong tilde timit}(2), $\widetilde{Z}$ arises via pullback from an étale cover $Z_j\rightarrow Y_j$. Using that $\widetilde{Y} \rightarrow Y_j$ is surjective, it is easy to see that the map $\widetilde{f}$ descends to a map $f_j\colon Z_j \rightarrow \Fs$ satisfying the required properties.
\end{proof}

\begin{remark}\label{Remark: BunGet is a v-sheaf}
    In situation of Theorem \ref{thm: geometric Hodge--Tate spectral sequence for general G degenerates}, there is therefore a filtration of $\BBun_{\Gs,v}^n$ by subsheaves on $\Perf_{S,\et}$
    \[  \FFil^n \sub \ldots \sub \FFil^0 =\BBun_{\Gs,v}^n \]
    with 
    \[\FFil^n = \BBun_{\Gs,\et}^n \text{ and }\FFil^i/\FFil^{i+1} \cong R^i\pi_{\Et,*}(\womega_{X/S}^j\otimes \Lie(\Gs)),\]
    for each $0\leq i<n$. In particular, since $\BBun_{\Gs,v}^n$ and $R^i\pi_{\Et,*}(\womega_{X/S}^j\otimes \Lie(\Gs))$ are small $v$-sheaves, we deduce, arguing inductively, that each $\FFil^i$ is a small $v$-sheaf. In particular, $\BBun_{\Gs,\et}^n$ is a small $v$-sheaf on $\Perf_S$. We do not know whether its canonical extension to $S_v$ agrees with $\BBun_{\Gs,\et}^{n,\rig}$ on $\Sm_{/S,\et}$. When the latter is representable by a smooth $S$-group, this is an easy consequence of Proposition \ref{Prop: main properties of diamantine higher direct images}(4).
\end{remark}

\begin{remark}\label{remark: local splittings}
    Assume that $S=\Spa(K,K^+)$. Then arguing as in the proof of \cite[Thm. 2.7.3]{heuer2023diamantine}, we may use the exponential of Lemma \ref{lemma: exponential converges} to produce a splitting of the short exact sequence
\[\begin{tikzcd}
	0 & {\FFil^{i+1}} & {\FFil^{i}} & {H^i(X,\womega_X^{n-i}\otimes \Lie(\Gs))\otimes_K \G_a} & 0
	\arrow[from=1-1, to=1-2]
	\arrow[from=1-2, to=1-3]
	\arrow[from=1-3, to=1-4]
	\arrow[from=1-4, to=1-5]
\end{tikzcd}\]
over an open subgroup of the form $\Lambda_i \otimes_{K^+} \G_a^+$, where $\Lambda_i \sub H^i(X,\womega_X^{n-i}\otimes \Lie(\Gs))$ is an open $K^+$-lattice.
\end{remark}

\subsection{The Hodge--Tate decomposition}
In this section, we prove the following Hodge--Tate decomposition with locally $p$-divisible coefficients.
\begin{thm}\label{Theorem: The classical sseq splits}
    Let $(K,K^+)$ be an algebraically closed, non-archimedean field extension of $\Q_p$. Let $X$ be a proper smooth rigid space over $K$ and let $\Gs \rightarrow X$ be an admissible locally $p$-divisible group. A choice of exponential for $K$ and of a flat lift $\X$ to $\BdR+/\xi^2$ induce a splitting of the Hodge--Tate spectral sequence with $\Gs$-coefficients (\ref{classical Hodge--Tate sequence for G reformulated}), natural in $\Gs$ and $\X$, and thus a decomposition
    \begin{align}
        H_{v}^n(X,\Gs) = H_{\et}^n(X,\Gs) \oplus \bigoplus_{i=0}^{n-1} H^i(X,\widetilde{\Omega}_X^{n-i}\otimes_{\Os_X}\Lie(\Gs)).
    \end{align}
\end{thm}
To this end, we will use the results of the previous section to prove the representability of the diamantine higher direct images $\BBun_{\widehat{\Gs},\et}^n$. We start with the following lemma about extensions of smooth relative groups, generalizing \cite[Prop. 18]{Farg19}.
\begin{lemma}\label{extension by étale rigid group}
Let $S$ be a good adic space over $\Q_p$. Consider a short exact sequence of abelian sheaves on $\Perf_{S,v}$
\[\begin{tikzcd}
	0 & \Gamma & \Gs & {\mathbf{V}} & {0,}
	\arrow[from=1-1, to=1-2]
	\arrow[from=1-2, to=1-3]
	\arrow["f", from=1-3, to=1-4]
	\arrow[from=1-4, to=1-5]
\end{tikzcd}\]
where $\mathbf{V}$ is an étale vector bundle on $S$ and $\Gamma$ is an étale separated $S$-group of $p^{\infty}$-torsion. Then the sheaf $\Gs$ is represented by a locally $p$-divisible $S$-group, $\Gs= \widehat{\Gs}$ and the following diagram commutes
\[\begin{tikzcd}
	& \ga \\
	\Gs \\
	& {\mathbf{V}.}
	\arrow["{D(f)}", from=1-2, to=3-2]
	\arrow["\cong"', from=1-2, to=3-2]
	\arrow["{\log_{\Gs}}", from=2-1, to=1-2]
	\arrow["f"', from=2-1, to=3-2]
\end{tikzcd}\]
If we suppose furthermore that $[p]\colon \Gamma \rightarrow \Gamma$ is finite étale surjective, then $\Gs$ is an analytic $p$-divisible $S$-group.
\end{lemma}
\begin{proof}
Let $v\colon Y \rightarrow \mathbf{V}$ be any $v$-cover from a perfectoid space. By assumption, there exists a $v$-cover $Y' \rightarrow Y$ and a section $g\in \Gs(Y')$ such that $f(g)=v\restr{Y'}$. Then $\Gs \times_{\mathbf{V}} Y' \cong \Gamma \times_S Y'$ is represented by an étale separated perfectoid space over $Y'$. By \cite[Prop. 10.11(iv)]{scholze2022etale}, $\Gs \rightarrow \mathbf{V}$ is étale. 
We deduce, using \cite[Lemma 15.6]{scholze2022etale}, that $\Gs$ is represented by an object of the small étale site $\mathbf{V}_{\et}$ of $\mathbf{V}$ considered as an adic space. This shows that $\Gs$ is represented by a smooth $S$-group. 

Next, observe that the composition 
\[h\colon \Gs \xrightarrow{f} \mathbf{V} \xrightarrow{D(f)^{-1}} \ga\]
is an étale morphism with $D(h) = \id$. By the unicity part of Lemma \ref{statements about p-topological torsion subgroup}(1), we deduce that $h\restr{\widehat{\Gs}}=\log_{\Gs}$. 

Let us now show that $\Gs=\widehat{\Gs}$. This may be checked locally, such that we may assume that $S$ is affinoid and $\mathbf{V}$ is a trivial vector bundle. Let $Y$ be an affinoid perfectoid space over $S$ and let $g \in \Gs(Y)$. It suffices to show that $[p^n]g \in \widehat{\Gs}(Y)$ for some $n\geq 0$, since the cokernel $\cj{\Gs}$ of the inclusion $\widehat{\Gs} \hookrightarrow \Gs$ is $p$-torsionfree, by \cite[Lemma 2.11.2]{heuer2022geometric}. Let $U \cong \B^n \sub \mathbf{V}$ denote an open subspace isomorphic to a closed ball where $\log_{\Gs}$, and hence also $f$ has a section, as in Lemma \ref{lemma: exponential converges}, see Remark \ref{rem: exponential on trivial vector bundles}. We have $f([p^n]g)=p^nf(g) \in U$ for large enough $n$. Note that we have $f^{-1}(U) \cong \Gamma \times U$ as rigid groups, such that
\[\widehat{f^{-1}(U)} = \widehat{U} \times \widehat{\Gamma} = U \times \Gamma =  f^{-1}(U). \]
 Since $[p^n]g \in \widehat{f^{-1}(U)} \subset \widehat{\Gs}$, this proves that $\Gs=\widehat{\Gs}$, as required. In particular, $\log_{\Gs} =h$ is étale and surjective, such that $\Gs$ is locally $p$-divisible.

For the last part of the statement, assume that $[p]\colon \Gamma \rightarrow \Gamma$ is surjective and finite étale. As $[p]\colon \mathbf{V} \rightarrow \mathbf{V}$ is an isomorphism because $\mathbf{V}$ is a vector bundle, the assumption implies that $[p]\colon \Gs \rightarrow \Gs$ is also a surjection with the kernel $\Gs[p] = \Gamma[p]$ being finite étale over $S$, which concludes the prooof.
\end{proof}

\begin{proposition}\label{representability of diamantine higher direct image for topologically torsion groups}
    Let $\pi\colon X \rightarrow S$ be a proper smooth morphism of seminormal rigid spaces over $K$ and let $\Gs \rightarrow X$ be an admissible locally $p$-divisible group. Assume that we are in one of the following situations:
    \begin{itemize}
        \item $S=\Spa(K,K^+)$, or
        \item The Lie algebra $\ga$ comes from a vector bundle over $S$ and $\widehat{\Gs} \xrightarrow{[p]}\widehat{\Gs}$ is finite étale surjective.
    \end{itemize}
    \begin{enumerate}
        \item For any $n\geq 0$, the sheaf $\BBun_{\widehat{\Gs},\et}^n$ is representable by an admissible locally $p$-divisible $S$-group of topological $p$-torsion, and the sequence (\ref{geometric cohomological log exact sequence})
\[\begin{tikzcd}
	0 & {\BBun_{\Gs[p^{\infty}],\et}^{n}} & {\BBun_{\widehat{\Gs},\et}^{n}} & {\BBun_{\ga,\et}^{n}} & 0
	\arrow[from=1-1, to=1-2]
	\arrow[from=1-2, to=1-3]
	\arrow["\log", from=1-3, to=1-4]
	\arrow[from=1-4, to=1-5]
\end{tikzcd}\]
coincides with the logarithm sequence for this group. If $S=\Spa(K,K^+)$, the analogous statement holds for $\BBun_{\widehat{\Gs},v}^n$.
        \item Assume that $S=\Spa(K,K^+)$ with $K$ algebraically closed. A choice of exponential for $K$ induces a splitting on $K$-points of the sequence (\ref{geometric cohomological log exact sequence}), for $\tau \in \{\et,v\}$
\begin{equation}\label{eq: split log exact sequence}\begin{tikzcd}
	0 & {H_{\et}^n(X,\Gs[p^{\infty}])} & {H_{\tau}^n(X,\widehat{\Gs})} & {H_{\tau}^{n}(X,\ga)} & {0,}
	\arrow[from=1-1, to=1-2]
	\arrow[from=1-2, to=1-3]
	\arrow["\log"', shift right, from=1-3, to=1-4]
	\arrow["{s_{\exp}}"', shift right, dashed, from=1-4, to=1-3]
	\arrow[from=1-4, to=1-5]
\end{tikzcd}\end{equation}
that is natural in $\Gs$.
    \end{enumerate}
\end{proposition}
\begin{proof}
Let us prove the first part. We apply Lemma \ref{extension by étale rigid group} to the exact sequence (\ref{geometric cohomological log exact sequence}). This is possible, since $\BBun_{\widehat{\Gs},\et }$ is a $v$-sheaf, by Remark \ref{Remark: BunGet is a v-sheaf}. In the case $S=\Spa(K,K^+)$ and $\tau=v$, Lemma \ref{extension by étale rigid group} also applies to the exact sequence (\ref{geometric cohomological log exact sequence}), as $\BBun_{\ga,v}^n = H_v^n(X,\ga) \otimes \G_a$ is an étale vector bundle, by Proposition \ref{Prop: main properties of diamantine higher direct images}(1). This proves the first point. The second point follows from Proposition \ref{exponential on K points of locally p-divisible group}. 
\end{proof}

\begin{proof}[Proof of Theorem \ref{Theorem: The classical sseq splits}]
    We have a commutative diagram
\[\begin{tikzcd}
	{\Fil^nH_{v}^n(X,\widehat{\Gs})} & {H_{\et}^n(X,\widehat{\Gs})} \\
	{\Fil^nH_{v}^n(X,\ga)} & {H_{\et}^n(X,\ga).}
	\arrow["\cong", from=1-1, to=1-2]
	\arrow["\log"', from=1-1, to=2-1]
	\arrow["\log"', shift right, from=1-2, to=2-2]
	\arrow["\cong"', from=2-1, to=2-2]
	\arrow["{s_{\exp}}"', shift right, dashed, from=2-2, to=1-2]
\end{tikzcd}\]
Proposition \ref{representability of diamantine higher direct image for topologically torsion groups} provides the splitting $s_{\exp}$ and hence yields a splitting $s_n$ of the left vertical map. For general $i<n$, we have a commutative diagram
\[\begin{tikzcd}
	0 & {\Fil^{i+1}H_{v}^n(X,\widehat{\Gs})} & {\Fil^iH_{v}^n(X,\widehat{\Gs})} & {H^i(X,\widetilde{\Omega}_X^{n-i}\otimes \Lie(\Gs))} & 0 \\
	0 & {\Fil^{i+1}H_{v}^n(X,\ga)} & {\Fil^iH_{v}^n(X,\ga)} & {H^i(X,\widetilde{\Omega}_X^{n-i}\otimes \Lie(\Gs))} & {0.}
	\arrow[from=1-1, to=1-2]
	\arrow[from=1-2, to=1-3]
	\arrow["\log"', shift right, from=1-2, to=2-2]
	\arrow[from=1-3, to=1-4]
	\arrow["\log"', from=1-3, to=2-3]
	\arrow[from=1-4, to=1-5]
	\arrow["\cong", from=1-4, to=2-4]
	\arrow[from=2-1, to=2-2]
	\arrow["{s_{i+1}}"', shift right, dashed, from=2-2, to=1-2]
	\arrow[from=2-2, to=2-3]
	\arrow[shift right, from=2-3, to=2-4]
	\arrow["{s_{\X}}"', shift right, dashed, from=2-4, to=2-3]
	\arrow[from=2-4, to=2-5]
\end{tikzcd}\]
By descending induction on $i$, we have a section $s_{i+1}$ splitting the left vertical map. Since the left square above is a pushout square, we obtain a section $s_i$ of the middle vertical map. By Proposition \ref{prop: degeneration of additive Hodge--Tate spectral sequence}, the lift $\X$ induces a splitting $s_{\X}$ of the Hodge--Tate spectral sequence with $\ga$-coefficients. We may precompose $s_i$ with $s_{\X}$ and post-compose it with the natural map
\[ \Fil^iH_v^n(X,\widehat{\Gs}) \rightarrow \Fil^iH_v^n(X,\Gs).\]
This yields a map $H^i(X,\widetilde{\Omega}_X^{n-i} \otimes \Lie(\Gs)) \rightarrow \Fil^iH_v^n(X,\Gs)$ that splits the Hodge--Tate spectral sequence with $\Gs$-coefficients, as required.
\end{proof}

\subsection{Cohomology with coefficients in $p$-adic universal covers}
In this section, we prove a Hodge--Tate decomposition with coefficients in the $p$-adic universal cover $\widetilde{\Gs}$ of locally $p$-divisible groups.
\begin{definition}
    Let $X$ be a good adic space over $\Q_p$ and let $\Gs$ be an admissible locally $p$-divisible $X$-group. Assume that $[p]\colon \Gs \rightarrow \Gs$ is finite étale surjective. We define the $p$-adic universal cover of $\Gs$ to be
    \[ \widetilde{\Gs} = \varprojlim_{[p]}\Gs,\]
    viewed as a group diamond over $X$. It fits in a short exact sequence of sheaves on $X_v$
\begin{equation}\begin{tikzcd}
	0 & {T_p\Gs} & {\widetilde{\Gs}} & \Gs & {0,}
	\arrow[from=1-1, to=1-2]
	\arrow[from=1-2, to=1-3]
	\arrow[from=1-3, to=1-4]
	\arrow[from=1-4, to=1-5]
\end{tikzcd}\end{equation}
    where $T_p\Gs = \varprojlim_m \Gs[p^m]$ is the Tate module of $\Gs$.
\end{definition}
\begin{remark}
   If $\Gs = \widehat{\Gs}$ is analytic $p$-divisible, we further have a short exact sequence of sheaves on $X_v$
\begin{equation}\label{eq: log for BC space}\begin{tikzcd}
	0 & {V_p\Gs} & {\widetilde{\Gs}} & \ga & {0,}
	\arrow[from=1-1, to=1-2]
	\arrow[from=1-2, to=1-3]
	\arrow["\log", from=1-3, to=1-4]
	\arrow[from=1-4, to=1-5]
\end{tikzcd}\end{equation}
where $V_p\Gs = T_p\Gs[\tfrac{1}{p}]$ is the rational Tate module of $\Gs$. In particular, $\widetilde{\Gs}$ then is an \emph{effective} Banach--Colmez space \cite[§4]{Fontaine03}. 
\end{remark}
\begin{example}
    Let $\Gs= \widehat{\G}_m$. Then by \cite[Prop. II.2.2]{fargues2024geometrization}, we have $\widetilde{\widehat{\G}}_m= \B^{\varphi=p}$, where $\B$ is the $v$-sheaf on $X$ defined by the formula
    \[ \B \colon S \in \Perf_X \longmapsto H^0(Y_S^{\FF},\Os_{Y_S^{\FF}}),\]
    and the sequence (\ref{eq: log for BC space}) coincides with the fundamental exact sequence of $p$-adic Hodge theory
\begin{equation}\label{eq: fund eq of p-adic HT}\begin{tikzcd}
	0 & {\Q_p(1)} & {\B^{\varphi=p}} & {\G_a} & {0.}
	\arrow[from=1-1, to=1-2]
	\arrow[from=1-2, to=1-3]
	\arrow["\theta", from=1-3, to=1-4]
	\arrow[from=1-4, to=1-5]
\end{tikzcd}\end{equation}
\end{example}

We obtain the following result.
\begin{thm}\label{thm: HT decomposition for BC space}
    Let $X$ be a proper smooth rigid space over an algebraically closed non-archimedean field $(K,K^+)$ over $\Q_p$. Let $\Gs$ be an admissible locally $p$-divisible $X$-group with $[p]\colon \Gs \rightarrow \Gs$ finite étale surjective. Then there is a spectral sequence, natural in $X$ and $\Gs$
    \begin{equation}\label{eq: sseq for BC spaces}
 E_2^{ij}= \left\{\begin{aligned}
  &H_{\et}^i(X,\widetilde{\Omega}_X^j \otimes_{\Os_X} \Lie(\Gs)) \quad &\text{if }j>0\\
  &{\varprojlim}_{[p]}H_{\et}^i(X,\Gs) \quad &\text{if }j=0
\end{aligned}\right\}  \Longrightarrow H_{v}^{i+j}(X,\widetilde{\Gs}),
\end{equation}
that lies over the Hodge--Tate spectral sequence with $\Gs$-coefficients (\ref{classical Hodge--Tate sequence for G reformulated}). Moreover, this spectral sequence degenerates at $E_2$. A choice of exponential for $K$ and of a flat lift $\X$ to $\BdR+/\xi^2$ induce a decomposition, natural in $\Gs$ and $\X$
    \begin{align}\label{eq: Hodge--Tate decomp for BC space}
        H_{v}^n(X,\widetilde{\Gs}) = \varprojlim_{[p]} H_{\et}^n(X,\Gs) \oplus \bigoplus_{i=0}^{n-1} H^i(X,\widetilde{\Omega}_X^{n-i}\otimes_{\Os_X}\Lie(\Gs)).
    \end{align}
\end{thm}
\begin{proof}
The spectral sequence is obtained by taking the inverse limit over $[p]$ of the spectral sequence (\ref{classical Hodge--Tate sequence for G reformulated}), which then is immediately seen to degenerate at $E_2$, by Corollary \ref{Cor: The classical sseq degenerates}. We claim that we have the following formula for the limiting term
\[ \varprojlim_{[p]}H_v^n(X,\Gs) = H_v^n(X,\widetilde{\Gs}).\]
As $\Sh(X_v)$ is a replete topos, we have
\[ \widetilde{\Gs} = \Rlim_{[p]} \Gs.\]
Hence
\[ \RGamma_v(X,\widetilde{\Gs}) = \Rlim_{[p]} \RGamma_v(X,\Gs).\]
It remains to show that
\[ R^1\varprojlim_{[p]} H_v^n(X,\Gs)=0,\]
for each $n\geq 0$. We have a short exact sequence on $X_v$
\[\begin{tikzcd}
	0 & {\widetilde{\widehat{\Gs}}} & {\widetilde{\Gs}} & {\cj{\Gs}} & {0.}
	\arrow[from=1-1, to=1-2]
	\arrow[from=1-2, to=1-3]
	\arrow[from=1-3, to=1-4]
	\arrow[from=1-4, to=1-5]
\end{tikzcd}\]
We end up with the following commutative diagram where the rows and the middle column are exact
\[\begin{tikzcd}
	& \ldots & \ldots & \ldots \\
	0 & {R^1\varprojlim_{[p]}H^{n-1}(X,\widehat{\Gs})} & {H_v^n(X,\widetilde{\widehat{\Gs}})} & {\varprojlim_{[p]} H_v^n(X,\widehat{\Gs})} & 0 \\
	0 & {R^1\varprojlim_{[p]}H^{n-1}(X,\Gs)} & {H_v^n(X,\widetilde{\Gs})} & {\varprojlim_{[p]} H_v^n(X,\Gs)} & 0 \\
	0 & 0 & {H_v^n(X,\cj{\Gs})} & {\varprojlim_{[p]}H_v^n(X,\cj{\Gs})} & {0.} \\
	& \ldots & \ldots & \ldots
	\arrow[from=1-2, to=2-2]
	\arrow[from=1-3, to=2-3]
	\arrow[from=1-4, to=2-4]
	\arrow[from=2-1, to=2-2]
	\arrow[from=2-2, to=2-3]
	\arrow[from=2-2, to=3-2]
	\arrow[from=2-3, to=2-4]
	\arrow[from=2-3, to=3-3]
	\arrow[from=2-4, to=2-5]
	\arrow[from=2-4, to=3-4]
	\arrow[from=3-1, to=3-2]
	\arrow[from=3-2, to=3-3]
	\arrow[from=3-2, to=4-2]
	\arrow[from=3-3, to=3-4]
	\arrow[from=3-3, to=4-3]
	\arrow[from=3-4, to=3-5]
	\arrow[from=3-4, to=4-4]
	\arrow[from=4-1, to=4-2]
	\arrow[from=4-2, to=4-3]
	\arrow[from=4-2, to=5-2]
	\arrow["\cong", from=4-3, to=4-4]
	\arrow[from=4-3, to=5-3]
	\arrow[from=4-4, to=4-5]
	\arrow[from=4-4, to=5-4]
\end{tikzcd}\]
Here, we use that $\cj{\Gs}$ is uniquely $p$-divisible, using our assumptions on $\Gs$ and \cite[Lemma 2.11.2]{heuer2022geometric}, such that $R^1\varprojlim_{[p]}H^{n}(X,\cj{\Gs})=0$. By applying the Snake Lemma to the above diagram, we are reduced to showing that
\[ R^1\varprojlim_{[p]} H_{v}^n(X,\widehat{\Gs})=0, \quad \fa n\geq 0.\]
Using the split short exact sequence (\ref{eq: split log exact sequence}), it is equivalent to show that
\[ R^1\varprojlim_{[p]} H_{v}^n(X,\Gs[p^{\infty}])=0, \quad \fa n\geq 0.\]
As $H_v^n(X,T_p\Gs)$ is a finitely generated $\Z_p$-module, by \cite[Thm. 5.1]{scholze2013padicHodge}, it follows that we can write
\[ H_v^n(X,\Gs[p^{\infty}]) = (\Q_p/\Z_p)^{\oplus r} \oplus T,\]
for some finite torsion group $T$. It follow that the system 
\[\begin{tikzcd}
	\ldots & {H_v^n(X,\Gs[p^{\infty}])} & {H_v^n(X,\Gs[p^{\infty}])} & \ldots
	\arrow["{[p]}", from=1-1, to=1-2]
	\arrow["{[p]}", from=1-2, to=1-3]
	\arrow["{[p]}", from=1-3, to=1-4]
\end{tikzcd}\]
is Mittag--Leffler, and the claim follows.
The rest of the statement now follows by applying $\varprojlim_{[p]}$ to the Hodge--Tate decomposition with $\Gs$-coefficients (Theorem \ref{Theorem: The classical sseq splits}).
\end{proof}

As an application, we obtain the following deformation of the Hodge--Tate decomposition along Fontaine's theta map.
\begin{cor}\label{cor: deformation of HT decomposition along theta map}
    Let $X$ be a proper smooth rigid space over a complete algebraically closed extension $K$ of $\Q_p$. Then a flat lift $\X$ to $\BdR+/\xi^2$ and an exponential for $K$ induces a decomposition, natural in $\X$
    \begin{align}
        H_{\et}^n(X,\Q_p) \otimes_{\Q_p}B^{\varphi =p} = \varprojlim_{[p]} H_{\et}^n(X,\widehat{\G}_m) \oplus \bigoplus_{i=0}^{n-1} H^i(X,\womega_X^{n-i}).
    \end{align}
This lies over the classical Hodge--Tate decomposition (\ref{eq: the classical Hodge--Tate decomposition}) via Fontaine's map $B^{\varphi=p} \xrightarrow{\theta} K$.
\end{cor}
\begin{proof}
    By Theorem \ref{thm: HT decomposition for BC space}, all that remains to be shown is the isomorphism
    \begin{align} H_{\et}^n(X,\Q_p) \otimes_{\Q_p}B^{\varphi =p} = H_v^n(X,\B^{\varphi =p}).\end{align}
    Using the short exact sequence $(\ref{eq: fund eq of p-adic HT})$ and the $5$-lemma, this follows from the primitive comparison Theorem \cite[Thm. 5.1]{scholze2013padicHodge}.
\end{proof}

\subsection{The case of abeloid varieties and curves}
Let $(K,K^+)$ be an algebraically closed non-archimedean field over $\Q_p$. Let $X$ be either an abeloid variety or a smooth connected proper curve of genus $g\geq 1$ over $K$, and let $x\in X(K)$. In this section, we reformulate the Hodge--Tate decomposition of $X$ with multiplicative coefficients in terms of continuous group cohomology of the étale fundamental group $\pi_1(X,x)$. 

\begin{definition}
    The universal pro-finite-étale cover $\widetilde{X}$ of $X$ is defined to be the diamond
    \[ \widetilde{X} = \varprojlim_{(X',x')} X',\]
    where the limit ranges over all connected finite étale covers $X' \rightarrow X$ together with a lift $x'\in X'(K)$ of $x$.
\end{definition}

The map $\widetilde{X} \rightarrow X$ is a pro-étale torsor under the profinite group $\pi_1(X,x)$. Under our assumptions on $X$, it is moreover known that $\widetilde{X}$ is representable by a qcqs perfectoid space, see \cite[Cor. 5.7-8]{HeuerWearBlakestad2018perfcover}.

\begin{lemma}\label{Lemma: cohomology of universal cover of abeloids and curves}
    For $\Fs$ any of the $v$-sheaves $\Z/n\Z, \Os^{+a}/p^m, \Os^{+a},\Os$, we have
    \begin{align*}
        H_v^i(\widetilde{X},\Fs) = 
        \begin{cases}
        \Fs(K) \quad &\text{if }i=0\\
        0\quad &\text{if }i>0.
    \end{cases} 
    \end{align*}
Here, the sheaf $\Os^{+a}$ is the almost integral structure sheaf, sending a perfectoid pair $(R,R^+)$ to $R^{+,a}$, where we use almost mathematics relative to $(\Os_K,\ma)$.
\end{lemma}
\begin{remark}
    If $X$ is an arbitrary connected, proper, seminormal rigid space, the lemma holds for $i=0,1$, see \cite[Prop. 4.9]{heuer2021line}. However, for $X=\Pro^n$, we have $\widetilde{X} = X$ and we immediately see that the statement need not hold for $i \geq 2$.
\end{remark}
\begin{proof}
    In the abeloid case, this follows from \cite[Prop. 4.2]{heuer2021lineonpefd} and its proof. For $X$ a curve, we proceed as follows. First, Proposition \ref{Prop: main known results about approximating sheaves}(3) gives us
    \[ H_v^i(\widetilde{X},\Z/n\Z) = H_{\et}^i(\widetilde{X},\Z/n\Z).\]
    By \cite[Prop. 14.9]{scholze2022etale}, we have
    \[ H_{\et}^i(\widetilde{X},\Z/n\Z) = \varinjlim_{(X',x')} H_{\et}^i(X',\Z/n\Z).\]
    The claim is thus clear for $i=0$. By Poincaré duality \cite[Thm. 1.1.1]{mann2022padic}, $H^i(X',\Z/n\Z) =0$ for $i>2$ and $H^2(X',\Z/n\Z) \cong H^0(X',\Z/n\Z) = \Z/n\Z$. Moreover, via this identification, the transition functions $f\colon X'' \rightarrow X'$ induce $[\deg(f)]$ on $\Z/n\Z$. It follows that the colimit above vanishes. Finally, for $i=1$, each class in $H^1(X',\Z/n\Z)$ represents an étale $\Z/n\Z$-torsor, and thus is trivialised by a finite étale cover $X' \rightarrow X$ (alternatively, one could invoke \cite[Prop. 3.9]{heuer2022geometric}). Thus the colimit above vanishes as well and the statement is proven for $\Fs=\Z/n\Z$. By the Primitive Comparison Theorem \cite[Thm. 5.1]{scholze2013padicHodge}, we also get 
    \[ H_v^i(\widetilde{X},\Os^+/p^m) \overset{a}{=} \varinjlim H_{\et}^i(X',\Os^+/p^m) \overset{a}{=} \varinjlim H_{\et}^i(X',\Z/p^m)\otimes K^+/p^m = \begin{cases}
         K^+/p^m \quad &\text{if }i=0\\
        0\quad &\text{if }i>0.
    \end{cases}\]
    Since $\Sh(X_v)$ is replete, it follows that
    \[ \RGamma_v(\widetilde{X},\Os^+) = \RGamma_v(\widetilde{X},\Rlim_m \Os^+/p^m) \overset{a}{=} \Rlim_m \RGamma_v(\widetilde{X},\Os^+/p^m) \overset{a}{=} K^+. \]
    Hence the statement follows for $\Fs= \Os^{+a},\Os$.
\end{proof}

We recall the definition of the topological torsion subgroup of $\G_m$ \cite[Def. 2.5]{heuer2022geometric}
\[ \G_m^{\ttt} = \Hom(\underline{\widehat{\Z}},\G_m) = \bigcup_{n\geq 1} \mu_n\cdot \widehat{\G}_m.\]

\begin{thm}\label{thm: case of abeloids and curves}
    Let $X$ be either an abeloid variety or a smooth connected proper curve of genus $g\geq 1$ over $K=\cj{K}$, and fix a point $x\in X(K)$. Then a choice of exponential for $K$ and of a flat lift $\X$ of $X$ to $\BdR+/\xi^2$ induces a decomposition, natural in $\X$, for any $n\geq 0$
    \begin{align}
        H_{\cts}^n(\pi_1(X,x),K^{\times}) = H_{\et}^n(X,\G_m^{\ttt}) \oplus \bigoplus_{i=0}^{n-1} H^i(X,\widetilde{\Omega}_X^{n-i}).
    \end{align}
\end{thm}
\begin{proof}
    Let $\mu = \bigcup_{n\geq 1}\mu_n \sub \G_m$ denote the sheaf of roots of unity on $X_v$. The logarithm sequence extends to a short exact sequence on $X_v$ \cite[Lemma 6.7]{heuer2023padic}
\[\begin{tikzcd}
	0 & \mu & {\G_m^{\ttt}} & {\G_a} & {0.}
	\arrow[from=1-1, to=1-2]
	\arrow[from=1-2, to=1-3]
	\arrow["\log", from=1-3, to=1-4]
	\arrow[from=1-4, to=1-5]
\end{tikzcd}\]
    By Lemma \ref{Lemma: cohomology of universal cover of abeloids and curves}, we deduce that
    \[ H_v^n(\widetilde{X},\G_m^{\ttt})=0 \, , \quad \fa n>0.\]
    Hence the \v{C}ech-to-derived spectral sequence yields an isomorphism
    \[ H_v^n(X,\G_m^{\ttt}) = \check{H}^n(\{\widetilde{X}/X\},\G_m^{\ttt}) = H_{\cts}^n(\pi_1(X,x),\G_m^{\ttt}(\widetilde{X})),\]
    where $\G_m^{\ttt}(\widetilde{X})$ comes equipped with a natural topology \cite[Lemma 2.11]{heuer2023padic} and where we use \cite[Lemma 5.4]{heuer2023padic} for the second equality. Since $\Os(\widetilde{X}) = K$, it follows that $\G_m^{\ttt}(\widetilde{X}) = \G_m^{\ttt}(K)$, in particular it is a trivial $\pi_1(X,x)$-module. We now consider the short exact sequence of topological groups
\[\begin{tikzcd}
	0 & {\G_m^{\ttt}(K)} & {\G_m(K)} & {\G_m/\G_m^{\ttt}(K)} & {0.}
	\arrow[from=1-1, to=1-2]
	\arrow[from=1-2, to=1-3]
	\arrow[from=1-3, to=1-4]
	\arrow[from=1-4, to=1-5]
\end{tikzcd}\]
    Since $\G_m/\G_m^{\ttt}(K)$ is discrete, the above induces a long exact sequence in continuous group cohomology. Moreover, $\G_m/\G_m^{\ttt}$ is uniquely divisible, by \cite[Lemma 2.16]{heuer2021line}. It follows that
    \[ H_{\cts}^n(\pi_1(X,x),\G_m^{\ttt}(K)) = H_{\cts}^n(\pi_1(X,x),\G_m(K)). \]
    The result now follows from the Hodge--Tate decomposition (Theorem \ref{Theorem: The classical sseq splits}) applied to the locally $p$-divisible group $\G_m^{\ttt}$.
\end{proof}

\subsection{Application to analytic Brauer groups}
In this section, we prove some additional facts on the diamantine higher direct images $\BBun_{\Gs,\tau}^n$ for a proper smooth rigid space $X$ over $K$, and we deduce some consequences for analytic Brauer groups. Here $(K,K^+)$ is any perfectoid field extension of $\Q_p$. The main result is the following.

\begin{proposition}\label{proposition 5-term exact sequence and abutment}
Let $X$ be a proper smooth rigid space over $K$. Let $\Gs \rightarrow X$ be an admissible locally $p$-divisible group. Then there are étale rigid groups $(\Gamma_n)_{n \geq 0}$ and exact sequences on $\Perf_{K,\tau}$, for $\tau \in \{\et,v\}$
\begin{equation}\label{five term exact sequence for higher direct images}\begin{tikzcd}
	0 & {\Gamma_{n-1}} & {\BBun_{\widehat{\Gs},\tau}^n} & {\BBun_{\Gs,\tau}^n} & {\BBun_{\cj{\Gs},\tau}^n} & {\Gamma_{n}} & {0.}
	\arrow[from=1-1, to=1-2]
	\arrow[from=1-2, to=1-3]
	\arrow[from=1-3, to=1-4]
	\arrow[from=1-4, to=1-5]
	\arrow[from=1-5, to=1-6]
	\arrow[from=1-6, to=1-7]
\end{tikzcd}\end{equation}
\end{proposition}
     
We begin by making preparations.
\begin{lemma}\label{morphisms from approximating sheaves to vector groups}
    Let $\Fs$ be an abelian sheaf on $\Adic_{K,\et}$ with the approximation property and let $H$ be a partially proper rigid group over $K$. Assume that $K$ is algebraically closed. Then any morphism of sheaves $f\colon \Fs \rightarrow H$ factors through the locally constant sheaf 
    \[ \underline{H(K)} \hookrightarrow H.\]
\end{lemma}

\begin{proof}
    Let $Y$ be any good affinoid adic space over $K$. By \cite[Prop. 3.17]{heuer2023diamantine}, we can write 
    \[ Y \approx \varprojlim_i Y_i,\]
    where each $Y_i$ is a smooth affinoid rigid space over $K$. We may apply Proposition \ref{Prop: main known results about approximating sheaves}(2) to obtain
    \[ \Fs(Y) = \varinjlim_i \Fs(Y_i).\]
    It is enough to show that $f$ sends $\Fs(Y_i)$ in the set of locally constant functions $Y_i \rightarrow H(K,K^+)$. Hence, we may assume that $Y$ is a connected smooth affinoid rigid space. Given $s\in \Fs(Y)$, it is enough to show that $f(s) \in H(Y)$ restricts to a constant map on $Y' = Y\times_{\Spa(K,K^+)} \Spa(K,\Os_K)$, using the valuative criterion. Therefore we may also assume that $K^+=\Os_K$.

    Let $s\in \Fs(Y)$ and $y\in Y(K)$. We claim that there is an open neighborhood $y\in U \sub Y$ such that $s\restr{U}$ is constant equal to $s(y) \in \Fs(K)$. We have morphisms of rigid spaces
    \[ \Spa(K,\Os_K) \xrightarrow{y} Y \xrightarrow{q} \Spa(K,\Os_K).\]
    Up to replacing $s$ by $s-q^*s(y)$, we may assume that $s(y)=0$. One easily checks that
    \[ \Spa(K,\Os_K) \approx \underset{y \in U_i \sub Y}{\varprojlim}\, U_i, \]
    where the limit ranges over all affinoid opens $U_i \sub Y$ containing $y$. Therefore, the approximation property applies and we deduce that already $s\restr{U_i} = 0$ for some $i$, as required. From this, we will deduce that $f(s)\colon Y \rightarrow H$ is constant. Indeed, let $y\in Y(K)$ with image $h\in H(K)$. We let $Y_h$ denote its fibre under $f(s)$
\[\begin{tikzcd}
	{Y_h} & {\Spa(K)} \\
	Y & {H.}
	\arrow[from=1-1, to=1-2]
	\arrow[from=1-1, to=2-1]
	\arrow["\lrcorner"{anchor=center, pos=0.125}, draw=none, from=1-1, to=2-2]
	\arrow["h", from=1-2, to=2-2]
	\arrow["{f(s)}", from=2-1, to=2-2]
\end{tikzcd}\]
Then $Y_h$ is a Zariski-closed subset of $Y$ containing $y$, and we claim that it equals all of $Y$. Indeed, by the above, there exists a rational open neighborhood $y\in U \sub Y$ contained in $Y_h$. As $\Os(Y) \rightarrow \Os(U)$ is flat \cite[§4.1, Cor. 5]{Bosch2014rigid} and $\Os(Y)$ is a domain, it follows that the surjection $\Os(Y) \twoheadrightarrow \Os(Y_h)$ is an isomorphism, as required. This concludes the proof.
\end{proof}

\begin{proof}[Proof of Proposition \ref{proposition 5-term exact sequence and abutment}]
We start by taking $\tau$ to be the $v$-topology. By Proposition \ref{Prop: main known results about approximating sheaves}(1), $\cj{\Gs}$ is a $v$-sheaf and we have a short exact sequence on $X_{v}$
\[\begin{tikzcd}
	0 & {\widehat{\Gs}} & \Gs & {\cj{\Gs}} & {0.}
	\arrow[from=1-1, to=1-2]
	\arrow[from=1-2, to=1-3]
	\arrow[from=1-3, to=1-4]
	\arrow[from=1-4, to=1-5]
\end{tikzcd}\]
Applying $R\pi_{v,*}$ yields the following long exact sequence of sheaves on $\Adic_{K,v}$
\[\begin{tikzcd}
	\ldots & {R^n\pi_{v,*}\widehat{\Gs}} & {R^n\pi_{v,*}\Gs} & {R^n\pi_{v,*}\cj{\Gs}} & {R^{n+1}\pi_{v,*}\widehat{\Gs}} & \ldots
	\arrow["{\delta_{n-1}}", from=1-1, to=1-2]
	\arrow[from=1-2, to=1-3]
	\arrow[from=1-3, to=1-4]
	\arrow["{\delta_n}", from=1-4, to=1-5]
	\arrow[from=1-5, to=1-6]
\end{tikzcd}\]
We now let $\Gamma_n$ denote the image in $v$-sheaves on $\Adic_K$ of the map $\delta_n$. To show that $\Gamma_n$ is an étale rigid group, we may pass to a completed algebraic closure and therefore we assume that $K$ is algebraically closed. Observe that $H=R^{n+1}\pi_{v,*}\widehat{\Gs}$ is partially proper. Indeed, it is separated, as is any rigid group over $K$. To see that $H$ is overconvergent, it is enough to show that $H[p^{\infty}] = R^{n+1}\pi_{v,*}\Gs[p^{\infty}]$ is overconvergent, which follows from Proposition \ref{Prop: main properties of diamantine higher direct images}(2). We may thus apply Lemma \ref{morphisms from approximating sheaves to vector groups} to find that the sheaf $\Gamma_n$ is a sub-$v$-sheaf of the locally constant sheaf $\underline{A}_K$ for some group $A$. By \cite[Prop. 10.5]{scholze2022etale} applied to each open and closed component $\underline{\{a\}}_K$, it follows that $\Gamma_n$ is representable by an open subspace of $\underline{A}_K$ and is therefore étale over $K$. This concludes, with $\tau$ being the $v$-topology.
     
We now treat the case $\tau=\et$. We define sheaves $Q$ and $Q'$ on $\Perf_{K,\et}$ as the following cokernels
\[ Q = \Coker(\BBun_{\widehat{\Gs},\et}^n \rightarrow \BBun_{\widehat{\Gs},v}^n)\, , \quad Q' = \Coker(\BBun_{\Gs,\et}^n \rightarrow \BBun_{\Gs,v}^n).\]
From the degeneration of the geometric spectral sequences (Theorem \ref{thm: geometric Hodge--Tate spectral sequence for general G degenerates}), we obtain a commutative diagram of étale sheaves
\[\begin{tikzcd}
	0 & {\BBun_{\widehat{\Gs},\et}^n} & {\BBun_{\widehat{\Gs},v}^n} & Q & 0 \\
	0 & {\BBun_{\Gs,\et}^n} & {\BBun_{\Gs,v}^n} & {Q'} & {0.}
	\arrow[from=1-1, to=1-2]
	\arrow[from=1-2, to=1-3]
	\arrow[from=1-2, to=2-2]
	\arrow[from=1-3, to=1-4]
	\arrow[from=1-3, to=2-3]
	\arrow[from=1-4, to=1-5]
	\arrow["\cong", dashed, from=1-4, to=2-4]
	\arrow[from=2-1, to=2-2]
	\arrow[from=2-2, to=2-3]
	\arrow[from=2-3, to=2-4]
	\arrow[from=2-4, to=2-5]
\end{tikzcd}\]
where the rows are exact and the right vertical map is an isomorphism. From the snake Lemma, we deduce that
    \[ \Ker(\BBun_{\widehat{\Gs},\et}^n \rightarrow \BBun_{\Gs,\et}^n) = \Ker(\BBun_{\widehat{\Gs},v}^n \rightarrow \BBun_{\Gs,v}^n) = \Gamma_{n-1}.\]
This concludes the proof.
\end{proof}

\begin{example}\label{Example: brauer groups shouldn't be torsion}
    Assume that $K$ is algebraically closed and let $X^{\alg} \xrightarrow{\pi^{\alg}}\Spec(K)$ be a proper smooth algebraic variety with analytification \[(X\xrightarrow{\pi} \Spa(K)) = (X^{\alg} \xrightarrow{\pi^{\alg}}\Spec(K))^{\an}.\]
    The exponential of Lemma \ref{lemma: exponential converges} yields an open subgroup $U$ of the locally $p$-divisible rigid group $\BBun_{\widehat{\G}_m,\et}^2$ isomorphic to the polydisc $\B^d$, where $d= \dim_K H_{\et}^2(X,\Os_X)$, and we assume that $d>0$. Then the image of $U$ in $\BBun_{\G_m,\et}^2$ is not torsion, since $U$ would otherwise be contained in the étale subgroup $\bigcup_{n\geq 1} [n]^{-1}(\Gamma_1)$. This gives a new perspective on why the Brauer group $H_{\et}^2(X,\G_m)$ of smooth proper rigid analytic varieties typically contains non-torsion elements, in contrast with the algebraic Brauer group $H_{\et}^2(X^{\alg},\G_m)$ which is torsion \cite[Cor. IV.2.6]{milne1980etale}. This also shows that the GAGA principle fails for $\G_m$-gerbes. 
\end{example}

\begin{example}\label{delta_1 doesn't vanish}
We make the sequence (\ref{five term exact sequence for higher direct images}) in low degree more explicit when $X = A$ is an abeloid variety over $K$ with good reduction and $\Gs=\G_m$. We assume that $K^+=\Os_K$ and that $K$ is algebraically closed. By \cite[Lemma 4.11.2]{heuer2023diamantine}, we have a surjection $\pi_*\G_m \twoheadrightarrow \pi_* \cj{\G}_m$ and thus $\Gamma_0 =0$. We thus have an exact sequence of sheaves on $\Perf_{K,\et}$
\[\begin{tikzcd}
	0 & {R^1\pi_{\Et,*}\widehat{\G}_m} & {R^1\pi_{\Et,*}\G_m} & {R^1\pi_{\Et,*}\cj{\G}_m} & {\Gamma_1} & {0.}
	\arrow[from=1-1, to=1-2]
	\arrow[from=1-2, to=1-3]
	\arrow[from=1-3, to=1-4]
	\arrow[from=1-4, to=1-5]
	\arrow[from=1-5, to=1-6]
\end{tikzcd}\]
Here, the relative Picard functor $\PPic_{A,\et} = R^1\pi_{\Et,*}\G_m$ is representable by a rigid group whose identity component is the dual abeloid variety $A^{\vee}$ \cite[§6]{Bosch1991DegeneratingAV}\cite[Cor. 2.9.(3), Cor. 5.4]{heuer2023diamantine}. Applying the functor $\widehat{(\cdot)}$ (which is left exact on rigid groups by Proposition \ref{statements about p-topological torsion subgroup}(1)) yields an identification
\[ R^1\pi_{\Et,*}\widehat{\G}_m = \widehat{\PPic}_{A,\et} = \widehat{A}^{\vee}.  \]
\begin{claim}
    We have
    \[ R^1\pi_{\Et,*}\cj{\G}_m = \PPic_{A_0,\et}^{\diamondsuit}[\tfrac{1}{p}], \]
where $\PPic_{A_0,\et}$ is the Picard variety of the special fibre $A_0$. See \cite[Def. 5.1]{heuer2021lineonpefd} for the meaning of $(\cdot)^{\diamondsuit}$ in this context.
\end{claim}
\begin{proof}
First, \cite[Cor. 5.4]{heuer2021lineonpefd} gives us
\[ \PPic_{A_0,\et}^{\diamondsuit}[\tfrac{1}{p}] = \PPic_{\widetilde{A},\et},\]
where 
\[ \widetilde{A} = \varprojlim_{[p]} A\] 
denotes the $p$-adic universal cover of $A$. Let $\widetilde{\pi}\colon \widetilde{A} \rightarrow \Spa(K)$ denote the structure map. Combining \cite[Lemma 5.5]{heuer2021lineonpefd} and the exponential sequence \cite[Lemma 2.12]{heuer2021lineonpefd}, we obtain
\[ \PPic_{\widetilde{A},\et} = R^1\widetilde{\pi}_{\Et,*}\cj{\G}_m. \]
Then by \cite[Prop. 4.12.2]{heuer2023diamantine}, it follows that
\[ R^1\widetilde{\pi}_{\Et,*}\cj{\G}_m = \varinjlim_{[p]^*}R^1\pi_{\Et,*}\cj{\G}_m.\]
It remains to show that this colimit is equal to $R^1\pi_{\Et,*}\cj{\G}_m[\frac{1}{p}] =R^1\pi_{\Et,*}\cj{\G}_m$. Consider the exact sequence
\[\begin{tikzcd}
	0 & {\widehat{\PPic}_{A,\et}} & {\PPic_{A,\et}} & {R^1\pi_{\Et,*}\cj{\G}_m} & {R^2\pi_{\Et,*}\widehat{\G}_m.}
	\arrow[from=1-1, to=1-2]
	\arrow[from=1-2, to=1-3]
	\arrow[from=1-3, to=1-4]
	\arrow[from=1-4, to=1-5]
\end{tikzcd}\]
It is enough to show that, for $\Fs$ any sheaf on the left or right of $R^1\pi_{\Et,*}\cj{\G}_m$ in the sequence above, we have
\[ \varinjlim_{[p]^*}\Fs= \Fs[\tfrac{1}{p}].\] 
For $\Fs = \widehat{\PPic}_{A,\et}$, this clear since $[p]^* = (\cdot)^{\otimes p}$ on $A^{\vee}$ . For $\Fs = \PPic_{A,\et}$, it follows from the fact that for a line bundle $\Ls$ on $A$, we have $[p]^*\Ls \equiv \Ls^{\otimes p^{2}}(\modulo A^{\vee})$: This is a consequence of the theorem of the cube \cite[Thm. 7.1.6.(a)]{lutkebohmert2016}, see \cite[II.8.(iv)]{mumford1974abelian}. For $\Fs= R^2\pi_{\Et,*}\widehat{\G}_m$, it follows from the exact sequence $(\ref{geometric cohomological log exact sequence})$. Hence the claim is proven.
\end{proof}
We now use the exact sequence \cite[Corollary 5.4]{heuer2021lineonpefd}
\[\begin{tikzcd}
	0 & {\widehat{\PPic}_{A,\et}} & {\PPic_{A,\et}} & {\PPic_{A_0,\et}^{\diamondsuit}[\tfrac{1}{p}]} & {\underline{\frac{\NS(A_0)[\frac{1}{p}]}{\NS(A)}}} & {0.}
	\arrow[from=1-1, to=1-2]
	\arrow[from=1-2, to=1-3]
	\arrow[from=1-3, to=1-4]
	\arrow[from=1-4, to=1-5]
	\arrow[from=1-5, to=1-6]
\end{tikzcd}\]
Here, $\NS(A)$ denotes the Neron--Severi group of $A$. We conclude that 
\[ \Gamma_1 = \underline{\frac{\NS(A_0)[\frac{1}{p}]}{\NS(A)}}.\]
In particular, $\Gamma_1$ is non-zero in that case.
\end{example}

\begin{cor}\label{cor: non-representability of analytic Brauer functor}
    There exists a proper smooth rigid space $X$ over $K$ such that $\BBun_{\G_m,\et}^2$ is not representable by a rigid group.
\end{cor}
\begin{proof}
    We take $K$ to be algebraically closed and $X = E\times E$, for $E$ an elliptic curve without complex multiplication, and with supersingular reduction $E_0$. Note that 
    \[ \rank_{\Z} \NS(E\times E) =2+ \rank_{\Z} \End(E) = 3.\]
    Similarly, one computes $\rank_{\Z} \NS(E_0 \times E_0) =6$. Hence, by Example \ref{delta_1 doesn't vanish},
    \[ \Gamma_1 = \frac{\Z[\tfrac{1}{p}]^{\oplus 6}}{\Z^{\oplus 3}}= \Z[\tfrac{1}{p}]^{\oplus 3} \oplus (\Q_p/\Z_p)^{\oplus 3}. \]
    We now apply the left exact functor $\underline{\Hom}(\Z_p,-)$ to the exact sequence
\[\begin{tikzcd}
	0 & {\Gamma_1} & {\BBun_{\widehat{\G}_m,\et}^2} & {\BBun_{\G_m,\et}^2.}
	\arrow[from=1-1, to=1-2]
	\arrow[from=1-2, to=1-3]
	\arrow[from=1-3, to=1-4]
\end{tikzcd}\]
Observe that
\[ \underline{\Hom}(\Z_p,\Gamma_1) = (\Q_p/\Z_p)^{\oplus 3}.\]
Hence we have a commutative diagram with exact rows
\[\begin{tikzcd}
	0 & {(\Q_p/\Z_p)^{\oplus 3}} & {\underline{\Hom}(\Z_p,\BBun_{\widehat{\G}_m,\et}^2)} & {\underline{\Hom}(\Z_p,\BBun_{\G_m,\et}^2)} \\
	0 & {\Gamma_1} & {\BBun_{\widehat{\G}_m,\et}^2} & {\BBun_{\G_m,\et}^2}
	\arrow[from=1-1, to=1-2]
	\arrow[from=1-2, to=1-3]
	\arrow["\neq"', hook, from=1-2, to=2-2]
	\arrow[from=1-3, to=1-4]
	\arrow["{\ev_1}"', from=1-3, to=2-3]
	\arrow["\cong", from=1-3, to=2-3]
	\arrow["{\ev_1}", from=1-4, to=2-4]
	\arrow[from=2-1, to=2-2]
	\arrow[from=2-2, to=2-3]
	\arrow[from=2-3, to=2-4]
\end{tikzcd}\]
It follows that the right vertical map cannot be injective. By Lemma \ref{statements about p-topological torsion subgroup}(1), in that case, $\BBun_{\G_m,\et}^2$ cannot be represented by a rigid group.
\end{proof}

\begin{remark}\label{remark: algebraic higher pushforward are not representable}
    Let $X^{\alg}$ be  an abelian variety of dimension $g>1$, Then the algebraic higher direct image
    \[ R^2\pi_{\Et,*}^{\alg}\G_m \colon \Sch_{K,\et}\rightarrow \Ab\]
    is not representable by an algebraic group. By contradiction, let $H$ be such a group, which is automatically smooth, as $K$ has characteristic $0$. Let $X^{\alg}[\epsilon]$ denote the trivial thickening of $X^{\alg}$ along $\Spec(K) \hookrightarrow \Spec(K[\epsilon])$. There is a short exact sequence on $X_{\et}^{\alg}$
\[\begin{tikzcd}
	0 & {\Os_{X^{\alg}}} & {\Os_{X^{\alg}[\epsilon]}^{\times}} & {\Os_{X^{\alg}}^{\times}} & {1.}
	\arrow[from=1-1, to=1-2]
	\arrow[from=1-2, to=1-3]
	\arrow[from=1-3, to=1-4]
	\arrow[from=1-4, to=1-5]
\end{tikzcd}\]
We deduce that
\[ \Lie(H) = \Ker(H(K[\epsilon]) \rightarrow H(K)) = \Ker(H_{\et}^2(X^{\alg}[\epsilon],\G_m) \rightarrow H_{\et}^2(X^{\alg},\G_m)) = H^2(X^{\alg},\Os_{X^{\alg}}).\]
Hence $\dim H = h^2(X^{\alg},\Os_{X^{\alg}})>0$, which contradicts the fact that $H(K)=H_{\et}^2(X^{\alg},\G_m)$ is torsion.
\end{remark}

\begin{remark}
    Arguing as in the proof of \cite[Cor. 2.9.5]{heuer2023diamantine}, we can use the local splittings of Remark \ref{remark: local splittings} to show that, for given $n \geq 0$, the étale higher direct image $\BBun_{\Gs,\et}^n$ is representable by a rigid group if and only if $\BBun_{\Gs,v}^n$ is representable by a rigid group.
\end{remark}

 \medskip
 
\printbibliography[
title={References}
]
\end{document}

\typeout{get arXiv to do 4 passes: Label(s) may have changed. Rerun}